\documentclass{article}
\usepackage{amsmath, geometry,amssymb,amscd, float, verbatim, latexsym,subfig,wrapfig}
\usepackage{algorithmic,algorithm,cite,setspace}
\usepackage[pdftex]{graphicx}
\usepackage[all]{xy}
\usepackage{lscape}
\usepackage{url}
\usepackage{amsthm}

\newtheorem{theorem}{Theorem}[section]
\newtheorem{lemma}[theorem]{Lemma}
\newtheorem{proposition}[theorem]{Proposition}
\newtheorem{corollary}[theorem]{Corollary}

\newtheorem{conjecture}[theorem]{Conjecture}
\newtheorem{remark}[theorem]{Remark}
\newtheorem{definition}[theorem]{Definition}

\numberwithin{equation}{section}

\newcommand{\bi}{\text{\textup{bi}}}
\newcommand{\sym}{\text{\textup{sym}}}
\renewcommand{\c}{a}
\newcommand{\e}{\mathrm{e}}
\newcommand{\tc}{\tilde{c}}

\newcommand{\R}{\mathbb{R}}
\newcommand{\C}{\mathbb{C}}

\newcommand{\N}{\mathbb{N}}
\newcommand{\Z}{\mathbb{Z}}
\newcommand{\pp}{\tfrac{\pi}{2}}
\newcommand{\II}{I_3}

\newcommand{\B}{\widehat{B}}
\newcommand{\tF}{\widetilde{F}}
\newcommand{\epseps}{\epsilon_*}
\newcommand{\tS}{\widetilde{S}}

\newcommand{\da}{\Delta_\alpha}
\newcommand{\dw}{\Delta_\omega}

\newcommand{\dc}{\delta_c}

\newcommand{\bx}{\bar{x}}
\newcommand{\balpha}{\bar{\alpha}}
\newcommand{\bomega}{\bar{\omega}}
\newcommand{\bc}{\bar{c}}
\newcommand{\bce}{\bc_{\epsilon}}
\newcommand{\rr}{\check{r}}

\newcommand{\ZZ}{\mathcal{Z}}
\newcommand{\QQ}{\mathcal{Q}}

\newcommand{\LL}{\zeta}
\newcommand{\upperbound}[1]{\overline{#1}}

\newcommand{\cO}{\mathcal{O}}

\title{A proof of Wright's conjecture}
\author{Jan Bouwe van den Berg\thanks{Partially supported by NWO VICI-grant 639.033.109} \thanks{Department of Mathematics, VU Amsterdam, de Boelelaan 1081, 1081 HV Amsterdam, The Netherlands, janbouwe@few.vu.nl}
\and
Jonathan Jaquette \thanks{Partially supported by NSF DMS 0915019,	NSF DMS 1248071} \thanks{Department of Mathematics, Rutgers, The State University Of New Jersey, 110 Frelinghuysen Rd., Piscataway, NJ 08854-8019, USA, jaquette@math.rutgers.edu}
}

\begin{document}

\maketitle

\begin{abstract}
Wright's conjecture states that the origin is the global attractor for the delay differential equation $y'(t) = - \alpha y(t-1) [ 1 + y(t) ] $  for all $\alpha \in (0,\pp]$. This has been proven to be true for a subset of parameter values $\alpha$. We extend the result to the full parameter range $\alpha \in (0,\pp]$, and thus prove Wright's conjecture to be true. Our approach relies on a careful investigation of the neighborhood of the Hopf bifurcation occurring at $\alpha = \pp$. This analysis fills the gap left by complementary work on Wright's conjecture, which covers parameter values further away from the bifurcation point. Furthermore, we show that the branch of (slowly oscillating) periodic orbits originating from this Hopf bifurcation does not have any subsequent bifurcations (and in particular no folds) for
$\alpha\in(\pp , \pp + 6.830 \times10^{-3}]$. 
When combined with other results, this proves that the branch of slowly oscillating solutions that originates from the Hopf bifurcation at $\alpha=\pp$ is globally parametrized by $\alpha > \pp$.
\end{abstract}

\begin{center}
	{\bf \small Keywords.} 
	{ \small Delay Differential Equation,
		Hopf Bifurcation, 
		Wright's Conjecture, \\
		Supercritical Bifurcation Branch,
		Newton-Kantorovich Theorem
		}
\end{center}

\section{Introduction}
\label{s:introduction}

In many biological and physical systems the dependency of future states relies not only on the present situation, but on a broader history of the system. 
For simplicity, mathematical models often ignore the causal influence of all but the present state.  
However, in a wide variety of applications delayed feedback loops play an inextricable role in the qualitative dynamics of a system \cite{kolmanovskii2013introduction}. 
These phenomena can be modeled using delay and integro-differential equations, the theory of which has developed significantly over the past 60 years \cite{Hale2006}. 
A canonical and well-studied example of a nonlinear delay differential equation is Wright's equation:
\begin{equation}
y'(t) = - \alpha \,y(t-1) \left[  1+ y(t)  \right] .
\label{eq:Wright}
\end{equation}
Here $\alpha$ is considered to be both real and positive. 
This equation has been a central example considered in the development of much of the theory of functional differential equations. 
For a short overview of this equation, we refer the reader to \cite{hale1971functional}. We cite some basic properties of its global dynamics \cite{wright1955non}:  
\begin{itemize}
	\item Corresponding to every 
	 $y \in C^0([-1,0])$, there is a unique solution of 
 \eqref{eq:Wright}  for all $t>0$.
	\item Wright's equation has two equilibria $ y \equiv -1$ and $ y \equiv 0$. Moreover, 
	solutions cannot cross $-1$. Any solution with $y(t_0)=-1$ (for some $t_0 \in \mathbb{R}$) is identically equal to $-1$.
	\item When $y<-1$ then the solution decreases monotonically without bound.
	 \item When $ y > -1$ then $ y(t) $ is globally bounded as $ t \to + \infty$. 
\end{itemize}

Henceforth we restrict our attention to $y>-1$.
In Wright’s seminal 1955 paper \cite{wright1955non}, he showed that if $ \alpha \leq \tfrac{3}{2}$ then any solution having $ y > -1$ is attracted to $ 0$ as $ t \to + \infty$.  
At $ \alpha = \pp$, the equilibrium $ y \equiv 0$ changes from asymptotically stable to unstable, and Wright formulated the following  conjecture: 

\begin{conjecture}[Wright's Conjecture]
	For every $0< \alpha \leq \pp  $, the zero solution to~\eqref{eq:Wright} is globally attractive. 
	\label{conj:ConjWright}
\end{conjecture}
For $ \alpha > \pp$, Wright proved the existence of oscillatory solutions to \eqref{eq:Wright} which do not tend towards $0$, and whose zeros are spaced at distances greater than the delay. 
Such a periodic solution is said to be \emph{slowly oscillating}, and formally defined as follows: 
\begin{definition}
	A \emph{slowly oscillating periodic solution (SOPS)}  is a periodic solution $y(t)$ which up to a time translation satisfies the following property: there exists some $t_{-}, t_{+} >1$ and $ L = t_{-} +t_{+} $ such that 
	$ y( t ) >0$ for $ t\in (0,t_{+})$,
	$y(t) < 0$ for $ t \in (-t_{-},0)$, 
	and $y(t+L) = y(t)$ for all $ t$,
	so that $ L $ is the minimal period of $y(t)$.  
\end{definition}

In Jones' 1962 paper \cite{jones1962existence} he proved that for $ \alpha > \pp $  there exists a  slowly oscillating periodic solution to~\eqref{eq:Wright}.  
Based on numerical calculations \cite{jones1962nonlinear}  Jones made the following conjecture:

\begin{conjecture}[Jones' Conjecture]
	For every $ \alpha > \pp  $ there exists a unique slowly oscillating periodic solution to~\eqref{eq:Wright}. 
	\label{conj:ConjJones}
\end{conjecture}

Slowly oscillating periodic solutions play a critical role in the global dynamics of~\eqref{eq:Wright}. 
The global attractor of~\eqref{eq:Wright} admits a Morse decomposition \cite{mccord1996global,mallet1988morse} and if there is a unique SOPS and $ \alpha > \pp$ then it must be asymptotically stable \cite{xie1991thesis,xie1993uniqueness}.  
In \cite{neumaier2014global} it is shown that there are no homoclinic solutions from $0$ to itself for $0 \leq   \alpha \leq \pp$. A corollary is the following theorem.
\begin{theorem}[Theorem 3.1 in \cite{neumaier2014global}]
	\label{thm:AttractiveNonexistenceEquivalence}
	The zero solution of \eqref{eq:Wright} is globally attracting if and only if \eqref{eq:Wright}
	has no slowly oscillating periodic solution. 
\end{theorem}

Despite a considerable amount of work studying Wright's equation, complete resolution of these conjectures has remained elusive (see the survey paper \cite{walther2014topics} and the references contained therein). 
To describe a few  results,
the global bifurcation analysis in \cite{nussbaum1975global} proved that for $ \alpha > \pp$ there is a continuum of pairs $(\phi,\alpha)$ where $ \phi$ is a periodic solution to~\eqref{eq:Wright}, and these pairs form a $2-$dimensional manifold~\cite{regala1989periodic}.
In \cite{chow1977integral} it was shown that Wright's equation has a supercritical Hopf bifurcation at $ \alpha = \pp$. 
By studying the Floquet multipliers of periodic solutions for large $ \alpha$, Xie showed in  \cite{xie1991thesis} that Conjecture 1.2 holds for $ \alpha \geq 5.67$.

Recent results using computer assisted proofs have narrowed the gap to resolving both conjectures. 
In the work in preparation \cite{jlm2016Floquet} it is shown that Conjecture \ref{conj:ConjJones} holds for $ [1.94,6.00]$. 
In \cite{lessard2010recent}, it is shown that the branch of periodic orbits emanating from the Hopf bifurcation does not have any subsequent bifurcations in the interval $ \alpha \in  [ \pp + \delta_1 , 2.3]$ where $ \delta_1 = 7.3165 \times 10^{-4}$.  
In \cite{neumaier2014global} it was shown that Conjecture \ref{conj:ConjWright} holds for $ \alpha \in [1.5, \pp - \delta_2 ]$ where $ \delta_2 = 1.9633 \times 10^{-4}$, and the authors remark that ``\emph{substantial improvement of the theoretical
	part of the present proof is needed to prove Wright's conjecture fully.}''

Many normal form techniques for functional differential equations have been developed  to transform a given equation into a simpler expression having the same qualitative behavior as the original equation (see \cite{faria2006normal} and references contained therein). 
While this transformation is valid in some neighborhood about the bifurcation point, such results usually do not describe the size of this neighborhood explicitly. 
In this paper we develop an explicit description of a neighborhood wherein the only periodic solutions are those originating from the Hopf bifurcation. 
The main result of this analysis is the resolution of Wright's conjecture. 

\begin{theorem}
	For every $0 < \alpha \leq \pp  $, the zero solution to~\eqref{eq:Wright} is globally attractive.
	\label{thm:IntroWrightConjecture} 
\end{theorem}

This result follows from Theorem~\ref{thm:WrightConjecture} combined with Theorem~\ref{thm:AttractiveNonexistenceEquivalence}. 
Roughly, by the work in \cite{neumaier2014global}, 
to prove Wright's conjecture
it is sufficient  to show that there do not exist any slowly oscillating periodic solutions for $ \alpha \in [ \pp - \delta_2, \pp]$, where $\delta_2 = 1.9633 \times 10^{-4}$.
Indeed, we construct an explicit neighborhood about $ \alpha = \pp$ for which the bifurcation branch of periodic orbits are the only periodic orbits. 
Then we show that throughout this entire neighborhood the solution branch behaves as expected from a supercritical bifurcation branch, i.e., it does not bend back into the parameter region $\alpha \leq \pp$.

Rather than trying to resolve all small bounded solutions near the bifurcation point through a center manifold analysis, we focus on periodic orbits only. In particular, we ignore orbits that connect the trivial state to the periodic states, since those are not relevant for our analysis.
The advantage is that, by restricting our attention to periodic solution, we can perform our analysis in Fourier space. We first note that all periodic solutions are smooth, as was established in~\cite{wright1955non} and more generally in~\cite{nussbaum-analytic}.
\begin{lemma}[\cite{nussbaum-analytic}]\label{l:analytic}
	All periodic solutions of~\eqref{eq:Wright} are real analytic.
\end{lemma}

For  a periodic function $y:\R \to \R$ with frequency $\omega >0$ we write 
\begin{equation}
y(t)  = \sum_{k \in \Z} \c_{k} e^{ i \omega k t } ,
\label{eq:FourierEquation}
\end{equation}
where $\c_k \in \C$. 
This transforms the delay equation~\eqref{eq:Wright} into 
\begin{equation}
( i \omega k + \alpha e^{ - i \omega k}) \c_k + \alpha \sum_{k_1 + k_2 = k} e^{- i \omega k_1} \c_{k_1} \c_{k_2} = 0 \qquad\text{for all } k \in \Z.
\label{eq:FourierSequenceEquation}
\end{equation}
In effect, the problem of finding periodic solutions to Wright's equation can be reformulated as finding a parameter $ \alpha$, a frequency $\omega$, and a sequence $ \{a_k\}$  for which~\eqref{eq:FourierSequenceEquation} is satisfied. 
In Section \ref{s:preliminaries} we define an appropriate sequence space to work in, and define a zero finding problem $ F_\epsilon( \alpha,\omega,c)=0$  equivalent to~\eqref{eq:FourierSequenceEquation}.  
The auxiliary variable $\epsilon$, which represents the dominant Fourier mode, corresponds to the rescaling $y \mapsto \epsilon y$ canonical to the study of Hopf bifurcations.

In Section \ref{s:local} we construct a Newton-like operator $T_\epsilon$ whose fixed points correspond to the zeros of $F_\epsilon(\alpha, \omega, c)$.  
By applying a Newton-Kantorovich like theorem, we identify explicit neighborhoods $B_\epsilon$ wherein $T_\epsilon: B_\epsilon \to B_\epsilon$ is a uniform contraction mapping. 
By the nature of our argument, we have the freedom to construct both large and small balls $ B_\epsilon$ on which we may apply the Banach fixed point theorem. Using smaller balls will produce tighter  approximations of the periodic solutions, while using larger balls will produce a larger region within which the periodic solution is unique.

These results are leveraged in  Section \ref{s:global} to derive global results such  as Theorem \ref{thm:IntroWrightConjecture}, as well as Theorem \ref{thm:IntroNoFold} which helps to resolve one part of the reformulated Jones conjecture presented in~\cite{lessard2010recent}. 
This result shows that the branch of solutions that bifurcates from the Hopf bifurcation at $\alpha = \pp$ provides a unique SOPS for every $\alpha > \pp$. 
\begin{theorem}
\label{thm:IntroNoFold}
There are no bifurcations in the branch of SOPS originating from the Hopf bifurcation for $\alpha > \pp  $. 
\end{theorem}

This result follows from Theorem~\ref{thm:NoFold} combined with the results in~\cite{lessard2010recent,jlm2016Floquet,xie1991thesis}, see Corollary~\ref{cor:collectreformulatedJones}. 
Roughly, by the work in~\cite{lessard2010recent,jlm2016Floquet,xie1991thesis},
to prove Theorem~\ref{thm:IntroNoFold} 
it suffices to show that there are no subsequent bifurcations for $ \alpha \in ( \pp , \pp + \delta_1)$, where $\delta_1 = 7.3165 \times 10^{-4}$. 
We prove in Proposition~\ref{prop:TightEstimate} that for  $ 0 < \epsilon \leq 0.1$  
 there is a locally unique $( \hat{\alpha}_\epsilon , \hat{\omega}_\epsilon, \hat{c}_\epsilon)$ which solves $F_\epsilon ( \hat{\alpha}_\epsilon , \hat{\omega}_\epsilon, \hat{c}_\epsilon)=0$. 
However, this is not sufficient. 
To show that the branch of periodic solutions does not have any subsequent bifurcations, we prove that $\hat{\alpha}_\epsilon $ is monotonically increasing in $ \epsilon$. 
Since  
$\frac{d}{d\epsilon} \hat{\alpha}_\epsilon \approx \tfrac{2 \epsilon}{5} ( \tfrac{3 \pi}{2} -1 ) $, 
in order to have any hope of proving $\frac{d}{d\epsilon} \hat{\alpha}_\epsilon >0$, it is imperative that we derive an $\cO(\epsilon^2)$ approximation of $ \frac{d}{d\epsilon} \hat{\alpha}_\epsilon$, an approach we take from the beginning of our analysis.

Theorem~\ref{thm:IntroNoFold} does not fully resolve Conjecture~\ref{conj:ConjJones}, as it makes no claims about the (non)existence of isolas of solutions (disjoint from the Hopf bifurcation branch).
Hence, as a corollary to Theorem \ref{thm:IntroNoFold}, we are able to reduce the Jones' conjecture~\ref{conj:ConjJones} to the following statement:
\begin{conjecture}
	The only slowly oscillating periodic solution to Wright's equation are those originating from the Hopf bifurcation. In particular, there are no isolas of SOPS.	
	\label{conj:ConjRemainder}
\end{conjecture}

Resolving Conjecture \ref{conj:ConjRemainder} is still a nontrivial task.
For $\alpha \geq 5.67$ a (purely analytic) proof is given in~\cite{xie1991thesis}, whereas for $\alpha \in [1.94,6]$ a (computer assisted) proof is provided in~\cite{jlm2016Floquet}. This leaves a gap of parameter values $\alpha \in (\pp, 1.94)$. 
In Theorem~\ref{thm:UniqunessNbd2} we prove a partial result: we construct a neighborhood about the bifurcation point independent of any $ \epsilon$-scaling such that the only periodic orbits for 
$\alpha \in ( \pp , \pp + 0.00553]$
are those originating from the Hopf bifurcation.
This implies that there are no ``spurious'' solutions (for example on isolas) 
in this explicit neighborhood of the bifurcation point. 
By applying the techniques used in \cite{neumaier2014global,jlm2016Floquet}  to rule out solutions which have either a large amplitude or a frequency dissimilar from $\pp$, we expect the Conjecture  \ref{conj:ConjRemainder}  could be proved for 
$ \alpha \in ( \pp , \pp + 0.00553]$.

\section{Preliminaries}
\label{s:preliminaries}
%

In this section we systematically recast the Hopf bifurcation problem in Fourier space. 
We introduce appropriate scalings, sequence spaces of Fourier coefficients and convenient operators on these spaces. 
To study Equation~\eqref{eq:FourierSequenceEquation} we consider Fourier sequences $ \{a_k\}$ and fix a Banach space in which these sequences reside. It is indispensable for our analysis that this space have an algebraic structure. 
The Wiener algebra of absolutely summable Fourier series is a natural candidate, which we use with minor modifications. 
In numerical applications, weighted sequence spaces with algebraic and geometric decay have been used to great effect to study periodic solutions which are $C^k$ and analytic, respectively~\cite{lessard2010recent,hungria2016rigorous}. 
Although it follows from Lemma~\ref{l:analytic} that the Fourier coefficients of any solution decay exponentially, we choose to work in a space of less regularity. 
The reason is that by working in a space with less regularity, we are better able to connect our results with the global estimates in \cite{neumaier2014global}, see Theorem~\ref{thm:UniqunessNbd2}.

%
%

\begin{remark}\label{r:a0}
There is considerable redundancy in Equation~\eqref{eq:FourierSequenceEquation}. First, since we are considering real-valued solutions $y$, we assume $\c_{-k}$ is the complex conjugate of $\c_k$. This symmetry implies it suffices to consider Equation~\eqref{eq:FourierSequenceEquation} for $k \geq 0$.
Second, we may effectively ignore the zeroth Fourier coefficient of any periodic solution \cite{jones1962existence}, since it is necessarily equal to $0$. 
		The self contained argument is as follows. 
		As mentioned in the introduction, any periodic solution to Wright's equation must satisfy $ y(t) > -1$ for all $t$. 
	By dividing Equation~\eqref{eq:Wright} by $(1+y(t))$, which never vanishes, we obtain
	\[
	\frac{d}{dt} \log (1 + y(t)) = - \alpha y(t-1).
	\]  
	Integrating over one period $L$ we derive the condition 
	$0=\int_0^L y(t) dt $.
	Hence $a_0=0$ for any periodic solution. 
	It will be shown in Theorem~\ref{thm:FourierEquivalence1} that a related argument implies that we do not need to consider Equation~\eqref{eq:FourierSequenceEquation} for $k=0$.
\end{remark}


%
We define the spaces of absolutely summable Fourier series
\begin{alignat*}{1}
	\ell^1 &:= \left\{ \{ \c_k \}_{k \geq 1} : 
    \sum_{k \geq 1} | \c_k| < \infty  \right\} , \\
	\ell^1_\bi &:= \left\{ \{ \c_k \}_{k \in \Z} : 
    \sum_{k \in \Z} | \c_k| < \infty  \right\} .
\end{alignat*} 
We identify any semi-infinite sequence $ \{ \c_k \}_{k \geq 1} \in \ell^1$ with the bi-infinite sequence $ \{ \c_k \}_{k \in \Z} \in \ell^1_\bi$ via the conventions (see Remark~\ref{r:a0})
\begin{equation}
  \c_0=0 \qquad\text{ and }\qquad \c_{-k} = \c_{k}^*. 
\end{equation}
In other word, we identify $\ell^1$ with the set
\begin{equation*}
   \ell^1_\sym := \left\{ \c \in \ell^1_\bi : 
	\c_0=0,~\c_{-k}=\c_k^* \right\} .
\end{equation*}
On $\ell^1$ we introduce the norm
\begin{equation}\label{e:lnorm}
  \| \c \| = \| \c \|_{\ell^1} := 2 \sum_{k = 1}^\infty |\c_k|.
\end{equation}
The factor $2$ in this norm is chosen to have a Banach algebra estimate.
Indeed, for $\c, \tilde{\c} \in \ell^1 \cong \ell^1_\sym$ we define
the discrete convolution 
\[
\left[ \c * \tilde{\c} \right]_k = \sum_{\substack{k_1,k_2\in\Z\\ k_1 + k_2 = k}} \c_{k_1} \tilde{\c}_{k_2} .
\]
Although $[\c*\tilde{\c}]_0$ does not necessarily vanish, we have $\{\c*\tilde{\c}\}_{k \geq 1} \in \ell^1 $ and 
\begin{equation*}
	\| \c*\tilde{\c} \| \leq \| \c \| \cdot  \| \tilde{\c} \| 
	\qquad\text{for all } \c , \tilde{\c} \in \ell^1, 
\end{equation*}
hence $\ell^1$ with norm~\eqref{e:lnorm} is a Banach algebra.

By Lemma~\ref{l:analytic} it is clear that any periodic solution of~\eqref{eq:Wright} has a well-defined Fourier series $\c \in \ell^1_\bi$. 
The next theorem shows that in order to study periodic orbits to Wright's equation we only need to study Equation~\eqref{eq:FourierSequenceEquation} 
for $k \geq 1$. For convenience we introduce the notation 
\[
G(\alpha,\omega,\c)_k=
( i \omega k + \alpha e^{ - i \omega k}) \c_k + \alpha \sum_{k_1 + k_2 = k} e^{- i \omega k_1} \c_{k_1} \c_{k_2} \qquad \text{for } k \in \N.
\]
We note that we may interpret the trivial solution $y(t)\equiv 0$ as a periodic solution of arbitrary period.
\begin{theorem}
\label{thm:FourierEquivalence1}
Let $\alpha>0$ and $\omega>0$.
If $\c \in \ell^1 \cong \ell^1_{\sym}$ solves
$G(\alpha,\omega,\c)_k =0$  for all $k \geq 1$,
then $y(t)$ given by~\eqref{eq:FourierEquation} is a periodic solution of~\eqref{eq:Wright} with period~$2\pi/\omega$.
Vice versa, if $y(t)$ is a periodic solution of~\eqref{eq:Wright} with period~$2\pi/\omega$ then its Fourier coefficients $\c \in \ell^1_\bi$ lie in $\ell^1_\sym \cong \ell^1$ and solve $G(\alpha,\omega,\c)_k =0$ for all $k \geq 1$.
\end{theorem}

\begin{proof}	
	If $y(t)$ is a periodic solution of~\eqref{eq:Wright} then it is real analytic by Lemma~\ref{l:analytic}, hence its Fourier series $\c$ is well-defined and $\c \in \ell^1_{\sym}$ by Remark~\ref{r:a0}.
	Plugging the Fourier series~\eqref{eq:FourierEquation} into~\eqref{eq:Wright} one easily derives that $\c$ solves~\eqref{eq:FourierSequenceEquation} for all $k \geq 1$.

To prove the reverse implication, assume that $\c \in \ell^1_\sym$ solves
Equation~\eqref{eq:FourierSequenceEquation} for all $k \geq 1$. Since $\c_{-k}
= \c_k^*$, Equation \eqref{eq:FourierSequenceEquation} is also satisfied for
all $k \leq -1$. It follows from the Banach algebra property and
\eqref{eq:FourierSequenceEquation} that $\{k \c_k\}_{k \in \Z} \in \ell^1_\bi$,
hence $y$, given by~\eqref{eq:FourierEquation}, is continuously differentiable.
	Since~\eqref{eq:FourierSequenceEquation} is satisfied for all $k \in \Z \setminus \{0\}$ (but not necessarily for $k=0$) one may perform the inverse Fourier transform on~\eqref{eq:FourierSequenceEquation} to conclude that
	$y$ satisfies the delay equation 
\begin{equation}\label{eq:delaywithK}
   	y'(t) = - \alpha y(t-1) [ 1 + y(t)] + C
\end{equation}
	for some constant $C \in \R$. 
   Finally, to prove that $C=0$ we argue by contradiction.
   Suppose $C \neq 0$. Then $y(t) \neq -1$ for all $t$.
   Namely, at any point where $y(t_0) =-1$ one would have $y'(t_0) = C$
   which has fixed sign,   hence it would follow that $y$ is not periodic
   ($y$ would not be able to cross $-1$ in the opposite direction, 
   preventing $y$  from being periodic).  
  We may thus divide~\eqref{eq:delaywithK} through by $1 + y(t)$ and obtain 
\begin{equation*}
	\frac{d}{dt} \log | 1 + y(t) | = - \alpha y(t-1) + \frac{C}{1+y(t)} .
\end{equation*}
	By integrating both sides of the equation over one period $L$ and by using that $\c_0=0$, we 
	obtain
	\[
	 C \int_0^L \frac{1}{1+y(t)} dt =0.
	\]
	Since the integrand is either strictly negative or strictly positive, this implies that $C=0$. Hence~\eqref{eq:delaywithK} reduces to~\eqref{eq:Wright},
	and $y$ satisfies Wright's equation. 
\end{proof}

To efficiently study Equation~\eqref{eq:FourierSequenceEquation}, we introduce the following linear operators on $ \ell^1$:
\begin{alignat*}{1}
   [K \c ]_k &:= k^{-1} \c_k  , \\ 
   [ U_\omega \c ]_k &:= e^{-i k \omega} \c_k  .
\end{alignat*}
The map $K$ is a compact operator, and it has a densely defined inverse $K^{-1}$. The domain of $K^{-1}$ is denoted by
\[
  \ell^K := \{ \c \in \ell^1 : K^{-1} \c \in \ell^1 \}.  
\]
The map $U_{\omega}$ is a unitary operator on $\ell^1$, but
it is discontinuous in $\omega$. 
With this notation, Theorem~\ref{thm:FourierEquivalence1} implies that our problem of finding a SOPS to~\eqref{eq:Wright} is equivalent to finding an $\c \in \ell^1$ such that
\begin{equation}
\label{e:defG}
  G(\alpha,\omega,\c) :=
  \left( i \omega K^{-1} + \alpha U_\omega \right) \c + \alpha \left[U_\omega \, \c \right] * \c  = 0.
\end{equation}


Periodic solutions are invariant under time translation: if $y(t)$ solves Wright's equation, then so does $ y(t+\tau)$ for any $\tau \in \R$. 
We remove this degeneracy by adding a phase condition. 
Without loss of generality, if $\c \in \ell^1$ solves Equation~\eqref{e:defG}, we may assume that $\c_1 = \epsilon$ for some 
\emph{real non-negative}~$\epsilon$:
\[
  \ell^1_{\epsilon} := \{\c \in \ell^1 : \c_1 = \epsilon \} 
  \qquad \text{where } \epsilon \in \R,  \epsilon \geq 0.
\]
In the rest of our analysis, we will split elements $\c \in \ell^1$ into two parts: $\c_1$ and $\{\c_{k}\}_{k \geq 2}$.  
We define the basis elements $\e_j \in \ell^1$ for $j=1,2,\dots$ as
\[
  [\e_j]_k = \begin{cases}
  1 & \text{if } k=j, \\
  0 & \text{if } k \neq j.
  \end{cases}
\]
We note that $\| \e_j \|=2$. 
Then we can decompose
any $\c \in \ell^1_\epsilon$ uniquely as
\begin{equation}\label{e:aepsc}
  \c= \epsilon \e_1 + \tc \qquad \text{with}\quad 
  \tc \in \ell^1_0 := \{ \tc \in \ell^1 : \tc_1 = 0 \}.
\end{equation}
We follow the classical approach in studying Hopf bifurcations and consider 
$\c_1 = \epsilon$ to be a parameter, and then find periodic solutions with Fourier modes in $\ell^1_{\epsilon}$.
This approach rewrites the function $G: \R^2 \times \ell^K \to \ell^1$ as a function $\tilde{F}_\epsilon : \R^2 \times \ell^K_0 \to \ell^1$, where 
we denote 
\[
\ell^K_0 := \ell^1_0 \cap \ell^K.
\]
\begin{definition}
We define the $\epsilon$-parameterized family of  functions $\tilde{F}_\epsilon: \R^2 \times \ell^K_0  \to \ell^1$ 
by 
\begin{equation}
\label{eq:fourieroperators}
\tilde{F}_{\epsilon}(\alpha,\omega, \tc) := 
\epsilon [i \omega + \alpha e^{-i \omega}] \e_1 + 
( i \omega K^{-1} + \alpha U_{\omega}) \tc + 
\epsilon^2 \alpha e^{-i \omega}  \e_2  +
\alpha \epsilon L_\omega \tc + 
\alpha  [ U_{\omega} \tc] * \tc ,
\end{equation}
where
$L_\omega : \ell^1_0 \to \ell^1$ is given by
\[
   L_{\omega} := \sigma^+( e^{- i \omega} I + U_{\omega}) + \sigma^-(e^{i \omega} I + U_{\omega}),
\]
with $I$ the identity and  $\sigma^\pm$ the shift operators on $\ell^1$:
\begin{alignat*}{2}
\left[ \sigma^- a \right]_k &:=  a_{k+1}  , \\
\left[ \sigma^+ a \right]_k &:=  a_{k-1}  &\qquad&\text{with the convention } \c_0=0.
\end{alignat*}
The operator $ L_\omega$ is discontinuous in $\omega$ and $ \| L_\omega \| \leq 4$. 
\end{definition} 

We reformulate Theorem~\ref{thm:FourierEquivalence1}  in terms of the map  $\tilde{F}$. 
We note that it follows from Lemma~\ref{l:analytic} and 
Equation~\eqref{eq:FourierSequenceEquation}  
that the Fourier coefficients of any periodic solution of~\eqref{eq:Wright} lie in $\ell^K$.
These observations are summarized in the following theorem.
\begin{theorem}
\label{thm:FourierEquivalence2}
	Let $ \epsilon \geq 0$,  $\tc \in \ell^K_0$, $\alpha>0$ and $ \omega >0$. 
	Define $y: \R\to \R$ as 
\begin{equation}\label{e:ytc}
	y(t) = 
	\epsilon \left( e^{i \omega t }  + e^{- i \omega t }\right) 
	+  \sum_{k = 2}^\infty   \tc_k e^{i \omega k t }  + \tc_k^* e^{- i \omega k t } .
\end{equation}
	Then $y(t)$ solves~\eqref{eq:Wright} if and only if $\tilde{F}_{\epsilon}( \alpha , \omega , \tc) = 0$. 
	Furthermore, up to time translation, any periodic solution of~\eqref{eq:Wright} with period $2\pi/\omega$ is described by a Fourier series of the form~\eqref{e:ytc} with $\epsilon \geq 0$ and $\tc \in \ell^K_0$.
\end{theorem}



Since we want to analyze a Hopf bifurcation, we will want to solve $\tilde{F}_\epsilon = 0$ for small values of~$\epsilon$. 
However, at the bifurcation point, $ D \tilde{F}_0(\pp  ,\pp , 0)$ is not invertible.
In order for our asymptotic analysis to be non-degenerate,
we work with a rescaled version of the problem. To this end, for any $\epsilon >0$, we rescale both $\tc$ and $\tilde{F}$ as follows. Let $\tc = \epsilon c$ and 
\begin{equation}\label{e:changeofvariables}
  \tilde{F}_\epsilon (\alpha,\omega,\epsilon c) = \epsilon F_\epsilon (\alpha,\omega,c).
\end{equation}
For $\epsilon>0$ the problem then reduces to finding zeros of 
\begin{equation}
\label{eq:FDefinition}
	F_\epsilon(\alpha,\omega, c) := 
	[i \omega + \alpha e^{-i \omega}] \e_1 + 
	( i \omega K^{-1} + \alpha U_{\omega}) c + 
	\epsilon \alpha e^{-i \omega} \e_2  +
	\alpha \epsilon L_\omega c + 
	\alpha \epsilon [ U_{\omega} c] * c.
\end{equation}
We denote the triple $(\alpha,\omega,c) \in \R^2 \times \ell^1_0$ by $x$.
To pinpoint the components of $x$ we use the projection operators
\[
   \pi_\alpha x = \alpha, \quad \pi_\omega x = \omega, \quad 
  \pi_c x = c \qquad\text{for any } x=(\alpha,\omega,c).
\]

After the change of variables~\eqref{e:changeofvariables} we now have an invertible Jacobian $D F_0(\pp  ,\pp , 0)$ at the bifurcation point.
On the other hand, for $\epsilon=0$ the zero finding problems for $\tilde{F}_\epsilon$ and $F_\epsilon$ are not equivalent. 
However, it follows from the following lemma that any nontrivial periodic solution having $ \epsilon=0$ must have a relatively large size when $ \alpha $ and $ \omega $ are close to the bifurcation point. 

\begin{lemma}\label{lem:Cone}
	Fix $ \epsilon \geq 0$ and $\alpha,\omega >0$. 
	Let
	\[
	b_* :=  \frac{\omega}{\alpha} - \frac{1}{2} - \epsilon  \left(\frac{2}{3}+ \frac{1}{2}\sqrt{2 + 2 |\omega-\pp| } \right).
	\]
Assume that $b_*> \sqrt{2} \epsilon$. 
Define
\begin{equation}\label{e:zstar}
z^{\pm}_* :=b_* \pm \sqrt{(b_*)^2- 2 \epsilon^2 } .
\end{equation}
If there exists a $\tc \in \ell^1_0$ such that $\tilde{F}_\epsilon(\alpha, \omega,\tc) = 0$, then \\
\mbox{}\quad\textup{(a)} either $ \|\tc\| \leq  z_*^-$ or $ \|\tc\| \geq z_*^+  $.\\
\mbox{}\quad\textup{(b)} 
$ \| K^{-1} \tc \| \leq (2\epsilon^2+ \|\tc\|^2) / b_*$. 
\end{lemma}
\begin{proof}
	The proof follows from Lemmas~\ref{lem:gamma} and~\ref{lem:thecone} in Appendix~\ref{appendix:aprioribounds}, combined with the observation that
$\frac{\omega}{\alpha} - \gamma \geq b_*$,
with $\gamma$ as defined in Lemma~\ref{lem:gamma}.
\end{proof}

\begin{remark}\label{r:smalleps}
We note that for $\alpha < 2\omega$
\begin{alignat*}{1}
z^+_* &\geq   \frac{2 \omega - \alpha}{\alpha} 
- \epsilon \left(4/3+\sqrt{2 + 2 |\omega-\pp| } \, \right) + \cO(\epsilon^2)
\\[1mm]
z^-_* & \leq   \cO(\epsilon^2)
\end{alignat*}
for small $\epsilon$. 
Hence Lemma~\ref{lem:Cone} implies that for values of $(\alpha,\omega)$ near $(\pp,\pp)$ any solution has either $\|\tc\|$ of order 1 or $\|\tc\| =  \cO(\epsilon^2)$. 
The asymptotically small term bounding $z_*^-$ is explicitly calculated in Lemma~\ref{lem:ZminusBound}. 
A related consequence is that for $\epsilon=0$ there are no nontrivial solutions 
of $\tilde{F}_0(\alpha,\omega,\tc)=0$ with 
$\| \tc \| < \frac{2 \omega - \alpha}{\alpha} $. 
\end{remark}

\begin{remark}\label{r:rhobound}
In Section~\ref{s:contraction} we will work on subsets of $\ell^K_0$ of the form
\[
  \ell_\rho := \{ c \in \ell^K_0 : \|K^{-1} c\| \leq \rho \} .
\]
Part (b) of Lemma~\ref{lem:Cone} will be used in Section~\ref{s:global} to guarantee that we are not missing any solutions by considering $\ell_\rho$ (for some specific choice of $\rho$) rather than the full space $\ell^K_0$.
In particular, we infer from Remark~\ref{r:smalleps} that  small solutions (meaning roughly that $\|\tc\| \to 0$ as $\epsilon \to 0$)
satisfy $\| K^{-1} \tc \| = \cO(\epsilon^2)$.
\end{remark}

The following theorem guarantees that near the bifurcation point the problem of finding all periodic solutions is equivalent to considering the rescaled problem $F_\epsilon(\alpha,\omega,c)=0$.
\begin{theorem}
\label{thm:FourierEquivalence3}
\textup{(a)} Let $ \epsilon > 0$,  $c \in \ell^K_0$, $\alpha>0$ and $ \omega >0$. 
	Define $y: \R\to \R$ as 
\begin{equation}\label{e:yc}
	y(t) = 
	\epsilon \left( e^{i \omega t }  + e^{- i \omega t }\right) 
	+ \epsilon  \sum_{k = 2}^\infty   c_k e^{i \omega k t }  + c_k^* e^{- i \omega k t } .
\end{equation}
	Then $y(t)$ solves~\eqref{eq:Wright} if and only if $F_{\epsilon}( \alpha , \omega , c) = 0$.\\
\textup{(b)}
Let $y(t) \not\equiv 0$ be a periodic solution of~\eqref{eq:Wright} of period $2\pi/\omega$
 with Fourier coefficients $\c$.
Suppose $\alpha < 2\omega$ and $\| \c \| < \frac{2 \omega - \alpha}{\alpha} $.
Then, up to time translation, $y(t)$ is described by a Fourier series of the form~\eqref{e:yc} with $\epsilon > 0$ and $c \in \ell^K_0$.
\end{theorem}

\begin{proof}
Part (a) follows directly from Theorem~\ref{thm:FourierEquivalence2} and the  change of variables~\eqref{e:changeofvariables}.
To prove part (b) we need to exclude the possibility that there is a nontrivial solution with $\epsilon=0$. The asserted bound on the ratio of $\alpha$ and $\omega$ guarantees, by Lemma~\ref{lem:Cone} (see also Remark~\ref{r:smalleps}), that indeed $\epsilon>0$ for any nontrivial solution. 
\end{proof}

We note that in practice (see Section~\ref{s:global}) a bound on $\| \c \|$ is derived from a bound on $y$ or $y'$ using Parseval's identity.

\begin{remark}\label{r:cone}
It follows from Theorem~\ref{thm:FourierEquivalence3} and Remark~\ref{r:smalleps} that for values of $(\alpha,\omega)$ near $(\pp,\pp)$ any reasonably bounded solution satisfies $\| c\| =  O(\epsilon)$ as well as $\|K^{-1} c \| = O(\epsilon)$ asymptotically (as $\epsilon \to 0$).
These bounds will be made explicit (and non-asymptotic) for specific choices of the parameters in Section~\ref{s:global}.
\end{remark}



We finish this section by defining a curve of approximate zeros $\bx_\epsilon$ of $F_\epsilon$ 
(see \cite{chow1977integral,hassard1981theory}). 

\begin{definition}\label{def:xepsilon}
Let
\begin{alignat*}{1}
	\balpha_\epsilon &:= \pp + \tfrac{\epsilon^2}{5} ( \tfrac{3\pi}{2} -1)  \\
	\bomega_\epsilon &:= \pp -  \tfrac{\epsilon^2}{5} \\
	\bc_\epsilon 	 &:= \left(\tfrac{2 - i}{5}\right) \epsilon \,  \e_2 \,.
\end{alignat*}
We define the approximate solution 
$ \bx_\epsilon := \left( \balpha_\epsilon , \bomega_\epsilon  , \bc_\epsilon \right)$
for all $\epsilon \geq 0$.
\end{definition}

We leave it to the reader to verify that both 
 $F_\epsilon(\pp,\pp,\bc_{\epsilon})=\cO(\epsilon^2)$ and $F_\epsilon(\bx_\epsilon)=\cO(\epsilon^2)$.
We choose to use the more accurate approximation 
for the $ \alpha$ and $ \omega $ components to improve our final quantitative results.

\section{Local results}
\label{s:local}

\subsection{Constructing a Newton-like operator}
\label{s:newtonlike}

In this section and in the appendices we often suppress the subscript in $F=F_\epsilon$.
We will find solutions to the equation $F(\alpha ,\omega , c)=0$ by the
constructing a Newton-like operator $T$ such that fixed points of $T$
corresponds precisely to zeros of $F$. In order to construct the map $T$ we
need an operator $A^{\dagger}$ which is an approximate inverse of 
$DF(\bx_\epsilon)$. 
We will use an approximation $A$ of 
$DF( \bx_\epsilon )$ that is linear in~$\epsilon$ and correct up to $\cO(\epsilon^2)$.
Likewise, we define $A^{\dagger}$ to be linear in $\epsilon$ (and again correct up to $\cO(\epsilon^2)$). 

It will be convenient to use the usual identification $i_\C : \R^2 \to \C$ given by $i_\C (x,y) = x+iy $. We also use $\omega_0 := \pi/2$.

\begin{definition}\label{def:A}
We introduce the linear maps $A:  \R^2 \times \ell^K_0 \to \ell^1$ and 
$ A^{\dagger}:  \ell^1 \to  \R^2 \times \ell^K_0 $ by
\begin{alignat*}{1}
A &:= A_0 + \epsilon A_1 \, , \\
A^{\dagger} &:= A_0^{-1} - \epsilon A_0^{-1} A_1 A_0^{-1} \,  ,
\end{alignat*}
where the linear maps $ A_0 , A_1 : \R^2 \times \ell^K_0 \to \ell^1$  are defined below. Writing $x=(\alpha,\omega,c)$, we set
\begin{alignat*}{1}
A_0	x = A_0 (\alpha,\omega,c) & := i_\C A_{0,1} 
\!\left[\!\! \begin{array}{c} \alpha \\ \omega \end{array} \!\!\right]  \e_1
 + A_{0,*}  c , \\
A_1 x =	A_1 (\alpha,\omega,c) & := i_\C  A_{1,2}
\!\left[\!\! \begin{array}{c} \alpha \\ \omega \end{array} \!\!\right]  \e_2
 + A_{1,*}  c .
\end{alignat*}
Here the matrices $A_{0,1}$ and $A_{1,2}$ are given by
\begin{equation}
A_{0,1} := 
\left[
\begin{matrix}
0 & - \pp \\
-1  & 1 
\end{matrix} 
\right]
\qquad\text{and}\qquad
A_{1,2} := \frac{1}{5}
\left[
\begin{matrix}
-2 & 2-\tfrac{3 \pi}{2} \\
-4  & 2(2+\pi) 
\end{matrix}  
\right]  ,
\label{eq:defA12}
\end{equation}
and the linear maps $A_{0,*} : \ell^K_0 \to \ell^1_0$ and
$A_{1,*} : \ell^K_0 \to \ell^1$
are given by
\begin{equation*}
A_{0,*} 	 := \tfrac{\pi}{2} ( i K^{-1} + U_{\omega_0}) 
\qquad\text{and}\qquad
A_{1,*} 	:= \tfrac{\pi}{2} L_{\omega_0} .
\end{equation*}
\end{definition}

Since $K$ and $U_{\omega_0}$ both act as diagonal operators, the inverse 
$A_{0,*}^{-1} : \ell^1_0 \to \ell^K_0$ of $A_{0,*}$ is given by
\begin{equation*}
	  (A_{0,*}^{-1} a)_k = \frac{2}{\pi} \frac{a_k}{ik+e^{-ik\omega_0}} 
	  \qquad\text{for all } k \geq 2.
\end{equation*} 
An explicit computation, which we leave to the reader, shows that these approximations are indeed correct up to $\cO(\epsilon^2)$. 
In particular, $A^{\dagger} = \left[ DF( \bx_\epsilon ) \right]^{-1} + \cO(\epsilon^2)$.
In Appendix~\ref{sec:OperatorNorms} several additional properties of these operators are derived. The most important one is the following.
\begin{proposition}
	\label{prop:Injective}
	For 
	$0 \leq \epsilon < \tfrac{\sqrt{10}}{4} \approx 0.790$
	 the operator $ A^{\dagger}$ is injective. 
\end{proposition}
\begin{proof}
	In order to show that $ A^{\dagger}$ is injective we show that 
	it has a left inverse. 
	Note that $ A A^{\dagger} = I - \epsilon^2 ( A_1 A_0^{-1})^2$. 
	By Proposition \ref{prop:A1A0} it follows that 
	 $ \| A_1 A_0^{-1} \| \leq \tfrac{2 \sqrt{10}}{5} $.  
	By choosing 
$ \epsilon < \tfrac{\sqrt{10}}{4}$ 
we obtain 
	$\|  \epsilon^2 ( A_1 A_0^{-1})^2 \| < 1$, whereby $ A A^{\dagger}$ is 
	invertible, and so $ A^{\dagger}$ is injective. 
\end{proof}

\begin{definition}
We define the operator $ T: \R^2 \times \ell^K_0 \to \R^2 \times \ell^K_0 $ by
\begin{equation*}
	T(x) :=  x - A^{\dagger} F(x) ,
\end{equation*}
	where  $F$ is defined in Equation~\eqref{eq:FDefinition}  and $A^{\dagger}$ in Definition~\ref{def:A}.
	We note that $F$, $A^{\dagger}$ and $T$ depend on the parameter $\epsilon \geq 0$, although we suppress this in the notation.
\end{definition}

%

\subsection{Explicit contraction bounds}
\label{s:contraction}

The map $T$ is not continuous on all of $\R^2 \times \ell^K_0$,
since $ U_{\omega} c $ is not continuous in $\omega$.
While continuity is ``recovered'' for terms of the form $A^{\dagger} U_{\omega} c$,  this is not the case for the nonlinear part $ - \alpha \epsilon A^{\dagger} [ U_{\omega} c ] * c$.  
%
%
We overcome this difficulty by fixing some $ \rho > 0$ and restricting the domain of $T$ to sets of the form 
\[
  \R^2 \times  \{ c \in \ell^K_0 : \|K^{-1} c\| \leq \rho \} = \R^2 \times \ell_\rho.
\]
Since we wish to center the domain of $T$ about the approximate solution~$\bx_\epsilon$, we introduce the following definition, which uses a triple of radii $r \in \R^3_+$, for which it will be convenient to use two different notations:
\[
  r = ( r_{\alpha } , r_{\omega} , r_c) = (r_1,r_2,r_3).
\]
\begin{definition}
	Fix   $ r \in \R^3_+$ and $ \rho > 0$ and let  $ \bx_\epsilon = ( \balpha_\epsilon , \bomega_\epsilon , \bc_\epsilon )$ be as defined in Definition~\ref{def:xepsilon}. 
    We define the $\rho$-ball $B_\epsilon(r,\rho) \subset \R^2 \times \ell^1_0$
    of radius $r$ centered at $\bx_\epsilon$ to be the set of points satisfying 
\begin{alignat*}{1}
	|  \alpha -\balpha_\epsilon | & \leq  r_\alpha  \\
	| \omega - \bomega_\epsilon  | & \leq  r_{\omega} \\
	\| c - \bc_\epsilon  \| & \leq r_c \\
	\|K^{-1} c\| & \leq  \rho .
\end{alignat*}
\end{definition}

We want to show that $T$ is a contraction map on some $\rho$-ball 
$B_\epsilon(r,\rho) \subset \R^2 \times \ell^1_0$ using a Newton-Kantorovich argument. 
This will require us to develop a bound on $DT$ using some norm on  $ X$.  
Unfortunately there is no natural choice of norm on the product space $ X$. 
Furthermore, it will not become apparent if one norm is better than another until after  significant calculation.  
For this reason, we use a notion of an ``upper bound'' which allows us to delay our choice of norm. 
We first introduce the operator $\zeta:  X  \to \R^3_{+}$
which consists of the norms of the three components:
\[
  \LL(x) :=   ( |\pi_\alpha x|, |\pi_\omega x|, \|\pi_c x\| )^T \in \R^3_{+}
  \qquad\text{for any } x \in X.
\]
\begin{definition}[upper bound]\label{def:upperbound}
We call $\upperbound{x} \in \R^3_+$ an upper bound on $x$ if $\LL(x) \leq \upperbound{x}$, where the inequality is interpreted componentwise in $\R^3$. 
Let $X'$ be a subspace of $X$ and let $X''$ be a subset of $X'$.   
An upper bound on a linear operator $A' : X' \to X $ over $X''$  is 
a $3 \times 3$ matrix $\upperbound{A'} \in \text{\textup{Mat}}(\R^3 , \R^3)$ such that
\[
   \LL(A' x ) \leq \upperbound{A'} \cdot \LL(x)  
     \qquad\text{for any }  x \in X'',
\]
where the inequality is again interpreted componentwise in $\R^3$. 
The notion of upper bound conveniently encapsulates bounds on the different components of the operator $A'$ on the product space $X$. Clearly the components of the matrix $\upperbound{A'}$ are nonnegative.

\end{definition}
For example, in Proposition \ref{prop:A0A1} we calculate an upper  bound on the map $A_0^{-1} A_1$.  
As for the domain of definition of $T$, in practice we use $X' = \R^2 \times  \ell^K_0  $ and  $X'' = \R^2 \times  \ell_\rho  $.
The subset $X''$ does not always affect the upper bound calculation (such as in Proposition \ref{prop:A0A1}). 
However, operators such as $U_{\omega} - U_{\omega_0}$ have upper bounds which contain $\rho$-terms (see for example Proposition \ref{prop:OmegaDerivatives}).

Using this terminology, we state a ``radii polynomial'' theorem, which allows us to check whether $T$ is a contraction map. This technique has been used frequently in a computer-assisted setting in the past decade. Early application include~\cite{daylessardmischaikow,lessardvandenberg}, while a previous implementation in the context of Wright's delay equation can be found in~\cite{lessard2010recent}. 
Although we use radii polynomials as well, our approach differs significantly from the computer-assisted setting mentioned above. 
While we do engage a computer (namely the Mathematica file~\cite{mathematicafile}) to optimize our quantitative results, the analysis is performed essentially in terms of pencil-and-paper mathematics (in particular, our operators do not involve any floating point numbers).
In our current setup we employ \emph{three} radii as a priori unknown variables,
which builds on an idea introduced in~\cite{vandenberg}.
We note that in most of the papers mentioned above the notation of $A$ and $A^\dagger$ is reversed compared to the current paper.

As preparation, the following lemma (of which the proof can be found in Appendix~\ref{sec:CompactDomain})  provides an explicit choice for $\rho$, as a function of $\epsilon$ and $r$, for which we have proper control on the image of $B_\epsilon(r,\rho)$ under $T$.
\begin{lemma}\label{lem:Crho}
For any $\epsilon \geq 0$ and $r \in\R^3_+$, let $C=C(\epsilon,r)$ be given  by Equation~\eqref{eq:RhoConstant}. 
If $C(\epsilon,r) >0$  then 
\begin{equation}\label{e:Cepsr} 
  \| K^{-1} \pi_c  T(x) \| \leq \rho 
  \quad\text{whenever } x \in B_\epsilon(r,\rho) \text{ and } \rho \geq C(\epsilon,r).
\end{equation}
Moreover, $C(\epsilon,r)$ is nondecreasing in $\epsilon$ and $r$. 
\end{lemma}

\begin{proof}
See Proposition~\ref{prop:DerivativeEndo}.
\end{proof}

\begin{theorem}
	\label{thm:RadPoly}
	Let  
 $0 \leq \epsilon < \tfrac{\sqrt{10}}{4} $  
 and fix $r = (r_\alpha, r_\omega, r_c) \in \R^3_+$. Fix $\rho > 0$ such that $ \rho \geq C(\epsilon,r)$, as given by Lemma~\ref{lem:Crho}.
%
Suppose that $Y(\epsilon) $ is an upper bound on $ T(\bx_\epsilon) - \bx_\epsilon$ and $Z(\epsilon , r ,\rho) $ a (uniform) upper bound on $ DT(x) $ for all $ x \in B_\epsilon(r,\rho)$. 
Define the \emph{radii polynomials}
$P :\R^5_+ \to \R^3 $  by 
 \begin{equation}
 \label{eq:RadPolyDef}
  P(\epsilon,r,\rho) := Y(\epsilon) - \left[ I - Z( \epsilon,r,\rho) \right] \cdot r  \,  .
 \end{equation}
If each component of $P(\epsilon,r,\rho)$ is negative, then there is  a unique $\hat{x}_\epsilon \in B_\epsilon( r , \rho)$ such that $F(\hat{x}_\epsilon) =0$. 
\end{theorem}

\begin{proof}    
Let $r \in \R^3_+$ be a triple such that $P(\epsilon,r,\rho)<0$.
By Proposition \ref{prop:Injective}, if 
$\epsilon <\tfrac{\sqrt{10}}{4} $
then $ A^{\dagger}$ is injective. 
Hence $ \hat{x}_{\epsilon} $ is a fixed point of $T$ if and only if $ F( \hat{x}_{\epsilon}) = 0$.  
In order to show  there is a unique fixed point $ \hat{x}_{\epsilon}$, we show that $T$ maps  $ B_{\epsilon}(r,\rho) $ into itself and that $ T $ is a contraction mapping. 

We first show that $T: B_\epsilon(r,\rho) \to B_\epsilon(r,\rho)$. 
Since $ \rho \geq C(\epsilon,r)$ then by Equation~\eqref{e:Cepsr} it follows that $ \| K^{-1} \pi_c T( x) \| \leq \rho$ for all $ x \in B_\epsilon(r,\rho)$.
In order to show that $T(x) \in B_\epsilon(r,\rho)$, it suffices to show that $ r=(r_\alpha , r_\omega, r_c)$ is an upper bound on $ T(x) - \bx_\epsilon$
for all $ x \in B_\epsilon(r,\rho)$.  
We decompose 
\begin{equation}\label{e:Tsplit}
	T(x) - \bx_\epsilon = [T(\bx_\epsilon) -\bx_\epsilon] +
	[T(x) - T(\bx_\epsilon)],
\end{equation}
and estimate each part separately. Concerning the first term,
by assumption, $Y(\epsilon)$ is an upper bound on $T(\bx_\epsilon) - \bx_\epsilon$. 
Concerning the second term, we claim that $ Z(\epsilon,r,\rho) \cdot r$ is an upper bound on $T(x) - T(\bx_\epsilon)$.
Indeed, we have the following somewhat stronger bound: 
\begin{equation}\label{e:DTisboundedbyZ}
	\LL(T(y) - T(x)) \leq Z(\epsilon,r,\rho) \cdot \LL(y-x)
	\qquad\text{for all } x,y \in B_\epsilon(r,\rho) .
\end{equation}
The latter follows from the mean value theorem, since 
$T$ is continuously Fr\'echet differentiable on $B_\epsilon(r,\rho)$.
%
%
Since $r$ is an upper bound on $x - \bx_\epsilon$ for all $ x \in B_\epsilon(r,\rho)$, we find, by using~\eqref{e:Tsplit}, that  
$Y(\epsilon) + Z(\epsilon,r,\rho) \cdot r \leq r$ (with the inequality, interpreted componentwise, following from $P(\epsilon,r,\rho)<0$) is an upper bound on $T(x) - \bx_\epsilon$
for all $ x \in B_\epsilon(r,\rho)$.  
That is to  say, if all of the  radii polynomials are negative, 
then  $T$ maps $B_\epsilon(r,\rho) $ into itself.

To finish the proof we show that $T$ is a contraction mapping. 
We abbreviate $Z=Z(\epsilon,r,\rho)$ and  recall that $r=(r_\alpha,r_\omega,r_c)=(r_1,r_2,r_3) \in \R^3_+$
is such that $Z \cdot r < r$, hence for some $\kappa <1$ we have
\begin{equation}\label{e:defkappa}
  \frac{(Z \cdot r)_i}{r_i} \leq \kappa  \qquad\text{for } i=1,2,3.
\end{equation}

We now need to choose a norm on $X$. 
We define a norm $ \| \cdot \|_r$ on elements $x = (\alpha,\omega,c) \in X$
by
\[  
\| (\alpha, \omega, c) \|_r := \max 
\left\{  		  
	 \frac{|\alpha|}{r_\alpha},
	 \frac{|\omega|}{r_\omega},
	 \frac{\|c\|}{r_{c}} \right\} , 
\]
or
\[
  \|x\|_r = \max_{i=1,2,3} \frac{ \LL(x)_i}{r_i}
  \qquad \text{for all } x \in X.
\]
%
By using the upper bound $Z$, we bound the Lipschitz constant of $T$ on $B_\epsilon(r, \rho)$ as follows:
\begin{alignat*}{1}
 \| T(y) - T(x) \|_r 
    &= \max_{i=1,2,3} \frac{\LL(T(y) - T(x))_i} {r_i} \\
    &\leq  \max_{i=1,2,3}  \frac{(Z \cdot \LL(y-x))_i}{r_i} \\
    &\leq  \max_{i=1,2,3} \max_{j=1,2,3}\frac{\LL(y-x)_j}{r_j}  
				   \frac{(Z \cdot r )_i}{r_i} \\
    & = \| y-x \|_r \max_{i=1,2,3} \frac{(Z \cdot r )_i}{r_i} \\
    & \leq \kappa \| y-x \|_r,
\end{alignat*}
where we have used~\eqref{e:DTisboundedbyZ} and~\eqref{e:defkappa} with $\kappa<1$.
Hence $T:B_{\epsilon}(r,\rho) \to B_{\epsilon}(r,\rho)$ is a contraction with respect to the $\| \cdot \|_r$ norm.

Since $B_\epsilon(r,\rho)$ with this norm is a complete metric space, by the Banach fixed point theorem $T$~has a unique fixed point $ \hat{x}_\epsilon \in B_\epsilon(r,\rho)$. 
Since $A^\dagger$ is injective,  it follows that $ \hat{x}_\epsilon$   is the unique point in $B_\epsilon(r,\rho)$ for which $ F(\hat{x}_\epsilon) =0$. 
\end{proof}

\begin{remark}\label{r:boundDT}
Under the assumptions in Theorem~\ref{thm:RadPoly},
essentially the same calculation as in the proof above
leads to the estimate
\[
  \| DT(x) y \|_r \leq \kappa \|y\|_r 
  \qquad \text{for all } y \in \R^2 \times \ell^K_0 , 
  \, x \in B_\epsilon(r,\rho),
\]
where $\kappa := \max_{i=1,2,3} (Z\cdot r)_i / r_i$.
\end{remark}

In Appendix \ref{sec:YBoundingFunctions} and Appendix \ref{sec:BoundingFunctions} we construct explicit upper bounds 
$Y(\epsilon)$ and $ Z(\epsilon,r,\rho)$, respectively.  
These functions are constructed such that their components are (multivariate) polynomials in $\epsilon$, $r$ and $ \rho$ with nonnegative coefficients, hence they are increasing in these variables. 
This construction enables us to make use of the uniform contraction principle. 

\begin{corollary}\label{cor:eps0}
Let 
 $0 < \epsilon_0 < \tfrac{\sqrt{10}}{4} $ 
and fix some $r = (r_\alpha, r_\omega, r_c) \in \R^3_+$.  
Fix $\rho > 0$ such that $ \rho \geq C(\epsilon_0,  r)$, as given by Lemma~\ref{lem:Crho}.
%
Let $Y(\epsilon)$ and $Z(\epsilon,r,\rho)$ be the upper bounds as given in  Propositions~\ref{prop:Ydef} and~\ref{prop:Zdef}. 
Let the radii polynomials $P$ be defined by Equation~\eqref{eq:RadPolyDef}.

If each component of  $P(\epsilon_0, r,\rho)$ is negative, 
then for all $ 0 \leq \epsilon \leq \epsilon_0$ there exists a unique $ \hat{x}_\epsilon \in B_\epsilon(  r , \rho)$ such that $ F(\hat{x}_\epsilon) =0$.  
The solution $\hat{x}_\epsilon$ depends smoothly on $\epsilon$.
\end{corollary}
\begin{proof} 
	Let $0 \leq  \epsilon \leq \epsilon_0$ be arbitrary.
	Because $\rho \geq C(\epsilon_0, r) \geq C(\epsilon, r)$ by Lemma~\ref{lem:Crho},
	Theorem~\ref{thm:RadPoly} implies that it suffices to show that $ P(\epsilon, r ,\rho) <0$. 	
Since  the bounds 
$Y(\epsilon)$ and $ Z(\epsilon,r,\rho)$ are monotonically increasing in their arguments, it follows that $ P(\epsilon,r,\rho) \leq P(\epsilon_0,r,\rho) <0$.  
Continuous and smooth dependence on $\epsilon$ of the fixed point follows from the uniform contraction principle (see for example~\cite{ChowHale}). 
\end{proof}

Given the upper bounds $ Y(\epsilon)$ and $ Z( \epsilon ,r , \rho)$, 
trying to apply Corollary~\ref{cor:eps0} amounts to finding values of $ \epsilon, r_\alpha, r_\omega, r_c,\rho$ for which the radii polynomials are negative.
Selecting a value for $ \rho$ is straightforward: all estimates improve with smaller values of $\rho$, and Proposition \ref{prop:DerivativeEndo} (see also Lemma~\ref{lem:Crho}) explicitly describes the smallest allowable choice of $\rho$ in terms of $ \epsilon,r_\alpha,r_\omega,r_c$. 

Beyond selecting a value for $ \rho$, it is difficult to pinpoint what constitutes an ``optimal'' choice of these variables. 
In general it is interesting to find such  viable radii (i.e.\ radii such that $P(r)<0$) which are both large and small.  
The smaller radius tells us how close the true solution is to our approximate solution. 
The larger radius tells us in how large a neighborhood our solution is unique.  With regard to $\epsilon$, larger values allow us to describe functions whose first Fourier mode is large. However this will ``grow'' the smallest viable radius and ``shrink'' the largest viable radius. 

Proposition \ref{prop:bigboxes} presents two selections of variables which satisfy the hypothesis of Corollary~\ref{cor:eps0}.  
We check the hypothesis is indeed satisfied by using interval arithmetic.
All details are provided in the Mathematica file~\cite{mathematicafile}. 
While the specific numbers used may appear to be somewhat arbitrary (see also the discussion in Remark~\ref{r:largeradii})  they have been chosen to be used later in Theorem 
\ref{thm:WrightConjecture} and Theorem \ref{thm:UniqunessNbd}.


\begin{proposition}
		\label{prop:bigboxes}
Fix the constants $ \epsilon_0$, $(r_\alpha, r_\omega,r_c)$  and $\rho$ according to one of the following choices:
\begin{enumerate}
	\item[\textup{(a)}]  $ \epsilon_0 = 0.029 $ and $ (r_\alpha , r_ \omega , r_c) = (  0.13, \, 0.17 , \, 0.17 ) $ and $\rho = 1.78$; 
	\item[\textup{(b)}]  $ \epsilon_0 = 0.09 $ and $ (r_\alpha , r_ \omega , r_c) = (  0.1753, \, 0.0941 , \, 0.3829 ) $ and $\rho = 1.5940$. 
\end{enumerate}
For either of the choices (a) and (b) we have the following: 
for all $0 \leq \epsilon \leq \epsilon_0$ there exists a unique point 
$(\hat{\alpha}_\epsilon,\hat{\omega}_\epsilon,\hat{c}_\epsilon) \in B_{\epsilon}(r,\rho)$ 
satisfying $F_\epsilon(\hat{\alpha}_\epsilon,\hat{\omega}_\epsilon,\hat{c}_\epsilon) = 0$ and 
\[ 	
 | \hat{\alpha}_\epsilon - \balpha_\epsilon| \leq r_\alpha , 
 \quad
 |\hat{\omega}_\epsilon - \bomega_\epsilon| \leq  r_\omega  ,
 \quad
 \| \hat{c}_\epsilon - \bc_\epsilon\| \leq r_c     ,
 \quad
 \| K^{-1} \hat{c}_\epsilon \| \leq \rho  .
\]
\end{proposition}
\begin{proof}
In the Mathematica file~\cite{mathematicafile}  we check, using interval arithmetic, that  $\rho \geq C(\epsilon_0, r)$ and  the radii polynomials $P(\epsilon_0,r,\rho)$ are negative for the choices (a) and (b). The result then follows from Corollary~\ref{cor:eps0}.	
\end{proof}

\begin{remark}\label{r:largeradii}	
In Proposition~\ref{prop:bigboxes} we aimed for large balls on which the solution is unique.
Even for a fixed value of $ \epsilon$, it is not immediately obvious how to find a ``largest'' viable radius $r$, 
since $r$ has three components. In particular, there is a trade-off between the different components of $r$. On the other hand, as explained in Remark~\ref{r:smallradii}, no such difficulty arises when looking for a ``smallest'' viable radius.
\end{remark}

We will also need a rescaled version of the radii polynomials, which takes into account the asymptotic behavior of the bound $Y$ on the residue $T(\bar{x}_\epsilon) -\bar{x}_\epsilon = - A^\dagger F(\bx_\epsilon)$  as $\epsilon \to 0$, namely it is of the form
$Y(\epsilon)= \epsilon^2 \tilde{Y}(\epsilon)$,
see Proposition~\ref{prop:Ydef}.
The proofs of the following monotonicity properties can be found in 
Appendices~\ref{sec:YBoundingFunctions} and~\ref{sec:BoundingFunctions}. 
\begin{lemma}\label{lem:YZ}
Let $\epsilon \geq 0$, $\rho >0$ 
and $r \in\R^3_+$. 
Then there are upper bounds
$Y(\epsilon) =\epsilon^2 \tilde{Y}(\epsilon)$ on $ T(\bx_\epsilon) - \bx_\epsilon$ and a (uniform) upper bound 
$Z(\epsilon , r ,\rho) $  on $ DT(x) $ for all $ x \in B_\epsilon(r,\rho)$.
These bounds are given explicitly by Propositions~\ref{prop:Ydef} and~\ref{prop:Zdef}, respectively. Moreover, $\tilde{Y}(\epsilon)$ is nondecreasing in $\epsilon$,
while $Z(\epsilon , r ,\rho) $ is nondecreasing in $\epsilon$, $r$ and $\rho$.
\end{lemma}

This implies, roughly speaking, that if we are able to show that $T$ is a contraction map on 
$B_{\epsilon_0}( \epsilon_0^2 \rr,\rho)$ for a particular choice of $ \epsilon_0$, then it will be a contraction map on $B_\epsilon( \epsilon^2 \rr,\rho)$ for all $ 0 \leq \epsilon \leq \epsilon_0$. Here, and in what follows, we use the notation $r = \epsilon^2 \rr$ for the $\epsilon$-scaled version of the radii.

\begin{corollary}
	\label{cor:RPUniformEpsilon}
	Let  
	 $0 < \epsilon_0 < \tfrac{\sqrt{10}}{4} $ 
	and fix some $\rr = (\rr_\alpha, \rr_\omega, \rr_c) \in \R^3_+$. 
	Fix $\rho > 0$ such that $ \rho \geq C(\epsilon_0, \epsilon_0^2 \rr)$, as given by Lemma~\ref{lem:Crho}. 
	Let $Y(\epsilon)$ and $Z(\epsilon,r,\rho)$ be the upper bounds as given by Lemma~\ref{lem:YZ}.  
Let the radii polynomials $P$ be defined by~\eqref{eq:RadPolyDef}. 

	If each component of  $P(\epsilon_0,\epsilon_0^2 \rr,\rho)$ is negative, 
	then for all $ 0 \leq \epsilon \leq \epsilon_0$ 
	there exists a unique $ \hat{x}_\epsilon \in B_\epsilon(\epsilon^2  \rr , \rho)$ 
	such that $ F(\hat{x}_\epsilon) =0$. 
	Furthermore, $\hat{x}_\epsilon$ depends smoothly on $\epsilon$.
\end{corollary}

\begin{proof}
	 Let $0 \leq  \epsilon < \epsilon_0$ be arbitrary.
	 Because $\rho \geq C(\epsilon_0,\epsilon_0^2 \rr) \geq C(\epsilon,\epsilon^2 \rr)$ by Lemma~\ref{lem:Crho},
	Theorem~\ref{thm:RadPoly} implies that it suffices to show that $ P(\epsilon,\epsilon^2 \rr ,\rho) <0$. 
	By using the monotonicity provided by Lemma~\ref{lem:YZ}, we obtain
	\begin{alignat*}{1}
		P(\epsilon,\epsilon^2 \rr ,\rho) &= Y(\epsilon) 
- \left[ I - Z(\epsilon,\epsilon^2 \rr,\rho)\right] \cdot \epsilon^2 \rr \\
		&=  (\epsilon / \epsilon_0)^{2} \left[ \epsilon_0^2   
		  \tilde{Y}(\epsilon) - \epsilon_0^2 \rr 
		+  Z(\epsilon,\epsilon^2 \rr,\rho) \cdot \epsilon_0^2 \rr  \right] \\
		&\leq  (\epsilon / \epsilon_0)^{2} \left[ \epsilon_0^2  
		  \tilde{Y}(\epsilon_0)  - \epsilon_0^2 \rr 
   +  Z(\epsilon_0,\epsilon_0^2 \rr,\rho) \cdot \epsilon_0^2 \rr  \right] \\
		&= (\epsilon / \epsilon_0)^{ 2} P(\epsilon_0 , \epsilon_0^2 \rr,\rho) \\
		& < 0,
	\end{alignat*}
where inequalities are interpreted componentwise in $\R^3$, as usual.
\end{proof}

These $\epsilon$-rescaled variables are used in
Proposition~\ref{prop:TightEstimate} below to derive \emph{tight} bounds on the
solution (in particular, tight enough to conclude that the bifurcation is
supercritical). The following remark explains that the monotonicity properties of
the bounds $Y$ and $Z$ imply that looking for small(est) radii which satisfy $P(r)<0$, is
a well-defined problem.

\begin{remark}\label{r:smallradii}
The set $R$ of radii for which the radii polynomials are negative is given by 
\[
  R := \{ r \in \R^3_+ : r_j > 0,  P_i(r) < 0 \text{ for } i,j=1,2,3 \} .
\] 
This set has the property that if
	$r,r' \in R$, then $r''\in R$, where $r''_j=\min\{ r_j,r'_j\}$.
Namely, the main observation is that we can write 
	$P_i(r)= \tilde{P}_i(r)-r_i$, where $\partial_{r_j} \tilde{P}_i \geq 0$ for all $i,j=1,2,3$.
Now fix any $i$; we want to show that $P_i(r'') < 0$.	
We have either $r''_i=r_i$ or $r''_i=r'_i$, hence assume $r''_i=r_i$ (otherwise
just exchange the roles of $r$ and $r'$). We infer that $P_i(r'') \leq P_i(r) <
0$, since $\partial_{r_j} P_i \geq 0$ for $j \neq i$.
We conclude that there are no trade-offs in looking for minimal/tight radii, as
opposed to looking for large radii, see Remark~\ref{r:largeradii}.
\end{remark}


\begin{proposition}
		\label{prop:TightEstimate}
	Fix $ \epsilon_0 = 0.10$ and 
$ (\rr_\alpha , \rr_ \omega , \rr_c) = (  0.0594, \, 0.0260 , \, 0.4929 ) $ 
and 
$\rho = 0.3191$. 
	For all $0< \epsilon \leq \epsilon_0$ there exists a unique point $\hat{x}_\epsilon = (\hat{\alpha}_\epsilon,\hat{\omega}_\epsilon,\hat{c}_\epsilon)$ 
	satisfying $F(\hat{x}_\epsilon) = 0$ and 
	\begin{align}
	\label{eq:TightBound}
 | \hat{\alpha}_\epsilon - \balpha_\epsilon| <& \rr_\alpha \epsilon^2 , 
 &|\hat{\omega}_\epsilon - \bomega_\epsilon| <&  \rr_\omega \epsilon^2 ,
 &
 \| \hat{c}_\epsilon - \bc_\epsilon\| <& \rr_c  \epsilon^2   ,
  &
  \| K^{-1} \hat{c}_\epsilon \| <& \rho  .
	\end{align}
Furthermore, $\hat{\alpha}_\epsilon > \pp$ for $ 0 < \epsilon < \epsilon_0$.
\end{proposition}

\begin{proof}
	In the Mathematica file~\cite{mathematicafile}  we check, using interval arithmetic, that  $\rho \geq C(\epsilon_0, \epsilon_0^2 \rr)$ and  the radii polynomials $P(\epsilon_0,\epsilon_0^2 \rr,\rho)$ are negative.  
 The inequalities in Equation~\eqref{eq:TightBound} follow from Corollary~\ref{cor:RPUniformEpsilon}. 
 Since $\hat{\alpha}_\epsilon \geq \balpha_\epsilon - \rr_\alpha \epsilon^2
 = \pp +\frac{1}{5}(\frac{3\pi}{2}-1)\epsilon^2 - \rr_\alpha \epsilon^2$ and $ \rr_\alpha < \tfrac{1}{5} ( \tfrac{3 \pi}{2} -1) $, it follows that $ \hat{\alpha}_\epsilon > \pp $ for all $ 0 < \epsilon \leq \epsilon_0$. 
\end{proof}

\begin{remark}\label{r:nested}
Since $\epsilon_0^2\rr < r$ for the choices (a) and (b) in Proposition~\ref{prop:bigboxes},
and the choices of $\rho$ and $\epsilon_0$ are compatible as well, the solutions found in Proposition~\ref{prop:bigboxes} are the same as those described by Proposition~\ref{prop:TightEstimate}. While the former proposition provides large isolation/uniqueness neighborhoods for the solutions,
the latter provides tight bounds and confirms the  supercriticality of the bifurcation suggested in Definition \ref{def:xepsilon}.



\end{remark}

%
%
%

\section{Global results}
\label{s:global}

When deriving global results from the local results in
Section~\ref{s:local}, we need to take into account that there are some obvious
reasons why the branch of periodic solutions, described by
$F_\epsilon(\alpha,\omega,c)=0$, bifurcating from the Hopf bifurcation point at
$(\alpha,\omega)=(\pp,\pp)$ does not describe the entire set of periodic
solutions for $\alpha$ near $\pp$. First, there is the trivial solution. In
particular, one needs to quantify in what sense the trivial solution is an
isolated invariant set. This is taken care of by Remarks~\ref{r:smalleps}
and~\ref{r:cone}, which show there are no ``spurious'' small solutions in the
parameter regime of interest to us (roughly as long as we stay away from the
next Hopf bifurcation at $\alpha = \tfrac{5\pi}{2}$). Second, one can interpret any periodic
solution with frequency $\omega$ as a periodic solution with frequency
$\omega/N$ as well, for any $N \in \mathbb{N}$. Since we are working in Fourier space,
showing that there are no ``spurious'' solutions with lower frequency would
require us to perform an analysis near $(\alpha,\omega)=(\pp,\tfrac{\pi}{2N})$
for all $N \geq 2$. This obstacle can be avoided by bounding (from below)
$\omega$ away from $\pi/4$. This is done in Lemma~\ref{lem:omegalarge}.

For later use, we recall an elementary Fourier analysis bound. 
\begin{lemma}\label{lem:fourierbound}
	Let $y \in C^1$ be a periodic function of period $2\pi/\omega$ with Fourier coefficients $\c \in \ell^1_\sym$ (in particular this means $\c_0=0$), as described by~\eqref{eq:FourierEquation}. 
	Then 
\[
 \| \c \| \leq \sqrt{\frac{\pi }{6 \omega}}\,  \| y' \|_{L^2([0,2\pi/\omega])}
\qquad\text{and}\qquad
 \| \c \| \leq \frac{\pi}{\omega\sqrt{3}}\, \|y'\|_\infty.
 \]
\end{lemma}
\begin{proof}
From the Cauchy-Schwarz inequality and  Parseval's identity it follows that
\begin{alignat*}{1}
		\| \c \| &= 2 \sum_{k=1}^{\infty} |\c_k|
		\leq 2 \left( \sum_{k=1}^{\infty} k^{-2} \right)^{1/2}
		\left( \sum_{k=1}^{\infty} |k \, \c_k|^2 \right)^{1/2} \\
     &=  \frac{\sqrt{2}}{\omega} \left(\frac{\pi^2}{6} \right)^{1/2} 
	 	 \left(2 \sum_{k=1}^{\infty} |i \omega k \, \c_k|^2 \right)^{1/2}
		 = \frac{\pi}{\omega \sqrt{3}} 
		\left(\sum_{k \in \Z} |i \omega k \, \c_k|^2 \right)^{1/2}\\
		&= \frac{\pi}{\omega \sqrt{3}} 
		\left( \frac{\omega}{2\pi} \int_0^{2\pi/\omega} | y'(t)|^2 dt  \right)^{1/2} 
		\leq \frac{\pi}{\omega \sqrt{3}} \,  \|y'\|_\infty.
\end{alignat*}	
\end{proof}


\subsection{A proof of Wright's conjecture}

Based on the work in \cite{neumaier2014global} and \cite{wright1955non}, in order to prove Wright's conjecture it suffices to prove that there are no slowly oscillating periodic solutions (SOPS) to Wright's equation for $ \alpha \in [1.5706,\pp]$. Moreover, in \cite{neumaier2014global} it was shown that no SOPS with $\| y \|_\infty \geq e^{0.04}-1$ exists for  $\alpha \in [1.5706,\pp]$. These results are summarized in the following proposition.

\begin{proposition}[\cite{neumaier2014global,wright1955non}]
\label{prop:neumaier}
Assume $y$ is a SOPS to Wright's equation for some $\alpha \leq \pp$. Then $\alpha \in [1.5706,\pp]$
and $\| y \|_\infty \leq e^{0.04}-1$. 
\end{proposition}
For convenience we introduce
\[
  \mu := e^{0.04}-1 \approx 0.0408.
\]
We now derive a lower bound on the frequency $\omega$ of the SOPS.
\begin{lemma}\label{lem:omegalarge}
Let $\alpha \in [1.5706,\pp]$.
Assume $y$ is a SOPS to Wright's equation with minimal period $2\pi/\omega$,
and assume that $\| y \|_\infty \leq \mu$.
Then $\omega \in [1.11,1.93]$.
\end{lemma}

\begin{proof}
	Without loss of generality, we assume in this proof that $ y(0) =0$, that $y(t) < 0$ for $t\in (-t_{-},0)$ and that $y(t) > 0$ for $t\in (0,t_+)$. 
	We will show that $t_-$ and $t_+$ are bounded by
	\begin{alignat*}{1}
	1+ \frac{1}{\alpha }  \frac{\log (1 + \mu)}{\mu}  < t_+ &<2 + \frac{1}{\alpha} , \\
	1+\frac{1}{\alpha} <t_- & < 3 .
	\end{alignat*}
	The lower bounds for both $t_-$ and $t_+$ follow directly from  Theorem 3.5 in \cite{jones1962nonlinear}. While Theorem 3.5 in~\cite{jones1962nonlinear} assumes $ \alpha \geq \pp$, this part of the theorem simply relies on Lemma 2.1 in \cite{jones1962nonlinear}, which only requires $ \alpha > e^{-1}$.

To obtain an upper bound on $t_+$, assume that $ t_+ \geq 2$. Set $t'_+ =\min\{t_+,3\}$. Then it follows from~\eqref{eq:Wright} that $y'(t) < 0$ for $t\in (1,t'_+]$, hence  $y(t-1) > y(2)$ for $ t \in [2,t'_+]$. We infer that for $t \in [2,t'+]$ we have 
$y'(t) = - \alpha y(t-1) [1+y(t)] < - \alpha y(2)$.
Solving the IVP $y'(t) < -\alpha y(2)$ with the initial condition $y(2) = y(2)$, we see that $y(t)$ hits $0$ before $t=2+\frac{1}{\alpha}$. Since $\alpha > 1$ (hence $2+\frac{1}{\alpha} < 3$), this implies that $t'_+=t_+$ and $t_+ < 2+\frac{1}{\alpha}$.

	To obtain the upper bound on $t_-$, assume for the sake of contradiction
	 that $ t_- \geq 3$. 
	 Then it follows from~\eqref{eq:Wright} that $y'(t) \geq 0$ for $t\in [-2,0]$, hence  $y(t) \leq y(-1)$ for $ t \in [-2,-1]$, and  $y'(t) \geq - \alpha y(-1) [1+y(t)]$ for $ t \in [-1,0]$.  
	Solving this IVP with the initial condition $y(-1) = -\nu$, we obtain $ y(t) \geq (1-\nu) e^{ \alpha \nu (t+1)}-1 $ for $ t \in [-1,0]$, and in particular  $y(0) \geq (1-\nu ) e^{-\alpha \nu}-1$. 
	By assumption $y(0)=0$ and $\nu=|y(-1)| \leq \mu$,
	but $  (1-\nu ) e^{-\alpha \nu}-1>0 $ 
	for $ \nu  \in (0,\mu]$
	and $\alpha \in [1.5706,\pp]$, a contradiction. Thereby $ t_- <3$.

The bound on $\alpha$ implies that 
 the minimal period $L = t_+ + t_-$ of the SOPS must lie in $[ 3.26,5.64]$.
It then follows that $ \omega \in [1.11,1.93]$	
\end{proof}

%

It turns out that this bound on $\omega$ can (and needs to be) sharpened.
This is the purpose of the following lemma, 
which considers solutions in  unscaled variables. 
\begin{lemma}\label{lem:ZeroNBD}
Suppose $ \tilde{F}_\epsilon(\alpha,\omega,\tc)=0$. If $\omega \in
[1.1,2]$ and $ \alpha \in [1.5,2.0]$  
then
\begin{equation}\label{e:tighterboundonomega}
   \frac{\sqrt{(\omega- \alpha)^2 + 2 \alpha \omega(1-\sin\omega)}}{2\alpha} 
   \leq 2 \epsilon + \| \tc \| .
\end{equation}
\end{lemma}
\begin{proof}
This follows from Proposition~\ref{prop:zeroneighborhood2} in
Appendix~\ref{appendix:aprioribounds}, combined with Proposition \ref{prop:G1Minimizer},  which shows that for
$\omega \in [1.1,2.0]$  and $ \alpha \in [1.5,2.0]$, the minimum in Equation~\eqref{e:minoverk} is attained for $k=1$.
\end{proof}

Next we derive bounds on $\epsilon$ and $\tc$, which also lead to improved bounds on $\omega$.
\begin{lemma}\label{lem:wrightbounds}
Let $\alpha \in [1.5706,\pp]$. Assume $y$ is a SOPS with $\| y \|_\infty \leq \mu$.
Then $y$ corresponds, through the Fourier representation~\eqref{e:yc}, to a zero of $F_\epsilon(\alpha,\omega,c)$ with $|\omega- \pp| \leq 0.1489$ and
\[
  0< \epsilon \leq \epseps := \mu/\sqrt{2} \leq 0.02886 ,
\] 
and 
$\| c \| \leq 0.0796$ 
and 
$\| K^{-1} c \| \leq  0.16 $.
\end{lemma}
\begin{proof}
First consider the Fourier representation~\eqref{e:ytc} of $y$ in unscaled variables. 
Recall that $\c_0$ vanishes (see Remark~\ref{r:a0}).
Since $|y'(t)| \leq \alpha |y(t-1)| (1+|y(t)|) \leq \alpha \mu (1+\mu)$
we see from  Lemma~\ref{lem:fourierbound} that 
\begin{equation}\label{e:epsilontc} 
  2 \epsilon + \| \tc \|  \leq \frac{\pi}{\omega\sqrt{3}} \alpha \mu (1+\mu).
\end{equation}
Combining this with Lemma~\ref{lem:ZeroNBD} leads to the inequality
\begin{equation}
\label{eq:APomegaBound}
	\omega 	\sqrt{(\omega- \alpha)^2 + 2 \alpha \omega ( 1- \sin \omega)} 
	\leq 
	\tfrac{2 \pi}{ \sqrt{3}}  \alpha^2 \mu ( 1 + \mu).
\end{equation}
In the Mathematica file \cite{mathematicafile} we show that when $ \alpha \in [ 1.5706,\pp]$, then inequality \eqref{eq:APomegaBound} is violated for any $\omega \in [ 1.1,2.0] \, \backslash \, [1.4219, 1.6887]$.
From Lemma~\ref{lem:omegalarge} we obtain the a priori  bound
$\omega \in [1.11,1.93]$, whereby it follows that  $\omega \in [1.4219,1.6887]$, and in particular  $|\omega - \pp| \leq 0.1489$.

Using this sharper bound on $\omega$ as well as $\alpha \in [1.5706,\pp]$
we conclude from~\eqref{e:epsilontc} that
\begin{equation}
	2 \epsilon + \| \tc \|  \leq \frac{\pi}{\omega\sqrt{3}} \alpha \mu (1+\mu)
	\leq \frac{2\omega - \alpha}{\alpha}.
\end{equation} 
Since we also infer that $\alpha < 2\omega$,  Theorem~\ref{thm:FourierEquivalence3}(b) shows that the solution corresponds to a zero of $F_\epsilon(\alpha,\omega,c)$, with $\tc = \epsilon c $.
We can improve the bound on $\epsilon$ from~\eqref{e:epsilontc}
by observing that
\[
	( \epsilon^2 + \epsilon^2)^{1/2} 
	\leq \left( \sum_{k\in\Z} |\c_k|^2 \right)^{1/2}
	 =  \left( \frac{\omega}{2\pi} \int_0^{2\pi/\omega}
	                         |y|^2 dt  \right)^{1/2} \leq \mu .
\]
Hence $\epsilon \leq \epseps := \mu/\sqrt{2}$.

Finally, we derive the bounds on $c$. Namely, for $\alpha \in  [1.5706,\pp]$,
$\omega \in [1.4219,1.6887]$ and $\epsilon \leq \epseps$,
we find that $b_*$ and  $z_*^+$, as defined in~\eqref{e:zstar}, are bounded below by $ b_* \geq 0.364$ and 
$z^+_* \geq 0.72 $. 
Since it follows from~\eqref{e:epsilontc} that 
$ \| \tc \|  \leq 0.09 $  
in the same parameter range  of  $\alpha$ and $\omega$, we infer from Lemma~\ref{lem:Cone}(a) that 
$\| \tc \| \leq z^-_* $.
Via an interval arithmetic computation, the latter can be bounded above
using Lemma~\ref{lem:ZminusBound}, for $\alpha \in  [1.5706,\pp]$,
$\omega \in [1.4219,1.6887]$ and $\epsilon \leq \epseps$, by
$z_*^- \leq 0.0796 \epsilon$. 
Hence 
$\| c \| \leq z_*^- / \epsilon \leq 0.0796$.
Furthermore, Lemma~\ref{lem:Cone}(b) implies 
the  bound 
$\| K^{-1} c \| \leq (2\epsilon^2 + (z_*^-)^2 )/(\epsilon b_*) \leq 5.52 \epsilon$.
Since $\epsilon \leq \epsilon_*$, it then follows that 
$\| K^{-1} c \| \leq  0.16 $.
\end{proof}

With these tight bounds on the solutions, we are in a position to apply the local bifurcation result formulated in Proposition~\ref{prop:TightEstimate} to prove the ultimate step of Wright's conjecture.

\begin{theorem}
	\label{thm:WrightConjecture}
	For $ \alpha \in [0,\pp]$ there is no SOPS to Wright's equation.
\end{theorem}
\begin{proof}
By Proposition~\ref{prop:neumaier} (see also the introduction of this paper) it suffices to exclude a slowly oscillating solution $y$ for $\alpha \in [1.5706,\pp]$ with $\| y \|_\infty \leq \mu$.
By Lemma~\ref{lem:wrightbounds}, if such a solution would exist, it corresponds to a solution of 
$F_\epsilon(\alpha,\omega,c)=0$ with $|\omega- \pp| \leq 0.1489$, 
$0< \epsilon \leq \epseps = \mu / \sqrt{2}$, 
$\| c \| \leq 0.0796 $ 
and 
$\| K^{-1} c \| \leq  0.16 $.
We claim that no such solution exists.
Indeed, we define the set 
\[
S :=  \{ (\alpha,\omega,c) \in X : 
|\alpha - \pp| \leq 0.0002; \,
|\omega - \pp| \leq 0.15; \, 
\| c \| \leq 0.08; \, 
\|K^{-1} c \| \leq 0.16  \}.
\]
To show that there is no SOPS for $\alpha \in [1.5706,\pp]$, it now suffices to show that all zeros of $F_\epsilon(\alpha,\omega,c)$ in $S$ for any $0< \epsilon \leq \epseps$ satisfy $\alpha> \pp$.

Let us consider
$B_\epsilon(r,\rho)$, which is centered at $\bx_\epsilon$ (see Definition~\ref{def:xepsilon})
with $r$ and $ \rho$ taken as in Proposition~\ref{prop:bigboxes}(a).
In the Mathematica file~\cite{mathematicafile} 
we check that the following inequalities are satisfied:
\begin{alignat*}{1}
r_\alpha &= 0.13 \geq  0.0002 + |\balpha_{\epseps}-\pp|,\\
r_\omega &= 0.17 \geq  0.15 + |\bomega_{\epseps}-\pp| ,\\
r_c &= 0.17 \geq   0.08 + \| \bc_{\epseps}\|,\\
\rho &= 1.78 \geq 0.16 . 
\end{alignat*}
By the triangle inequality we obtain that $S \subset B_\epsilon(r,\rho)$ for all $0<\epsilon\leq \epseps$.
Proposition~\ref{prop:bigboxes}(a) shows that for each $0<\epsilon\leq \epseps$
there is a unique zero $\hat{x}_\epsilon=
(\hat{\alpha}_\epsilon,\hat{\omega}_\epsilon,\hat{c}_\epsilon) \in B_\epsilon (r,\rho)$ of $F_\epsilon$.
By Proposition~\ref{prop:TightEstimate} and Remark~\ref{r:nested}
this zero satisfies $\hat{\alpha}_\epsilon > \pp$.
Hence, for any $0<\epsilon\leq\epseps$ the only zero of $F_\epsilon$ in $S$ (if there is one) satisfies $\alpha>\pp$. This completes the proof.

\end{proof}

\subsection{Towards Jones' conjecture}
\label{s:Jones}

Jones' conjecture states that for $ \alpha > \pp$ there exists a (globally) unique SOPS to Wright's equation.
 Theorem \ref{thm:RadPoly} shows that for a fixed small $\epsilon$ there is a (locally) unique $\alpha$ at which Wright's equation has a SOPS, represented by
 $(\hat{\alpha}_\epsilon,\hat{\omega}_\epsilon,\hat{c}_\epsilon)$. 
This is not sufficient to prove the local case of Jones conjecture.
To accomplish the latter, we show in Theorem \ref{thm:UniqunessNbd} 
 that near the bifurcation point there is, for each fixed $\alpha>\pp$, a (locally) unique SOPS to Wright's equation. 
We begin by showing that on the solution branch emanating from the Hopf bifurcation 
$\hat{\alpha}_\epsilon$ is monotonically increasing in~$\epsilon$,
i.e.\ $ \tfrac{d}{d \epsilon} \hat{\alpha}_{\epsilon} >0$.  
Since $\balpha_\epsilon = \pp + \tfrac{1}{5}(\tfrac{3 \pi}{2}-1) \epsilon^2  $,
we expect that $\tfrac{d}{d \epsilon} \hat{\alpha}(\epsilon) = \tfrac{2}{5}(\tfrac{3 \pi}{2}-1) \epsilon + \cO(\epsilon^2)$. 
For this reason it is essential that we calculate an approximation of $\tfrac{d}{d \epsilon} \hat{\alpha}_\epsilon$ which is accurate up to order $ \cO(\epsilon^2)$.   

\begin{theorem}
	\label{thm:NoFold}
For $0 < \epsilon \leq 0.1$ we have $ \tfrac{d}{d \epsilon} \hat{\alpha}_{\epsilon} >0$. 
For 
$\pp  < \alpha \leq \pp + 6.830  \times 10^{-3}$ 
there are no  bifurcations in the branch of SOPS that originates from the Hopf bifurcation. 
\end{theorem}

%

\begin{proof}
We show that the branch of solutions  $ \hat{x}_\epsilon =  (\hat{\alpha}_\epsilon , \hat{\omega}_\epsilon , \hat{c}_\epsilon)$ obtained in Proposition \ref{prop:TightEstimate} satisfies $  \tfrac{d}{d \epsilon} \hat{\alpha}_\epsilon >0$ for $0<\epsilon \leq 0.1$.
This implies that the solution branch is (smoothly) parametrized by~$\alpha$,
i.e.,  there are no secondary nor any saddle-node bifurcations in this branch.
We then show that these $\epsilon$-values cover the range 
$\pp  < \alpha \leq \pp + 6.830  \times 10^{-3}$.
	
We begin by
	differentiating the equation $ F( \hat{x}_\epsilon) =0$ with respect to $ \epsilon$:
%
\begin{equation}
 \frac{\partial F}{\partial  \epsilon}(\hat{x}_\epsilon) + D F( \hat{x}_\epsilon)  \frac{d }{d  \epsilon} \hat{x}_\epsilon  = 0 .
\end{equation}
In terms of the map $T$ we obtain the relation
\[
\left[I-DT(\hat{x}_\epsilon)  \right]  \frac{d }{d \epsilon} \hat{x}_\epsilon   
=- A^{\dagger} \frac{\partial F}{\partial  \epsilon}(\hat{x}_\epsilon)  .
\]

To isolate $\frac{d }{d \epsilon} \hat{x}_\epsilon   $, we wish to left-multiply each side of the above equation by $[I-DT(\hat{x}_\epsilon)]^{-1}$. 
To that end, we define an upper bound on $DT(\hat{x}_\epsilon)$ by  the matrix 
\begin{equation}\label{e:defZeps}
	\ZZ_\epsilon := Z(\epsilon,\epsilon^2 \rr, \rho) ,
\end{equation}
with $\rr$ and $\rho$ as in Proposition~\ref{prop:TightEstimate}.
We know from Remark~\ref{r:boundDT} 
that with respect to the norm $\| \cdot \|_{\rr}$ on $\R^2 \times \ell^K_0$
\[
\| DT(\hat{x}_\epsilon)  \|_{\rr} \leq \max_{i=1,2,3} \frac{( \ZZ_\epsilon \cdot \rr)_i}{\rr_i} < 1, \qquad\text{for all } 0 \leq \epsilon \leq \epsilon_0, 
\]
with $\epsilon_0$ given in Proposition~\ref{prop:TightEstimate}. 
Hence $I-DT(\hat{x}_\epsilon) $ is invertible. In particular,
\begin{alignat*}{1}
\frac{d }{d \epsilon} \hat{x}_\epsilon   
& =- \left[I-DT(\hat{x}_\epsilon)  \right]^{-1}  A^{\dagger} \frac{\partial F}{\partial  \epsilon}(\hat{x}_\epsilon)  \\
& = - \left[I + \sum_{n=1}^\infty DT(\hat{x}_\epsilon)^n  \right]  A^{\dagger} \frac{\partial F}{\partial  \epsilon}(\hat{x}_\epsilon) .
\end{alignat*}
We  have an upper bound $\QQ_\epsilon \in \R^3_+$ on $A^{\dagger} \frac{\partial F}{\partial  \epsilon}(\hat{x}_\epsilon)$, as defined in Definition~\ref{def:upperbound}, given by Lemma~\ref{lem:Qeps}. 
We define $\II$ to be the $3 \times 3$ identity matrix.
For the $\alpha$-component we then obtain the estimate
\begin{alignat}{1}
\frac{d }{d \epsilon} \hat{\alpha}_\epsilon  
&\geq - \pi_\alpha  A^{\dagger} \frac{\partial F}{\partial  \epsilon}(\hat{x}_\epsilon)  
- \left( \sum_{n=1}^\infty \ZZ_\epsilon^n \QQ_\epsilon \right)_1 \nonumber \\
& = - \pi_\alpha  A^{\dagger} \frac{\partial F}{\partial  \epsilon}(\hat{x}_\epsilon)  - \left( \ZZ_\epsilon (\II-\ZZ_\epsilon)^{-1} \QQ_\epsilon \right)_1 . \label{e:alphaepsilon}
\end{alignat}
We approximate $\frac{\partial F}{\partial  \epsilon}(\hat{x}_\epsilon)$ by 
\[
	\Gamma_\epsilon := \pp \tfrac{3i -1}{5} \epsilon \, \e_1 - i \pp \,  \e_2 - \pp \tfrac{3+i}{5} \epsilon \, \e_3 ,
\]
which is accurate up to quadratic terms in $\epsilon$.
In Lemma \ref{lem:ImplicitApprox} it is shown that
\begin{equation}\label{e:linearepsilon}
- \pi_\alpha A^{\dagger} \Gamma _\epsilon = \tfrac{2}{5} ( \tfrac{3 \pi}{2} -1) \epsilon.
\end{equation}
It remains to incorporate two explicit bounds for the remaining terms in~\eqref{e:alphaepsilon}. 
In Lemma~\ref{lem:Meps} we define $M_\epsilon$ and $M'_\epsilon$ that satisfy the following inequalities:
\begin{alignat}{1}
\left| \pi_\alpha A^{\dagger} \left( \tfrac{\partial F}{\partial  \epsilon}(\hat{x}_\epsilon) - \Gamma_\epsilon \right)  \right| &\leq 
\epsilon^2 M_\epsilon  , \label{e:boundM} \\
\left( \ZZ_\epsilon (\II-\ZZ_\epsilon)^{-1} \QQ_\epsilon \right)_1 &\leq 
\epsilon^2 M'_\epsilon . \label{e:boundMp} 
\end{alignat}
Moreover, we infer from Lemma~\ref{lem:Meps} that $M_\epsilon$ and $M'_\epsilon$ are positive, increasing in $\epsilon$, and can be 
obtained explicitly by performing an interval arithmetic computation, using the explicit expressions for the matrix $\ZZ_\epsilon$ and the vector $\QQ_\epsilon$ given by Equation~\eqref{e:defZeps} and Lemma~\ref{lem:Qeps}, respectively (the expression for $Z(\epsilon,r,\rho)$ is provided in Appendix~\ref{sec:BoundingFunctions}).
 
Finally, we combine~\eqref{e:alphaepsilon},~\eqref{e:linearepsilon},~\eqref{e:boundM} and~\eqref{e:boundMp} to obtain
\[
 \frac{d }{d \epsilon} \hat{\alpha}_\epsilon  \geq 
 \tfrac{2}{5} ( \tfrac{3 \pi}{2} -1) \epsilon - \epsilon^2 ( M_{\epsilon} + M'_{\epsilon}).
\]
From the monotonicity of the bounds $M_\epsilon$ and $M'_\epsilon$ in terms of $\epsilon$, we infer that in order to conclude that  $\frac{d }{d \epsilon} \hat{\alpha}_\epsilon >0 $ for $0<\epsilon\leq\epsilon_0$  it suffices to check, using interval arithmetic, that
\begin{equation}
 \tfrac{2}{5} ( \tfrac{3 \pi}{2} -1) \epsilon_0 - \epsilon_0^2 (M_{\epsilon_0} + M'_{\epsilon_0})  > 0 . \label{e:Mepsilon0}
\end{equation} 
In the Mathematica file~\cite{mathematicafile} we check that~\eqref{e:Mepsilon0} is satisfied 
for $\epsilon_0 = 0.1$.
Since $\balpha_{\epsilon_0} \geq \pp + 7.4247\times 10^{-3}$,
and taking into account the control provided by Proposition~\ref{prop:TightEstimate} on the distance between $\hat{\alpha}_\epsilon$ and $\balpha_\epsilon$ in terms of $\rr_\alpha$, we
find that  
 $\hat{\alpha}_{\epsilon_0} \geq \balpha_{\epsilon_0} - \epsilon_0^2 \rr_\alpha \geq \pp + 6.830  \times 10^{-3}$.
Hence  there can be no bifurcation on the solution branch for 
 $ \pp < \alpha \leq \pp + 6.830   \times 10^{-3}$.
\end{proof}

The above theorem provides the missing piece in the proof of Theorem~\ref{thm:IntroNoFold}.

\begin{corollary}\label{cor:collectreformulatedJones}
The branch of SOPS originating from the Hopf bifurcation at $\alpha = \pp$ has no folds or secondary bifurcations for any $\alpha > \pp$. 
\end{corollary}
\begin{proof}
	
	We prove the corollary by combining results on four overlapping subintervals of $ ( \pp, \infty)$. 
	In Theorem~\ref{thm:NoFold} we show that the (continuous) branch of SOPS originating from the Hopf bifurcation does not have any folds or secondary bifurcations for 
		$ \alpha \in ( \pp  , \pp + \delta_3] $ where 
 $ \delta_3 = 6.830  \times 10^{-3}$. 
	In~\cite{lessard2010recent} the same result is proved for $ \alpha \in [ \pp + \delta_1, 2.3]$, where $\delta_1 = 7.3165 \times 10^{-4}$. 
	In~\cite{jlm2016Floquet} 
	 it is shown that there is a unique SOPS  for $ \alpha $ in the interval $[1.94,6.00]$. 
	 Since $1.94 \leq 2.3$, then the SOPS in this interval belong to the branch originating from the Hopf bifurcation, and since they are unique for each $\alpha$, the branch is continuous and cannot have any folds or secondary bifurcations. 
	  In~\cite{xie1991thesis} it is shown that there is a unique SOPS for $ \alpha $ in the interval $ [5.67, +\infty)$, and by a similar argument the branch of SOPS cannot have any folds or secondary bifurcations in this interval either. 
	 Since 
	\[
	(\pp, \infty) = (\pp, \pp + \delta_3] \cup [ \pp + \delta_1,2.3] \cup [1.94,6.00] \cup [5.67, \infty) ,
	\]
	it follows that branch of SOPS originating from the Hopf bifurcation at $\alpha = \pp$ has no folds or secondary bifurcations for any $\alpha > \pp$. 
\end{proof}

To prove Jones' conjecture, it is insufficient to prove only locally that Wright's equation has a unique SOPS.  
We must be able to connect our local results with global estimates. 
When we make the change of variables $\tc = \epsilon c$ in defining the function $F_\epsilon$, we restrict ourselves to proving local results.  
Theorems~\ref{thm:UniqunessNbd} and~\ref{thm:UniqunessNbd2} connect these local results with a global argument, and construct neighborhoods, independent of any $ \epsilon$-scaling, within which the only SOPS to Wright's equation are those originating from the Hopf bifurcation.  

The next theorem uses the large radius calculation from Proposition \ref{prop:bigboxes}(b) to show that for  
$\alpha \in ( \pp , \pp+ 5.53 \times 10^{-3} ]$
all periodic solutions in a neighborhood of $0$ lie on the Hopf bifurcation curve, which has neither folds nor secondary bifurcations.  


\begin{theorem}
	\label{thm:UniqunessNbd}
	For each $\alpha  \in  (\pp , \pp + 5.53 \times 10^{-3} ] $ there is a unique triple $ ( \epsilon, \omega, c)$ in the range 
$ 0 < \epsilon \leq 0.09$
	and 
$ | \omega - \pp| < 0.0924 $
	and  
$ \| c \| \leq 0.30232 $ 
such that $ F_\epsilon(\alpha, \omega, c)=0$. 	
\end{theorem}
 
\begin{proof}
Fix $ \alpha \in ( \pp , \pp + 5.53 \times 10^{-3}]$ and let $F_\epsilon(\alpha, \omega, c)=0$ for some $\epsilon, \omega, c$ satisfying the assumed bounds. 
From Lemma~\ref{lem:Cone}(b) it follows that 
$\|K^{-1}c\| \leq \epsilon^2 (2+ \|c\|^2) /(\epsilon b_*) \leq 0.61$
for $\epsilon \leq \epsilon_0$, 
since $b_* \geq 0.31$. 
Hence the zeros under consideration all lie in the set 
\[
\tS :=  \{ (\alpha,\omega,c) \in X : | \alpha - \pp | \leq 0.00553 , |\omega - \pp| \leq 0.0924, \| c \| \leq 0.30232, \|K^{-1} c \| \leq 0.61  \}.
\]
Proposition~\ref{prop:bigboxes}(b) shows that for each $0\leq\epsilon\leq 0.09$
there is a unique zero $\hat{x}_\epsilon=
(\hat{\alpha}_\epsilon,\hat{\omega}_\epsilon,\hat{c}_\epsilon) \in B_\epsilon(r,\rho)$ of $F_\epsilon$,
with $r=(r_\alpha,r_\omega,r_c) = (0.1753,0.0941,0.3829)$ and $\rho= 1.5940$.
For each $0 \leq \epsilon \leq 0.09$ it follows from the triangle inequality  that $\tS \subset B_\epsilon(r,\rho)$.  
This shows that $F_\epsilon$ has at most one zero in $\tS$ for each $ 0 \leq \epsilon \leq \epsilon_0$. 
By Remark~\ref{r:nested} this solution lies on the branch $\hat{x}_\epsilon$ originating from the Hopf bifurcation, in particular $\hat{x}_0=(\pp,\pp,0) \in \tS$.
Proposition~\ref{prop:TightEstimate} gives us tight bounds 
\[
|\hat{\omega}_\epsilon - \pp| \leq |\bomega_\epsilon - \pp|  + \rr_\omega \epsilon^2 \leq 0.0924
\qquad\text{and}\qquad \| \hat{c}_\epsilon \| \leq \| \bc_\epsilon\|  + \rr_c \epsilon^2 \leq 0.30232
\]
for all $0 \leq \epsilon \leq \epsilon_0$.  
Moreover, from similar considerations it follows that $\hat{\alpha}_{\epsilon_0} \geq  \balpha_{\epsilon_0} - r_\alpha \epsilon_0^2 > 0.00553$. Hence $\hat{x}_{\epsilon_0} \notin \tS$ and the solution curve leaves $\tS$ through $|\alpha- \pp| = 0.00553$ for some $0<\epsilon <\epsilon_0$.
Since $0.00553  < 6.830 \times 10^{-3}$ the assertion now follows directly from Theorem~\ref{thm:NoFold}.
\end{proof}

Finally, we translate this result to function space.

\begin{theorem}
For each 
$\alpha  \in  (\pp , \pp + 5.53 \times 10^{-3} ] $ 
there is at most one (up to time translation) periodic solution to Wright's equation satisfying 
$ \| y' \|_{L^2([0,2\pi/\omega])} \leq  0.302$
	and having frequency 
$ | \omega - \pp | \leq 0.0924$. 
	\label{thm:UniqunessNbd2}
\end{theorem}

\begin{proof}
We show that any periodic solution~$y$ to Wright's equation of period $2\pi/\omega$ that satisfies 
$ \| y' \|_{L^2} \leq 0.302$ 
has Fourier coefficients satisfying the bounds in Theorem~\ref{thm:UniqunessNbd}.
For the Fourier coefficients $a$ of $y$ we infer from Lemma~\ref{lem:fourierbound} that 
$\| a \| \leq \sqrt{\frac{\pi}{6\omega}} \cdot 0.302 \leq 0.18 $. 
Furthermore, for the parameter range of $\alpha$ and $\omega$ under consideration we conclude that $\alpha < 2\omega $ and 
$\|a\| < \frac{2\omega-\alpha}{\alpha}$. Hence we see from Theorem~\ref{thm:FourierEquivalence3}
that $y$ corresponds to a zero of $F_\epsilon$. 
The a priori bound on $\|a\|$ translates via~\eqref{e:aepsc} into the bounds
\[
\epsilon \leq 0.09
\qquad\text{and}\qquad
\| \tc \| \leq 0.18 .
\]
We now derive further bounds on $c=\tc/\epsilon$, 
 as in the proof of Lemma~\ref{lem:wrightbounds}.
Namely, for 
	$|\alpha-\pp| \leq 0.00553$,
	$|\omega-\pp| \leq 0.0924$ and  $\epsilon \leq 0.09$,
we find that  $z_*^+$, as defined in~\eqref{e:zstar}, is bounded below by 
 $z^+_* \geq 0.595$. 
It follows that 
$ \| \tc \|  \leq 0.18 \leq z^+_*$, 
so we infer from Lemma~\ref{lem:Cone}(a) that 
$\| \tc \| \leq z^-_* $.
Via Lemma~\ref{lem:ZminusBound} and an interval arithmetic computation, the latter can be bounded above, for 
	 $|\alpha-\pp| \leq 0.00553$,
	$|\omega-\pp| \leq 0.0924$ and  $\epsilon \leq 0.09$, by
	$z_*^- \leq 0.30226 \epsilon$. 
Hence 
 $\| c \| \leq z_*^- / \epsilon \leq 0.30232$.
We conclude that $y$ corresponds to a zero of $F_\epsilon(\alpha,\omega,c)$ in the parameter set described by Theorem~\ref{thm:UniqunessNbd}, which implies uniqueness.
\end{proof}

\bibliographystyle{abbrv}
\bibliography{BibWright}


\appendix



\section{Appendix: Operator Norms}
\label{sec:OperatorNorms}
We set $\omega_0 = \pp$ and recall that 
\begin{alignat*}{1}
	[U_\omega a]_k & =  e^{-i k\omega} a_k \\
	[U_{\omega_0} a]_k & = (-i)^k a_k \\
	L_{\omega}  & =  \sigma^+( e^{-i\omega}  I + U_{\omega}) + \sigma^-( e^{i\omega} I + U_{\omega})  \\
	L_{\omega_0} & = \sigma^+( -i  I + U_{\omega_0}) + \sigma^-( i I + U_{\omega_0})  .
\end{alignat*}
To more efficiently express the inverse of $ A_{0,*}$ we define an operator $\hat{U}: \ell^1_0 \to \ell^1_0 $ by
%
\begin{equation}\label{e:defUhat}
[\hat{U} c]_{k\geq 2} := (1 - i k^{-1}e^{-i k \pi /2} )^{-1} c_k,
\end{equation}
so that  
$ A_{0,*}^{-1}=  \frac{2}{ i \pi } \hat{U} K $.
%
%

%

The operator norm of  $Q \in B(\ell^1_0,\ell^1)$ can be expressed using the basis elements $\e_k$ (which have norm $\|\e_k\|=2$):
\begin{equation}\label{e:operatornorm}
  \| Q \| = \frac{1}{2} \sup_{k \geq 2} \|Q \e_k\| .
\end{equation}
Some of the operators in $B(\ell^1_0,\ell^1)$ considered in these appendices restrict naturally to $B(\ell^1_0)$, with the same expression for the norm. For operators in $B(\ell^1)$ a similar expression for the norm holds (the supremum being over $k\geq 1$). We will abuse the notation $\|Q\|$ by not indicating explicitly which of these operator norms is considered; this will always be clear from the context.
\begin{proposition}\label{p:severalnorms}
	The operators $\hat{U}, \hat{U} K, L_{\omega}, A_{0,*}^{-1}   $ and $A_{1,*}$ in $B(\ell^1_0,\ell^1)$  satisfy the bounds
\begin{align*}
\| \hat{U} \| 		=& \tfrac{5}{4} 						&\| A_{0,*}^{-1} \| =& \tfrac{2}{ \pi \sqrt{5}}	\\ 
\| \hat{U} K \| 	=& \tfrac{1}{ \sqrt{5}}	&\| A_{1,*} \| \leq& 2 \pi	 \\
\| L_{\omega} \| \leq& 4
\end{align*}
\end{proposition}
\begin{proof}
		The value $\| \hat{U} \e_k \|$ is maximized when $k=5$, whence  $\| \hat{U} \| = 5/4$. 
		The value $\| \hat{U} K \e_k \|$ is maximized when $k=2$, whence $\| \hat{U} K \| 	= \frac{1}{ \sqrt{5}}$ and $\| A_{0,*}^{-1}\| = \frac{2}{\pi \sqrt{5}} $. 
It follows from the definition of $L_\omega$ and the fact that $U_\omega$ is unitary that $ \| L_{\omega} \| \leq 4$, whereby it follows that  
$ \| A_{1,*} \| = \| \pp L_{\omega_0} \|  \leq 2 \pi$.
\end{proof}

We recall, for any $a\in \ell^1$, the splitting $a=a_1 \e_1 + \tilde{a}$ with $a_1 \in \C$ and $\tilde{a} \in \ell^1_0 $,  and as a tool in the estimates below we  introduce the projections 
\begin{alignat}{1}
	\pi_1 a &= a_1  \in \C \\
	\pi_{\geq 2} a & = \tilde{a}. \label{e:pige2}
\end{alignat}

\begin{proposition}
	\label{prop:A1A0}
	We have for the map $ A_1 A_0^{-1} : \ell^1 \to \ell^1$ that 
	\begin{equation}
	\label{eq:A1A0}	
	\| A_1 A_0^{-1}\| = \frac{2 \sqrt{10}}{5} \, .
	\end{equation}
\end{proposition}

\begin{proof}
Expanding $A_1A_{0}^{-1}$ we see that it splits into two parts: $A_{1,2} A_{0,1}^{-1}$ and $A_{1,*}     A_{0,*}^{-1}$, which we estimate separately. To be precise
\[
  A_1A_{0}^{-1} a = (i_\C A_{1,2} A_{0,1} i_\C^{-1} \pi_1 a) \e_2 
                    +  A_{1,*} A_{0,*}^{-1} \pi_{\ge 2} a.
\]
First, we calculate the matrix
\[
  A_{1,2} A_{0,1}^{-1}  = 
   \frac{1}{5}
  \left[
  \begin{matrix}
  3 & 2 \\
  -4  & 4 
  \end{matrix} 
  \right] .
\]
Using the identification of $\R^2$ and $\C$, which is an isometry if one uses the $2$-norm on $\R^2$,
this matrix contributes to $A_1 A_0^{-1}$
as an operator mapping the (complex) one-dimensional subspace spanned by~$\e_1$ to the (complex) one-dimensional subspace spanned by~$\e_2$. 
To determine its contribution to the estimate of the norm of $A_1 A_0^{-1}$,
we thus need to determine the $2$-norm of the matrix (as a linear map from $\R^2 \to \R^2$):
\[
  \| A_{1,2} A_{0,1}^{-1} \|  = \frac{1}{5} \sqrt{\frac{45+5\sqrt{17}}{2}}.
\]
Next, we calculate a bound on the map $ A_{1,*}     A_{0,*}^{-1}: \ell^1_0 \to \ell^1$:
\begin{equation}\label{eq:LUK}
  \| A_{1,*}     A_{0,*}^{-1} \| =  \| L_{\omega_0} \hat{U} K \| .
\end{equation}
To bound \eqref{eq:LUK} we first compute how $L_{\omega_0} K \hat{U}    $ operates on basis elements $\e_k$ for $k\geq 2$: 
\[
L_{\omega_0} K \hat{U}     \e_{k} = \frac{ -i+(-i)^k }{k-i (-i)^{k}}   \e_{k+1}
+
\frac{ i+(-i)^k  }{k-i (-i)^{k}}   \e_{k-1} .
\]
Since the norm of this expression is maximized when $k=2$ and $  \| L_{\omega_0} K \hat{U}    \e_2 \| = \tfrac{4\sqrt{10}}{5}$,
 we have calculated the $B(\ell^1_0,\ell^1)$ operator norm $ \|L_{\omega_0} K \hat{U}      \| = \tfrac{2\sqrt{10}}{5}$. 
As $\|A_1A_{0}^{-1}\|$ is equal to the maximum of  $ \| A_{1,2} A_{0,1}^{-1}\|$ and $\|A_{1,*}     A_{0,*}^{-1}\|$, it follows that 
	$ \| A_1 A_0^{-1}\| = \max\{ \frac{1}{5}\sqrt{\frac{45+5\sqrt{17}}{2}}, \frac{2 \sqrt{10}}{5} \}  = \frac{2 \sqrt{10}}{5}$.
%
\end{proof}

\begin{proposition}
		\label{prop:A0A1}
		Define $\overline{A_0^{-1} A_1 } \in \text{\textup{Mat}}((\R^3,\R^3)$ by
	\[
\overline{A_0^{-1} A_1 } :=
\left(
\begin{array}{ccc}
0 & 0 & \tfrac{1}{2}\sqrt{2+\frac{\pi ^2}{2}} \\
0 & 0 & \frac{1}{\sqrt{2}}  \\
\frac{8}{5 \pi } & \frac{2\sqrt{16+8 \pi +5 \pi ^2}}{5 \pi } & \frac{2}{\sqrt{5}} \\
\end{array}
\right)
	\] 
	Then $\overline{A_0^{-1} A_1 } $ is an upper bound (as defined in Definition~\ref{def:upperbound}) for $A_0^{-1} A_1 $.
\end{proposition}
\begin{proof}
We write $x=(\alpha,\omega,c)$.
Let $\pi_{\alpha,\omega}$ be the projection onto $\R^2$, whereas $\pi_c$ is the projection onto $\ell^1_0$. 
Then we can expand $  A_0^{-1} A_1 $ as follows:
\begin{alignat}{1}
  \pi_{\alpha,\omega} A_0^{-1} A_1 x &=   A_{0,1}^{-1}  i_\C^{-1} \pi_{1} A_{1,*} \pi_c x  \label{e:complicated1} \\
  \pi_{c} A_0^{-1} A_1 x  &= 
  A_{0,*}^{-1} (( i_\C A_{1,2} \pi_{\alpha,\omega} x)   \e_2 ) + A_{0,*}^{-1} \pi_{\geq 2} A_{1,*} \pi_c x . \label{e:complicated2}
\end{alignat}
We estimate the three operators that appear separately.

First, we note that the term $A_{0,*}^{-1} (( i_\C A_{1,2} \pi_{\alpha,\omega} x)   \e_2 ) $ in~\eqref{e:complicated2} essentially represents an operator from $\R^2$ to the (complex) one-dimensional subspace spanned by $\e_2$. Using the identification of $\C$ with $\R^2$, this map is represented by the matrix 
\[
  \frac{-2 }{25 \pi }
  \left[
  \begin{matrix}
  1 & -2 \\
  2 & 1
  \end{matrix} 
  \right] 
  \cdot
  \left[
  \begin{matrix}
  -2 & 2-\tfrac{3 \pi}{2} \\
  -4  & 2(2+\pi) 
  \end{matrix} 
  \right] \\
  = \frac{2 }{25 \pi }
  \left[
  \begin{matrix}
  -6 & 6+ 11 \pp \\
  8  & \pi -8
  \end{matrix} 
  \right] .
\]
It then follows that 
\begin{alignat*}{1}
  \| A_{0,*}^{-1} (( i_\C A_{1,2} \pi_{\alpha,\omega} x)   \e_2 ) \|
&\leq 	\frac{4}{25 \pi} 
	\left(|\alpha| \sqrt{(-6)^2+8^2}   + 
	 |\omega|  \sqrt{( 6+ 11 \pp )^2 + (\pi-8)^2} 
	 \right) \\
	&= \frac{4}{5 \pi} \left( 2 |\alpha| + 
	\frac{\sqrt{16+8 \pi +5 \pi ^2}}{2} |\omega| \right) .
\end{alignat*}

%
%
%
%

Next, we note that the term $A_{0,1}^{-1}  i_\C^{-1} \pi_{1} A_{1,*} \pi_c x $
in~\eqref{e:complicated1} essentially represents an operator from the (complex) one-dimensional subspace spanned by $\e_2$ to $\R^2$. Using the identification of $\C$ with $\R^2$, this map is represented by the matrix 
\[
 \pp 
 \left[
 \begin{matrix}
 0 & - \pp \\
 -1  & 1 
 \end{matrix} 
 \right]^{-1}
 \cdot
 \left[
 \begin{matrix}
 -1 & -1 \\
 1 &-1 
 \end{matrix} 
 \right] \\
 =
 \left[
 \begin{matrix}
 1 - \pp & 1 + \pp \\
 1  & 1
 \end{matrix} 
 \right] ,
\]
because $\pi_1 A_{1,*} \e_2 = \pp (i-1)$.
Hence
\begin{alignat*}{1}
  | \pi_\alpha  A_0^{-1} A_1 x |  &\leq  \tfrac{1}{2} \sqrt{ 2 + \tfrac{\pi^2}{2}}  \|c\| \\ 
  | \pi_\omega  A_0^{-1} A_1 x |  &\leq  \tfrac{1}{2}  \sqrt{2} \|c\|  . 
\end{alignat*}

%
%
%
 
Finally, note that the term $A_{0,*}^{-1}  \pi_{\geq 2} A_{1,*}$
appearing in~\eqref{e:complicated2} maps $\ell^1_0$ to itself. It can be expressed as
\[
A_{0,*}^{-1}  \pi_{\geq 2} A_{1,*}  = - i K \hat{U} \pi_{\geq 2} L_{\omega_0} .
\]
The operator $K \hat{U} \pi_{\geq 2} L_{\omega_0}$
 acts on basis elements $\{ \e_k\}_{k \geq 2}$ as follows:
\begin{alignat*}{1}
	K \hat{U} \pi_{\geq 2}  L_{\omega_0}  \e_2 &= -\frac{1+i}{4} \e_{3}  \\
	K \hat{U} \pi_{\geq 2}  L_{\omega_0}  \e_{k} &= \frac{-i+(-i)^k}{(k+1)-i (-i)^{k+1}} \e_{k+1} + \frac{i+(-i)^k}{(k-1) - i (-i)^{k-1}} \e_{k-1}
	\qquad\text{for } k \geq 3.
\end{alignat*}
Since $\max_{k\geq 2 } \| K \hat{U} \pi_{\geq 2} L_{\omega_0} \e_k \| = \| K
\hat{U} L_{\omega_0} \e_3 \| = \tfrac{4}{\sqrt{5}}$, 
the operator norm of $A_{0,*}^{-1}  \pi_{\geq 2} A_{1,*}$ is $\tfrac{2}{\sqrt{5}}$.

These three bounds on the three operators appearing in~\eqref{e:complicated1} and~\eqref{e:complicated2} lead to the asserted upper bound.
%
\end{proof}




\section{Appendix: Endomorphism on a Compact Domain}
\label{sec:CompactDomain}

In order to construct the Newton-like map $T$ we defined operators $ A =  DF(\bar{x}_\epsilon) + \cO(\epsilon^2)$ and $A^{\dagger} = A^{-1} + \cO(\epsilon^2)$. 
However, as $(\bar{\alpha}_\epsilon,\bar{\omega}_\epsilon,\bar{c}_\epsilon) = (\pp,\pp,\bar{c}_\epsilon) + \cO(\epsilon^2)$,  the map $A$ can be better thought of as an $\cO(\epsilon^2)$ approximation of $DF(\pp,\pp,\bar{c}_\epsilon)$. 
Thus, when working with the map $T$ and considering points $ x \in  B_\epsilon(r,\rho)$ in its domain, we will often have to measure the distances of $ \alpha$ and $ \omega $ from $ \pp$. 
To that end, we define the following variables which will be used throughout the rest of the appendices. 
\begin{definition}
	\label{def:DeltaDef}
For $ \epsilon \geq 0$, and $r_\alpha,r_\omega,r_c >0$ we define 
\begin{alignat*}{2}
	\da^0 	&:= \tfrac{\epsilon^2}{5} ( 3 \pp -1) & \qquad\qquad
	\da 	&:= \da^0 + r_\alpha \\
	\dw^0 &:=  \tfrac{\epsilon^2}{5} &
	\dw &:=  \dw^0 + r_{\omega} \\ 
	\dc^0 &:=  \tfrac{2 \epsilon}{\sqrt{5}} &
	\dc &:=  \dc^0 + r_c . 
\end{alignat*}
\end{definition}

%
When considering an element $ ( \alpha , \omega, c)$ for our $\cO(\epsilon^2)$ analysis, we are often concerned with the 
 distances $|\alpha - \pp|$, $|\omega - \pp|$ and $ \| c - \bar{c}_\epsilon\|$, each of which is of order $\epsilon^2$.  
To create some  notational consistency in these definitions, $\da^0$ and $\dw^0$ are of order $\epsilon^2$, whereas $\dc^0$ is not capitalized as it is of order $\epsilon$. 
Using these definitions, it follows that for any $\rho>0$ and all  $(\alpha, \omega, c ) \in B_\epsilon(r,\rho)$ we have: 
\begin{alignat*}{1}
| \alpha - \pp | & \leq  \da       \\ 
	 | \omega - \pp| & \leq  \dw   \\
	\|c \| &\leq  \dc  .
\end{alignat*}
In this interpretation the superscript $0$ simply refers to $r=0$, i.e., the center of the ball $(\alpha,\omega,c) = \bx_\epsilon$.

The following elementary lemma will be used frequently in the estimates. 
\begin{lemma}\label{lem:deltatheta}
For all $x\in \R$ we have $|e^{ix}-1| \leq |x|$.
Furthermore, for all $|\omega - \bomega_\epsilon  | \leq r_\omega$  
we have 
$ |e^{- i \omega} + i| \leq  \dw$ and
$ | e^{-2 i \omega } +1| \leq 2 \dw $ .
\end{lemma}
\begin{proof}
We start with
\[
  |e^{ix}-1|^2 = (\cos x -1)^2+(\sin x)^2=2(1-\cos x) \leq 2 \cdot \tfrac{1}{2} x^2 = x^2.
\]

Let $\theta = \omega - \pp$. Then $|\theta| \leq \dw$ and, using the previous inequality,
\[
| e^{- i \omega} + i|^2=
|e^{-i(\pp+\theta)}+i|^2=|e^{-i\theta}-1|^2 \leq  \theta^2 \leq  \dw^2.
\]
The final asserted inequality follows from an analogous argument.
\end{proof}

While the operators $U_\omega$ and $L_\omega$ are not continuous in $ \omega$ on all of $ \ell^1_0$, they are within the compact set $ B_\epsilon(r,\rho)$. 
To denote the derivative of these operators, we  define
\begin{alignat}{1}
	U_{\omega}' &:=  - i K^{-1} U_{\omega} \nonumber \\
	L_{\omega}' &:= - i \sigma^+( e^{- i \omega} I + K^{-1} U_{\omega}) + i \sigma^-(e^{i \omega} I - K^{-1} U_{\omega})  , \label{e:Lomegaprime}
\end{alignat}
and we derive Lipschitz bounds on $U_\omega$ and $L_\omega$ in the following proposition.
 
\begin{proposition}
	\label{prop:OmegaDerivatives}
	For the definitions above, $ \frac{\partial }{\partial  \omega} U_\omega = U_{\omega}' $ and $ \frac{\partial }{\partial  \omega}  L_\omega= L_{\omega}' $. 
	Furthermore,  for any $ (\alpha, \omega,c) \in B_\epsilon(r,\rho)$, we have the norm estimates
	\begin{alignat}{1}
	\| (U_{\omega} - U_{\omega_0} )c \| &\leq   \dw  \rho \nonumber  \\
	\|( L_{\omega} - L_{\omega_0} )c \| &\leq  2  \dw (  \dc +  \rho) .
	\label{e:LomegaLip}
	\end{alignat}
\end{proposition}

\begin{proof}
One easily calculates that $ \frac{\partial U_\omega}{\partial  \omega} =  U_{\omega}'$,  whereby
$
	\| (U_{\omega} - U_{\omega_0} )c \| \leq \int_{\omega_0}^\omega \| \tfrac{\partial}{\partial \omega} U_\omega c \|  \leq    \dw  \rho  
$. 
Calculating $ \frac{\partial }{\partial  \omega}  L_{\omega} $, we obtain the following:
\begin{alignat*}{1}
 \frac{\partial }{\partial  \omega}  L_{\omega} 
&=  \frac{\partial }{\partial  \omega} \left[  \sigma^+( e^{- i \omega} I + U_{\omega}) + \sigma^-(e^{i \omega} I + U_{\omega}) \right] \\
&= - i \sigma^+( e^{- i \omega} I + K^{-1} U_{\omega}) + i \sigma^-(e^{i \omega} I - K^{-1} U_{\omega}) ,
\end{alignat*}
thus proving $ \frac{\partial L_\omega}{\partial  \omega} =  L_{\omega}'$,
and 
$\|( L_{\omega} - L_{\omega_0} )c \| \leq  \int_{\omega_0}^\omega \| \tfrac{\partial}{\partial \omega} L_\omega c \|  \leq   \dw ( 2  \dc + 2 \rho)$.
\end{proof}

\begin{proposition}
	Let $\epsilon\geq 0$ and  $r=(r_\alpha,r_\omega,r_c) \in \R^3_+$. 
	For any $ \rho > 0$ the map 
	 $T:B_{\epsilon}(r,\rho) \to \R^2 \times \ell^K_0 $ is well defined. 	
	We define functions 
	\begin{alignat*}{1}
	C_0 &:=  \frac{2 \epsilon^2}{\pi} 
	\left[
	\frac{8}{5},\frac{2}{5} \sqrt{16+ 8\pi + 5 \pi^2},\frac{5 \pi }{2} 
	\right]
	\cdot \overline{A_0^{-1} A_1} \cdot [0,0 , \dc ]^T ,
	\\
	C_1 &:= \frac{5 }{2 \pi} + \frac{\epsilon \sqrt{10}}{\pi}, \\
	C_2 &:= \dw  \left[  (1 + \pp) + \epsilon \pi  \right] , \\
	C_3 &:=  
	\da (2+ \dc) +	2 \dw (1+\pp) 
		+ \epsilon \left[ \pi + 2\da  + 4 \dc \da + \pi \dw \dc  + (\pp + \da ) \dc^2 \right] ,
	\end{alignat*}
where the expression for $C_0$ should be read as a product of a row vector, a $(3 \times 3)$ matrix and a column vector.
Furthermore we define, for any $\epsilon,r_\omega$ such that $C_1 C_2 <1$,
	\begin{equation}
		C(\epsilon,r_\alpha,r_\omega,r_c) := \frac{C_0+ C_1 C_3}{1 - C_1 C_2}
		 \, .
		\label{eq:RhoConstant}
	\end{equation}
	All of the functions $C_0,C_1,C_2,C_3$ and $C$ are nonnegative and monotonically increasing in their arguments $\epsilon$ and~$r$. 
	Furthermore, if  $C_1 C_2 < 1$ and $	C(\epsilon,r_\alpha,r_\omega,r_c) \leq \rho $
	then $\| K^{-1} \pi_c  T( x) \| \leq \rho $
	for $x \in B_{\epsilon}(r,\rho)$. 
	\label{prop:DerivativeEndo}
\end{proposition}


\begin{proof}
	Given their definitions, it is straightforward to check that the functions $C_i$ and $C$ are monotonically increasing in their arguments.  
	To prove the second half of the proposition, we split 
	$K^{-1} \pi_c  T(x)$ into several pieces. 
	We define the projection $\pi_c^0 x = (0,0,\pi_c x)$.
We then obtain
	\begin{alignat*}{1}
	K^{-1} \pi_c  T(x)  &= K^{-1} \pi_c   [ x - A^{\dagger} F(x) ]   \\
	&= K^{-1} \pi_c  [ I \pi_c^0 x -    A^{\dagger} ( A \pi_c^0 x + F(x) - A \pi_c^0 x)]  \\
	&= \epsilon^2 K^{-1} \pi_c (A_0^{-1}A_{1})^2 \pi_c^0 x + K^{-1} \pi_c A^{\dagger} (F(x) - A \pi_c^0 x) \nonumber \\
	&=  \frac{2 \epsilon^2}{i\pi} \hat{U} \pi_{\ge 2} A_1 A_0^{-1}A_{1} \pi_c^0 x +\frac{2 }{i\pi} \hat{U}  \pi_{\ge 2} (I-\epsilon A_1 A_0^{-1}) (F(x) - A\pi_c^0 x)  ,
\end{alignat*}
where we have used that $K^{-1} \pi_c A_0^{-1} = \frac{2}{i\pi} \hat{U} \pi_{\ge 2}$, with the projection $\pi_{\ge 2}$ defined in~\eqref{e:pige2}.
By using $\| \hat{U} \| \leq \frac{5}{4}$, see Proposition~\ref{p:severalnorms}, we obtain the estimate
\begin{equation}
	\| K^{-1} \pi_c T(x) \| \leq   \frac{2 \epsilon^2}{\pi} \overline{\hat{U}\pi_{\ge 2} A_1} \cdot  \overline{ A_0^{-1}A_{1}}  \cdot
	[0,0,\dc ]^T +\frac{5 }{2 \pi} \left(1 + \epsilon \| A_1 A_0^{-1} \| \right) \|F(x) - A\pi_c^0 x \| .
	\label{eq:DerivativeEndo}
\end{equation}
Here the $(1 \times 3)$ row vector $\overline{\hat{U}\pi_{\ge 2} A_1}$ is an upper bound on $\hat{U}\pi_{\ge 2} A_1$ interpreted as a linear operator from $\R^2 \times \ell^1_0$ to $\ell^1_0$, thus extending in a straightforward manner the definition of upper bounds given in  Definition~\ref{def:upperbound}.

	We have already calculated  an expression for
	 $ \overline{ A_0^{-1}A_{1}}$ in Proposition~\ref{prop:A0A1},  and  $  \| A_1 A_0^{-1}\| =\frac{2\sqrt{10}}{5}$ by Proposition~\ref{prop:A1A0}.  In order to finish the calculation of the right hand side of Equation \eqref{eq:DerivativeEndo}, we need to  estimate  $\| F(x) - A\pi_c^0 x \|$ and $\overline{\hat{U} \pi_{\ge 2} A_1} $. 
	We first calculate a bound on $\hat{U} \pi_{\ge 2} A_1 $. 
	We note that $ \hat{U} \pi_{\ge 2} A_1  =  \hat{U} \e_2 ( i_\C A_{1,2} \pi_{\alpha,\omega})+ \hat{U} \pi_{\ge 2}A_{1,*} \pi_c$.	
As $\|\hat{U} e_2\| = \| \tfrac{4-2i}{5} \e_2\|$,
it follows from the definition of $A_{1,2}$ 
that 
\[
	 \left| i_\C  A_{1,2}
	 \left( \!\!\begin{array}{c}\alpha \\ \omega \end{array} \!\!\right) \right|  
	 \cdot \| \hat{U} \e_2 \| 
	 \leq 
	 \left(\frac{\sqrt{20}}{5} |\alpha| +  \frac{\sqrt{(2-3 \pi/2)^2 +4(2+\pi)^2}}{5} |\omega| \right)  \cdot \frac{4}{\sqrt{5}}.
\]
	To calculate $ \| \hat{U} \pi_{\ge 2} A_{1,*} \|$ we note that $ \| \hat{U}\| \leq \frac{5}{4}$ and $ \|A_{1,*}\| = \pp \| L_{\omega_0} \| \leq 2 \pi$. 
	Hence $ \| \hat{U} \pi_{\ge 2} A_{1,*} \| \leq \frac{5 \pi}{2}$. 
	Combining these results, we obtain  that
	\[
	\overline{\hat{U} \pi_{\ge 2}  A_1 } = \left[\frac{8}{5},\frac{2}{5} \sqrt{16 + 8 \pi + 5 \pi^2},\frac{5 \pi }{2} \right].
	\] 
Thereby, it follows from~\eqref{eq:DerivativeEndo} that 
\begin{equation}\label{e:C0C1}
	\| K^{-1} \pi_c T(x) \| \leq C_0 + C_1 \| F(x) - A \pi_c^0 x\|. 
\end{equation}
We now calculate
	\begin{alignat*}{1}
	F(x) - A \pi_c^0 x &= 
	(i \omega + \alpha e^{-i \omega} ) \e_1 + 
	( i \omega K^{-1} + \alpha U_{\omega}) c + 
	\epsilon \alpha e^{-i \omega} \e_2  +
	\alpha \epsilon L_\omega c + 
	\alpha \epsilon [ U_{\omega} c] * c  
	\\ &\qquad 
	- \pp (i K^{-1} + U_{\omega_0} + \epsilon L_{\omega_0} ) c \\
	&= i ( \omega - \pp) K^{-1} c + ( \alpha - \pp) U_{\omega} c +  \pp ( U_{\omega} - U_{\omega_0})c  \nonumber \\
	&\qquad  + \left[i ( \omega - \pp ) + ( \alpha - \pp) e^{-i \omega} + \pp( e^{- i \omega }+ i)\right] \e_1  \nonumber
	\\ 
	&\qquad  +\epsilon  \alpha   e^{-i \omega}  \e_2  
+  ( \alpha- \pp)  \epsilon L_{\omega} c + \pp \epsilon ( L_{\omega} - L_{\omega_0}) c + \alpha \epsilon [ U_{\omega} c ] * c .
	\end{alignat*}
Taking norms and using~\eqref{e:LomegaLip} and Lemma~\ref{lem:deltatheta}, we obtain 
	\begin{alignat*}{1}
	\| F(x) - A \pi_c^0 x\|& \leq  
	 \dw \rho + \da \dc + \pp \dw \rho
    +	2 (\dw + \da + \pp \dw)  
	   \\
	&\qquad + \epsilon \left[ 2(\pp + \da ) + 4 \dc \da + \pi  \dw (  \dc + \rho) + (\pp + \da ) \dc^2 \right]  \\
		&= \dw [ (1+\pp) +   \epsilon \pi ] \rho \nonumber \\ 
	&\qquad +  \da (2 + \dc)
	+	2 \dw (1+\pp) 
	+ \epsilon \left[ \pi + 2\da  + 4 \dc \da + \pi \dw \dc  + (\pp + \da ) \dc^2 \right].  
	\end{alignat*}

	We have now computed all of the necessary constants. Thus $ \| F(x) - A \pi_c^0 x \| \leq C_2 \rho + C_3$, and from~\eqref{e:C0C1}   we obtain 
	\begin{eqnarray*}
	\| K^{-1} \pi_c T(c) \|
	&\leq & C_0 +  C_1 ( C_2  \rho + C_3),
	\end{eqnarray*}
with the constants defined in the statement of the proposition.
We would like to select values of $\rho$ for which 
	\[
	\| K^{-1} \pi_c T(c) \| \leq \rho
	\]
	This is true if  
	$	C_0 +  C_1 ( C_2  \rho + C_3) \leq \rho$, 
	or equivalently 
	\[
	\frac{C_0 + C_1 C_3 }{1 - C_1 C_2} \leq \rho.
	\]
	This proves the theorem.
\end{proof}




\section{Appendix: The bounding functions for $Y(\epsilon)$}
\label{sec:YBoundingFunctions}

We need to define $Y(\epsilon)$ so that it bounds $ T(\bar{x}_\epsilon) -\bar{x}_\epsilon = A^{\dagger} F(\bar{x}_\epsilon)$. 
We introduce $c_2(\epsilon) := \frac{2-i}{5} \epsilon$. 
We can explicitly calculate $F(\bar{x}_\epsilon)$ as follows:
\begin{alignat*}{1}
	F_1( \bar{x}_\epsilon) &=
	( i \bar{\omega}_\epsilon + \bar{\alpha}_\epsilon e^{-i \bar{\omega}_\epsilon}) + 
	\bar{\alpha}_\epsilon \epsilon ( e^{i \bar{\omega}_\epsilon} + e^{-2 i \bar{\omega}_\epsilon}) c_2( \epsilon) \\
	F_2( \bar{x}_\epsilon) &= 
	( 2 i \bar{\omega}_\epsilon + \bar{\alpha}_\epsilon e^{- 2 i \bar{\omega}_\epsilon}) c_2( \epsilon)+ 
	\bar{\alpha}_\epsilon \epsilon  e^{ -i \bar{\omega}_\epsilon} \\
	F_3( \bar{x}_\epsilon) &= \bar{\alpha}_\epsilon \epsilon ( e^{- i \bar{\omega}_\epsilon} + e^{-2 i \bar{\omega}_\epsilon}) c_2( \epsilon) \\
	F_4(\bar{x}_\epsilon) &=  \bar{\alpha}_\epsilon \epsilon  e^{-2 i \bar{\omega}_\epsilon} c_2( \epsilon)^2 \\
	F_{k}(\bar{x}_\epsilon) &=  0 \qquad\text{for all } k\geq 5.
\end{alignat*}
By using the definition of $ A^{\dagger} = A_{0}^{-1} - \epsilon A_{0}^{-1} A_1 A_{0}^{-1} $ we can calculate $A^{\dagger} F(\bar{x}_\epsilon)$ explicitly using a finite number of operations.  
However, proving $ \epsilon^{-2} Y (\epsilon) $ is well defined and increasing requires more work.  
To estimate $A^{\dagger} F(\bar{x}_\epsilon)$ in Theorem~\ref{prop:Ydef} below, we will take entry-wise absolute values in the constituents of $A^{\dagger}$, as clarified in the next remark.
\begin{remark}
 Since $F(\bar{x}_\epsilon)$ is a finite linear combination of the basis elements $\e_k$, and the operators $A_0$ and $A_1$ are diagonal and tridiagonal, respectively,  we can represent $A_0^{-1} \cdot F(\bar{x}_\epsilon)$ and $A_{0}^{-1} A_1 A_0^{-1} \cdot F(\bar{x}_\epsilon)$ by finite dimensional matrix-vector products. 
 By $|A_0^{-1}|$ and $|A_{0}^{-1} A_1 A_0^{-1} |$ we denote the entry-wise absolute values of these matrices. 
\end{remark}

\begin{theorem} \label{prop:Ydef}
	Let  $ f_i: \R \to \R$ for $i=1,2,3,4$ be defined as in Propositions \ref{prop:YfBound1}, \ref{prop:YfBound2},  \ref{prop:YfBound3}, and  \ref{prop:YfBound4} below. 
Define $f(\epsilon)= \sum_{i=1}^4 f_i \e_i \in \ell^1$ and 
define the function $\hat{Y}: \R \to \R^2 \times \ell^1_0$ 
to be 
\begin{equation}
		\hat{Y}(\epsilon) : =  \left| A_0^{-1}  \right| \cdot  f(\epsilon)    + \epsilon \left| A_0^{-1} A_1  A_0^{-1}   \right|    \cdot  f(\epsilon)   .
\end{equation}	
Then the only nonzero components of   $\hat{Y}=(\hat{Y}_\alpha,\hat{Y}_\omega,\hat{Y}_c)$ are $\hat{Y}_\alpha$, $\hat{Y}_\omega$ and $(\hat{Y}_c)_k$ for $k=2,3,4,5$.
Furthermore, define 
\begin{align}
	Y_\alpha(\epsilon) :=& \hat{Y}_\alpha(\epsilon) &
	Y_\omega(\epsilon) :=& \hat{Y}_\omega(\epsilon) &	
	Y_c(\epsilon) :=& 2 \sum_{k=2}^5 (\hat{Y}_c)_k(\epsilon) 
\end{align}
	Then $[Y_\alpha(\epsilon),Y_\omega(\epsilon),Y_c(\epsilon) ]^T$ is an upper bound on $ T(\bar{x}_\epsilon ) - \bar{x}_\epsilon$, and $ \epsilon^{-2} [Y_\alpha(\epsilon),Y_\omega(\epsilon),Y_c(\epsilon) ]$ is non-decreasing in $\epsilon$.  
\end{theorem}
\begin{proof}
	By Propositions \ref{prop:YfBound1}, \ref{prop:YfBound2},  \ref{prop:YfBound3} and  \ref{prop:YfBound4} it follows that $|F_i(\bar{x}_\epsilon)| \leq f_i(\epsilon)$ for $i=1,2,3,4$. 
	By taking the entry-wise absolute values $ \left| A_0^{-1}  \right|   $  and $\left| A_0^{-1} A_1  A_0^{-1}   \right| $, it follows that $ |T(\bar{x}_\epsilon ) - \bar{x}_\epsilon| \leq \hat{Y}$, where the absolute values and inequalities are taken element-wise. 
We note that  in defining $ Y_c$ the factor~$2$ arises from our choice of norm in \eqref{e:lnorm}. 
	To see that  $(\hat{Y}_c)_k$ is non-zero for $k = 2,3,4,5$ only, we note that while $A_0^{-1}$ is a block diagonal operator, $A_1$  has off-diagonal terms. In particular, $ A_{1,*} \e_k = \pp (-i+(-i)^k) \e_{k+1} + \pp (i+(-i)^k) \e_{k-1}$ for $k \geq 2$, whereby  $( \hat{Y})_k =0$ for $ k \geq 6$.

Next we show that  $\epsilon^{-2} [ Y_\alpha(\epsilon), Y_\omega(\epsilon), Y_c(\epsilon)]^T$ is nondecreasing in $\epsilon$. We note that it follows from Definition \ref{def:DeltaDef} that each function $f_i(\epsilon)$ is a polynomial in $\epsilon$ with nonnegative coefficients, and the lowest degree term is at least $ \epsilon^2$. 
Additionally, $\left| A_0^{-1}  \right| \cdot  f(\epsilon)   $  is a positive linear combination of the functions $\{f_i(\epsilon)\}_{i=1}^4$, 
whereas $\left| A_0^{-1} A_1  A_0^{-1}   \right|    \cdot  f(\epsilon)  $ is $ \epsilon$ times a positive linear combination of  $\{f_i(\epsilon)\}_{i=1}^4$. 
It follows that each component of  $\hat{Y}$  is a polynomial in $\epsilon$ with nonnegative coefficients, and the lowest degree term is at least $ \epsilon^2$. 
	Thereby $\epsilon^{-2} [ Y_\alpha(\epsilon),Y_\omega(\epsilon),Y_c(\epsilon)]^T$ is nondecreasing in $\epsilon$.

\end{proof}
Before presenting Propositions \ref{prop:YfBound1}, \ref{prop:YfBound2},  \ref{prop:YfBound3} and  \ref{prop:YfBound4}, 
we recall that the definitions of $\da^0$, $\dw^0$ and $\dc^0$ are given  in  Definition~\ref{def:DeltaDef}. 
\begin{proposition}
	\label{prop:YfBound1}
Define 
\begin{equation}\label{e:f1}
f_1 (\epsilon) := \pp ( \tfrac{1}{2} (\dw^0)^2 + \tfrac{1}{6} (\dw^0)^3) + \da^0 \dw^0 +  \da^0 \epsilon \dc^0 + \tfrac{3\pi   }{4} \dw^0 \epsilon \dc^0 .
\end{equation}
	Then $| F_1(\bx_\epsilon) | \leq  f_1(\epsilon)$. 
\end{proposition}
\begin{proof}
Note that 
\begin{equation}\label{e:F1}
F_1(\bar{x}_\epsilon ) =
 i \bar{\omega}_\epsilon + \bar{\alpha}_\epsilon e^{-i \bar{\omega}_\epsilon} + 
\bar{\alpha}_\epsilon \epsilon  c_2( \epsilon)  ( e^{i \bar{\omega}_\epsilon} + e^{-2 i \bar{\omega}_\epsilon}) .
\end{equation}
We will show that all of the $ \cO(\epsilon^3)$ in $F_1(\bx_\epsilon)$ cancel. 
We first expand the first summand~\eqref{e:F1}:
\begin{equation*}
	i \bar{\omega}_\epsilon  = i \pp -  i \dw^0   \label{e:Y1a}.
\end{equation*}
Next, we expand the second summand in~\eqref{e:F1}:
\begin{alignat}{1}
\bar{\alpha}_\epsilon e^{- i \bar{\omega}_\epsilon }	&= - i \bar{\alpha}_\epsilon e^{i \dw^0  }	
=  -i \left( \pp  e^{i \dw^0 } +\da^0  e^{i \dw^0 } \right) \nonumber\\
&= -i \left(\pp  (1 + i \dw^0 )   +\da^0  \right)  
 -i \left( \pp  (e^{i \dw^0 }-1- i \dw^0) + \da^0 (e^{i \dw^0 }-1)\right) . \label{e:Y1b}
\end{alignat}
Finally, we expand the third summand~\eqref{e:F1} as
\begin{alignat}{1}
\bar{\alpha}_\epsilon \epsilon^2 \tfrac{2-i}{5} ( e^{i \bar{\omega}_\epsilon} + e^{-2 i \bar{\omega}_\epsilon})   &= 
\pp \epsilon^2 \tfrac{2-i}{5} (i  - 1) + \pp \epsilon^2 \tfrac{2-i}{5} \left(i (e^{- i \dw^0} -1)-(e^{2 i \dw^0}-1) \right) \nonumber \\
&\hspace*{3cm} +\da^0 \epsilon^2 \tfrac{2-i}{5} \left(i e^{- i \dw^0} - e^{2 i \dw^0} \right)  \label{e:Y1c}.
\end{alignat}
If we now collect the final term from~\eqref{e:Y1b} and the final two terms from \eqref{e:Y1c}
in 
\begin{alignat*}{1}
g(\epsilon) &:=
-i \left( \pp  (e^{i \dw^0 }-1- i \dw^0) + \da^0  (e^{i \dw^0 }-1)\right) \\
& \quad\qquad +
\da^0 \epsilon^2 \tfrac{2-i}{5} \left(i e^{- i \dw^0} - e^{2 i \dw^0}\right) \\
& \quad\qquad\qquad+
\pp \epsilon^2 \tfrac{2-i}{5} \left(i (e^{- i \dw^0} -1)-(e^{2 i \dw^0}-1)\right) ,
\end{alignat*}
then we can write $F_1(\bar{x}_\epsilon)$ as 
\begin{alignat*}{1}
F_1(\bar{x}_\epsilon) &= g(\epsilon) 
+ i \pp -  i \dw^0  
-i \left(\pp  (1 + i \dw^0 )   + \da^0  \right)
+ \pp \epsilon^2 \tfrac{2-i}{5} (i  - 1) \\
&= g(\epsilon).
\end{alignat*}
Using Lemma~\ref{lem:deltatheta} it  is not difficult to see that $|g(\epsilon)|$ can be bounded by $f_1(\epsilon)$, as defined in~\eqref{e:f1}.
%
\end{proof}

\begin{proposition}
		\label{prop:YfBound2}
	Define 
		\begin{equation}\label{e:f2}
		f_2(\epsilon) :=		(\pp + \da^0) \dw^0 ( \dc^0 + \epsilon  ) +   \tfrac{1}{2}\dc^0 ( 2 \dw^0 + \da^0 ) + \epsilon \da^0.
		\end{equation}
	Then $| F_2(\bx_\epsilon) | \leq f_2(\epsilon)$. 
\end{proposition}
\begin{proof}
	First note that 
\begin{alignat}{1}
	F_2(\bar{x}_\epsilon) &= 
	( 2 i \bar{\omega}_\epsilon + \bar{\alpha}_\epsilon e^{- 2 i \bar{\omega}_\epsilon}) c_2( \epsilon)+ 
	\bar{\alpha}_\epsilon \epsilon  e^{ -i \bar{\omega}_\epsilon} 
= \left(  2 i \bar{\omega}_\epsilon - \bar{\alpha}_\epsilon e^{ 2 i \dw^0 } \right) \tfrac{2-i}{5} \epsilon - i	\bar{\alpha}_\epsilon \epsilon  e^{ i \dw^0 }   \nonumber \\
		&=  \left(  2 i \bar{\omega}_\epsilon - \bar{\alpha}_\epsilon  \right) \tfrac{2-i}{5} \epsilon - i	\bar{\alpha}_\epsilon \epsilon   
   - \bar{\alpha}_\epsilon (e^{ 2 i \dw^0 } -1)  \tfrac{2-i}{5} \epsilon -i	\bar{\alpha}_\epsilon \epsilon  (e^{ i \dw^0 }-1).
\label{e:F2part}
\end{alignat}
We expand the first part of the right hand side in~\eqref{e:F2part} as
	\begin{alignat*}{1}
 \left(  2 i \bar{\omega}_\epsilon - \bar{\alpha}_\epsilon  \right) \tfrac{2-i}{5} \epsilon - i	\bar{\alpha}_\epsilon \epsilon  &= \left( 2 i \pp - \pp \right) \tfrac{2-i}{5} \epsilon - i \pp \epsilon 
 +  \left(  - 2 i \dw^0 - \da^0 \right) \tfrac{2-i}{5} \epsilon -	i \da^0  \epsilon
 \\
& = - \left(   2 i \dw^0 + \da^0 \right) \tfrac{2-i}{5} \epsilon -	i \da^0  \epsilon .
\end{alignat*}
Hence, we can rewrite $F_2(\epsilon)$ as 
	\begin{equation*}
F_2(\bx_\epsilon) =   - \bar{\alpha}_\epsilon (e^{ 2 i \dw^0 } -1)  \tfrac{2-i}{5} \epsilon -i	\bar{\alpha}_\epsilon \epsilon  (e^{ i \dw^0 }-1)
 -\left(   2 i \dw^0 + \da^0 \right) \tfrac{2-i}{5} \epsilon -	i \da^0  \epsilon .
	\end{equation*}
Using Lemma~\ref{lem:deltatheta} it  is then not difficult to see that $|F_2(\bx_\epsilon)|$ can be bounded by $f_2(\epsilon)$, as defined in~\eqref{e:f2}.
\end{proof}

\begin{proposition}
		\label{prop:YfBound3}
	Define 
		\begin{equation}\label{e:f3} 
		f_3(\epsilon):= \tfrac{1}{2}(\pp + \da^0 )  ( \sqrt{2} + 3 \dw^0 )
			 \epsilon \dc^0   .
		\end{equation} 
	Then $| F_3(\bx_\epsilon) | \leq  f_3(\epsilon)$. 
\end{proposition}
\begin{proof}
Note that 
\[
F_3( \bar{x}_\epsilon) = \bar{\alpha}_\epsilon \epsilon ( e^{- i \bar{\omega}_\epsilon} + e^{-2 i \bar{\omega}_\epsilon}) c_2( \epsilon)  . \\
\]
We expand this as
\begin{alignat*}{1}
F_3( \bar{x}_\epsilon) &= -\bar{\alpha}_\epsilon \epsilon^2 \tfrac{2-i}{5} ( i e^{ i \dw^0 } + e^{2 i \dw^0 }) \\
&= -\bar{\alpha}_\epsilon \epsilon^2 \tfrac{2-i}{5} ( i+1) -\bar{\alpha}_\epsilon \epsilon^2 \tfrac{2-i}{5} \left( i ( e^{ i \dw^0 }-1) +( e^{2 i \dw^0 } -1)\right)  .
\end{alignat*}
Using Lemma~\ref{lem:deltatheta} it is then not difficult to see that $|F_3(\bx_\epsilon)|$ can be bounded by $f_3(\epsilon)$, as defined in~\eqref{e:f3}.
\end{proof}

\begin{proposition}
		\label{prop:YfBound4}
Define
\begin{equation}\label{e:f4}
	f_4(\epsilon) := \tfrac{1}{5} (\pp+\da^0) \epsilon^3   
\end{equation}
	Then $| F_4(\bx_\epsilon) | \leq  f_4(\epsilon)$. 
\end{proposition}
\begin{proof}
Note that 
\[
F_4( \bar{x}_\epsilon ) =  \bar{\alpha}_\epsilon \epsilon  e^{-2 i \bar{\omega}_\epsilon} \left(\tfrac{2-i}{5}\epsilon \right)^2 ,
\]
from which it follows that $|F_4(\bx_\epsilon)|$ can be bounded by $f_4(\epsilon)$, as defined in~\eqref{e:f4}.
\end{proof}

\section{Appendix: The bounding functions for $Z(\epsilon,r,\rho)$}
\label{sec:BoundingFunctions}

	In this section we calculate an upper bound on $DT$.  
	To do so we first calculate 
	\(
	DF = 
	\left[ \frac{\partial F}{\partial  \alpha}, 
	\frac{\partial F}{\partial  \omega},
	\frac{\partial F}{\partial  c}
	\right]
	\):
	\begin{alignat}{1}
	\label{eq:FpartialA}
	\frac{\partial F}{\partial  \alpha} &= e^{-i \omega} \e_1 + U_\omega c + \epsilon e^{-i \omega} \e_2 + \epsilon L_\omega c + \epsilon [ U_\omega c] * c , \\
	\label{eq:FpartialW} 
	\frac{\partial F}{\partial  \omega} &=
	i(1-\alpha e^{-i \omega}) \e_1 + 
	i K^{-1} ( I - \alpha U_{\omega} ) c  -
	i \alpha \epsilon e^{-i \omega} \e_2 + 
	\alpha \epsilon L_{\omega}' c - i \alpha \epsilon [ K^{-1} U_\omega c ] *c ,
	\\
	\frac{\partial F}{\partial  c} \cdot b 
	& =
	( i \omega K^{-1} + \alpha U_{\omega}) b + \alpha \epsilon \left( L_\omega b  + [ U_\omega b] * c + [U_{\omega} c ]*b \right)  , \qquad \text{for all $b\in \ell^K_0$},
	\label{eq:Fcderivative}
\end{alignat}	
where $L_{\omega}'$ is given in~\eqref{e:Lomegaprime}, and $\frac{\partial F}{\partial  c}$ is expressed in terms of the directional derivative. 
Recall that $\II$ is used to denote the $ 3 \times 3$ identity matrix. 

%
%
%

\begin{theorem}
	\label{prop:Zdef}
	Define $\overline{A_0^{-1} A_1}$ as in Proposition \ref{prop:A0A1} and define the matrix 
	\begin{equation}\label{e:defM}
	M := 
	\left(
	\begin{array}{cc}
	\sqrt{\tfrac{4}{\pi^2}+1} & 0 \\
	\frac{2}{\pi } & 0 \\
	0 & 1 \\
	\end{array}
	\right)
	\left(
	\begin{array}{ccc}
	f_{1,\alpha } & f_{1,\omega } & f_{1,c} \\
	f_{*,\alpha } & f_{*,\omega } & f_{*,c} \\
	\end{array}
	\right) ,
	\end{equation}
	where the functions $f_{1,\cdot}(\epsilon,r,\rho)$ and $f_{*,\cdot}(\epsilon,r,\rho)$ are defined as in Propositions \ref{prop:Z1a}--\ref{prop:Zsc}. 
	If we define $Z(\epsilon,r,\rho)$ as 
	\begin{equation}
		Z(\epsilon,r,\rho) := \epsilon^2  \left(\overline{ A_0^{-1} A_1 }\right)^2  + 
		\left(\II + \epsilon \overline{ A_0^{-1} A_1 } \right) \cdot M ,
	\end{equation}
	then $Z(\epsilon,r)$ is an upper bound (in the sense of Definition~\ref{def:upperbound}) on $DT(x)$ for all $ x \in B_\epsilon(r , \rho)$. 
	Furthermore, the components of $Z(\epsilon,r,\rho)$ are increasing in  $ \epsilon$, $r$ and $\rho$. 
\end{theorem}

\begin{proof}

	If we fix some $x \in B_\epsilon(r,\rho)$, then we obtain 
\begin{alignat*}{1}
		D T( x ) &=  I - A^{\dagger}  D F( x)  \\
		&= ( I - A^{\dagger} A) - A^{\dagger} \left[ D F( x)  - A \right]\\
		&=   \epsilon^2 (A_0^{-1} A_1 )^2 -    [I - \epsilon (A_0^{-1} A_1 ) ] \cdot  A_0^{-1} \cdot   \left[ D F( x) - A \right] ,
\end{alignat*}
hence an upper bound on $DT(x)$ is given by
\begin{equation*}
\epsilon^2  \left(\overline{ A_0^{-1} A_1 }\right)^2  + 
\left(\II + \epsilon \overline{ A_0^{-1} A_1 } \right) \cdot 
\overline{A_0^{-1}  \left[D F( x ) - A \right] },
\end{equation*}
	where $\overline{A_0^{-1}  \left[D F( x ) - A \right] }$ is a yet to be determined upper bound on $A_0^{-1}  \left[D F( x ) - A \right]$.  
To calculate this upper bound, we break it up into two parts: 
	\begin{alignat}{1}
		\pi_{\alpha,\omega} A_0^{-1}  \left( D F( x) - A \right)  &= 
		A_{0,1}^{-1}  i_{\C}^{-1} \pi_1   \left( D F( x) - A \right) 
\label{eq:Zfinite}\\ 
		\pi_c 		A_0^{-1}  \left( D F( x) - A \right)  &=
		A_{0,*}^{-1}  \pi_{\geq 2} \left( DF(x) - A \right)  . \label{eq:Zstar} 
\end{alignat}
	
To calculate an upper bound on \eqref{eq:Zfinite}, we use the explicit expression for $A_{0,1}^{-1}$ to estimate  
	\begin{alignat*}{1}
	\left| \pi_\alpha  A_{0,1}^{-1} \pi_1 \left( D F( x) - A \right)\right| &\leq  \sqrt{\tfrac{4}{\pi^2} + 1} \, \overline{\pi_1( DF( x ) - A) } \\
	\left| \pi_\omega  A_{0,1}^{-1} \pi_1 \left( D F( x) - A \right)\right| &\leq  \tfrac{2}{\pi} \,  \overline{ \pi_1( DF( x ) - A)  } ,
	\end{alignat*}
where $\overline{ \pi_1( DF( x ) - A)  }$ is an upper bound on $\pi_1( DF( x ) - A)$, viewed as an operator from $\R^2 \times \ell^K_0 $ to $\C$ (a straightforward generalization of Definition~\ref{def:upperbound}).
Indeed, in Propositions \ref{prop:Z1a}, \ref{prop:Z1w} and \ref{prop:Z1c} we   determine functions $f_{1,\cdot}$ such that, for all $x \in B_\epsilon(r,\rho)$, 
	\begin{alignat*}{1}
	f_{1,\alpha} (\epsilon,r,\rho) &\geq    \left|  \frac{\partial F_1 }{\partial \alpha} (x) + i  \right|  , \\
	f_{1,\omega} (\epsilon,r,\rho) &\geq   \left|  \frac{\partial F_1}{\partial \omega} (x)- (i- \pp)  \right|   , \\
	f_{1,c} (\epsilon,r,\rho) &\geq   \left|  \frac{\partial F_1}{\partial c} (x) \cdot b -  \pp \epsilon (i-1) \pi_2 b \right| ,
	\qquad\text{for all $b\in\ell^K_0$ with $\|b\| \leq 1$}.
	\end{alignat*} 
Here the projection $\pi_2$ is defined as $\pi_2 b := b_2 \in \C$ for $b=\{b_k\}_{k=1}^{\infty} \in \ell^1$.	
	Hence $ [ f_{1,\alpha} , f_{1,\omega}, f_{1,c}]$ is an upper bound on $\pi_1 (DF( x ) -A )$.

	For calculating an upper bound on Equation~\eqref{eq:Zstar}, in Propositions  \ref{prop:Zsa}, \ref{prop:Zsw}  and \ref{prop:Zsc} we  determine functions $f_{*,\cdot}$ such that, for all $x \in B_\epsilon(r,\rho)$,   
	\begin{alignat*}{1}
	f_{*,\alpha} (\epsilon,r,\rho)&\geq  \left\| A_{0,*}^{-1}  \pi_{\geq 2} \left(
	\frac{\partial F}{\partial \alpha}(x)  +    \epsilon  \tfrac{2 +4 i}{5} \e_2  \right) \right\| , \\
	%
%
	f_{*,\omega} (\epsilon,r,\rho)&\geq  \left\| A_{0,*}^{-1}  \pi_{\geq 2}\left( \frac{\partial F}{\partial \omega}(x) -  \epsilon \left[ \tfrac{4-3\pi}{10} + \tfrac{2(2 + \pi)}{5}i \right] \e_2  \right) \right\| ,  \\
	f_{*,c} (\epsilon,r,\rho) &\geq    \left\|  A_{0,*}^{-1} \pi_{\geq 2} \left( \frac{\partial F}{\partial c}(x) \cdot b  - (A_{0,*} + \epsilon A_{1,*}) b \right) 
	\right\| , \qquad\text{for all $b\in\ell^K_0$ with $\|b\| \leq 1$}.
	\end{alignat*}
	Hence $ [ f_{*,\alpha} , f_{*,\omega}, f_{*,c}]$ is an upper bound on $A_{0,*}^{-1}  \pi_{\geq 2} \left( D F( x) - A \right)$, viewed as an operator from $\R^2 \times \ell^K_0$ to $\ell^1_0$. 
		We have thereby shown that $M$, as defined in~\eqref{e:defM}, is an upper bound on $\overline{A_0^{-1}  \left[D F( x ) - A \right] }$, which concludes the proof.	
\end{proof}


\begin{proposition}
	\label{prop:Z1a}
	Define
	\[
	f_{1,\alpha} :=  \dw +  \epsilon \frac{\dc  (2 + \dc) }{2} .
	\]
	Then for all $x = (\alpha,\omega,c) \in B_\epsilon(r,\rho)$ 
	\[
	f_{1,\alpha} \geq   \left|  \frac{\partial F_1}{\partial \alpha} (x) + i  \right| .
	\]
\end{proposition}

\begin{proof}
	We calculate
\begin{equation*}
	\frac{\partial F_1}{\partial \alpha} (x) + i =
	e^{- i \omega} + i  
	+ \epsilon \left( e^{i \omega} + e^{-2 i  \omega} \right) \pi_2 c
	+ \epsilon \pi_1 ([ U_{\omega} c] * c) ,
\end{equation*}
hence, using Lemma~\ref{lem:deltatheta},
\begin{equation*}
	\left|  \frac{\partial F_1}{\partial \alpha} (x) + i  \right|   \leq 
| e^{-i \omega } +i | + 2 \epsilon \frac{\dc}{2} + \epsilon \frac{1}{2} \dc^2  
	\leq  \dw +  \epsilon \frac{\dc  (2 + \dc)}{2} .
\end{equation*}
Here we have used that $|\pi_k a| \leq \frac{1}{2}\|a\|$ for $k=1,2$ and all $a \in \ell^1$.
\end{proof}



\begin{proposition}
		\label{prop:Z1w}
	Define
		\[
		f_{1,\omega} := 
		\da + \pp \dw + (\pp + \da) \frac{  \epsilon \dc}{2} ( 3 + \rho)  .
		\]
Then for all $x= (\alpha,\omega,c) \in B_\epsilon(r,\rho)$
	\[
	f_{1,\omega} \geq   \left|  \frac{\partial F_1}{\partial \omega } (x)- (i- \pp)  \right| .
	\]
\end{proposition}

\begin{proof}
We calculate 
\begin{alignat*}{1}
	\frac{\partial F_1}{\partial \omega} (x) - (i - \pp)  &=
	(i - i\alpha e^{- i \omega}) - (i - \pp) 
	+ \alpha \epsilon ( i e^{i \omega }- 2 e^{- 2 i \omega}) \pi_2 c -i \alpha \epsilon \pi_1  ([ K^{-1} U_\omega c ] *c )\\
	&= -i (\alpha - \pp) e^{-i \omega} - i \pp ( i + e^{-i\omega} )
	+ \alpha \epsilon ( i e^{i \omega }- 2 e^{- 2 i \omega}) \pi_2 c -i \alpha \epsilon \pi_1( [ K^{-1} U_\omega c ] *c) ,
\end{alignat*}
hence, using Lemma~\ref{lem:deltatheta} again,
\begin{equation*}
	\left|  \frac{\partial F_1}{\partial \omega} ( x)- (i- \pp)  \right|  \leq
	\da +  \pp \dw  + \frac{3}{2} \alpha \epsilon \dc +  \frac{1}{2} \alpha \epsilon \rho \dc  .
\end{equation*}
\end{proof}


\begin{proposition}
		\label{prop:Z1c}
	Define
	\[
	f_{1,c} := 
	\epsilon \left(  \da + \tfrac{3 \pi}{4} \dw +  (\pp + \da ) \dc   \right) .
	\]
	Then for all $x= (\alpha,\omega,c) \in B_\epsilon(r,\rho)$
	\[
	f_{1,c} \geq   \left|  \frac{\partial F_1 }{\partial c } (x) \cdot b -  \pp \epsilon (i-1) \pi_2 b \right|, 
	\qquad\text{for all $b\in\ell^K_0$ with $\|b\| \leq 1$}.
	\]
\end{proposition}

\begin{proof}
	We calculate
\begin{alignat*}{1}
		\frac{\partial F_1 }{\partial c } (x) \cdot b -  \pp \epsilon (i-1) \pi_2 b 
	& =
	\epsilon [ \alpha (e^{i \omega} + e^{-2i \omega})  - \pp(i -1)] \pi_2 b 
	+ \alpha \epsilon  \pi_1 \bigl(  [ U_{\omega} b ] * c + [ U_{\omega} c ]*b \bigr)  \\
	&= \epsilon [ (\alpha - \pp) (e^{i \omega} + e^{-2i \omega})  ]  \pi_2 b  +
	\epsilon  \pp [  (e^{i \omega} + e^{-2i \omega})  - (i -1)]  \pi_2 b  \nonumber \\
	& \qquad\quad + \alpha \epsilon \pi_1 \bigl(  [ U_{\omega} b ] * c + [ U_{\omega} c ]*b \bigr)  ,
	\end{alignat*}
hence, for $\|b\| \leq 1$,
\begin{equation*}
	\left| 
	\frac{\partial F_1 }{\partial c } (x) \cdot b -  \pp \epsilon (i-1) \pi_2 b 
	\right| 
	\leq
  \epsilon \left(  \da + \tfrac{\pi}{4}(\dw + 2 \dw ) +  (\pp + \da )  \dc   \right)  .
	\end{equation*}
\end{proof}


\begin{proposition}
		\label{prop:Zsa}
Define  
		\[
		f_{*,\alpha} := \frac{2}{\pi \sqrt{5}}\left( r_c +  2 \dw (  \dc^0 +  \epsilon )  + \epsilon \dc (4 + \dc )  \right)  .
		\]
	Then for all $x= (\alpha,\omega,c) \in B_\epsilon(r,\rho)$
	\[
		f_{*,\alpha} \geq  \left\| A_{0,*}^{-1} \pi_{\geq 2} \left(
		\frac{\partial F}{\partial \alpha}(x)  +   \epsilon  \tfrac{2 +4 i}{5}  \e_2  \right) \right\|  .
	\]
\end{proposition}

\begin{proof}
We note that $\epsilon  \tfrac{2 +4 i}{5} \e_2= \bc_\epsilon + \epsilon i \e_2$
and calculate 
	\[
\pi_{\geq 2} 
		\frac{\partial F}{\partial \alpha}(x)  +   \epsilon  \tfrac{2 +4 i}{5}  \e_2  =
	U_\omega (c-\bc_\epsilon) + (1+ e^{-2i\omega}) \bc_\epsilon + \epsilon  (e^{-i \omega} +i)\e_2 + \epsilon \pi_{\geq 2} L_\omega c + \epsilon  \pi_{\geq 2} ([ U_\omega c] * c ).
	\]
By using Proposition~\ref{p:severalnorms} and Lemma~\ref{lem:deltatheta}, we obtain the estimate
	\begin{alignat*}{1}
	\left\| A_{0,*}^{-1} \pi_{\geq 2} \left(
	\frac{\partial F}{\partial \alpha}(x)  +   \epsilon  \tfrac{2 +4 i}{5}  \e_2  \right) \right\|
	&\leq \| A_{0,*}^{-1} \| \left( r_c +  \dc^0 |1+ e^{-2i\omega} | +  2 \epsilon | e^{-i \omega } +i| + 4 \epsilon \dc + \epsilon \dc^2 \right) \\
	&\leq \frac{2}{\pi \sqrt{5}}\left( r_c +  2 \dw (  \dc^0 +  \epsilon )  + \epsilon \dc (4 + \dc )  \right) .
	\end{alignat*}
\end{proof}

\begin{proposition}
			\label{prop:Zsw}
Define 
		\begin{alignat}{1}
		f_{*,\omega} &:=
		\tfrac{5}{2 \pi}  (1+ \pp)  r_c 
		+
		\tfrac{2}{\sqrt{5}}  \epsilon \left( (1+\tfrac{4}{\sqrt{5}}) \dw + \tfrac{2}{\pi}\da \right) \nonumber
		+
		\tfrac{5}{2\pi} \da ( r_c + \dc) 
		\\& \qquad + 
		  \tfrac{2}{ \pi} \epsilon(\pp + 	 \da ) \left( \frac{1}{\sqrt{5}} ( \dc + r_c) + \frac{5}{4} \left(  \dc  + \tfrac{3}{2}r_c \right)  +  \frac{ \rho \dc   }{\sqrt{5}}\right) . \label{e:fstaromega}
		\end{alignat}
%
	Then for all $x= (\alpha,\omega,c) \in B_\epsilon(r,\rho)$
	\[
	f_{*,\omega} \geq  
	 \left\| A_{0,*}^{-1}  \pi_{\geq 2}\left( \frac{\partial F}{\partial \omega}(x) -  \epsilon \left[ \tfrac{4-3\pi}{10} + \tfrac{2(2 + \pi)}{5}i \right] \e_2  \right) \right\| .  
	\]
\end{proposition}

\begin{proof}
We note that 
$
\epsilon  \left[ \tfrac{4-3\pi}{10} + \tfrac{2(2 + \pi)}{5}i \right] \e_2
= i (2+\pi) \bc_\epsilon - \pp \epsilon \e_2
$
and calculate 
\begin{alignat*}{1}
\pi_{\geq 2} \frac{\partial F}{\partial \omega}(x) -  \epsilon \left[ \tfrac{4-3\pi}{10} + \tfrac{2(2 + \pi)}{5}i \right] \e_2   &=
		i K^{-1} ( I - \alpha U_{\omega} ) c  -
		 i  \alpha \epsilon e^{-i \omega} \e_2 + 
		\alpha \epsilon \pi_{\geq 2} L_{\omega}' c \\
		&\qquad - i \alpha \epsilon \pi_{\geq 2} ([ K^{-1} U_\omega c ] *c  ) - i K^{-1} ( I - \pp U_{\omega_0}) \bc_\epsilon +  \pp \epsilon \e_2\\
		&= i K^{-1} ( c - \bc_\epsilon) - \epsilon  (i  \alpha e^{-i \omega}  -  \pp ) \e_2\\
		&\qquad
		-i K^{-1} \left[  U_\omega \left(  \pp ( c - \bc_\epsilon) + ( \alpha - \pp) c  \right) + \left( U_\omega - U_{\omega_0} \right) \pp \bc_\epsilon) \right]
		\\
		&\qquad\qquad
			+ \alpha \epsilon \pi_{\geq 2} L_{\omega}' c 
			- i \alpha \epsilon \pi_{\geq 2} ([ K^{-1} U_\omega c ] *c ) .
\end{alignat*}	
Applying the operator $A_{0,*}^{-1}$ to this expression, we obtain (with $\hat{U}$ defined in~\eqref{e:defUhat})
%
\begin{alignat*}{1}
A_{0,*}^{-1}  \pi_{\geq 2}\left( \frac{\partial F}{\partial \omega}(x) -  \epsilon \left[ \tfrac{4-3\pi}{10} + \tfrac{2(2 + \pi)}{5}i \right] \e_2  \right) 
			&= \frac{2}{\pi} \hat{U}  ( c - \bc_\epsilon) 
			- \frac{2 \epsilon}{i \pi} \hat{U} K  (i  \alpha  e^{-i \omega}  -  \pp  ) \e_2 \\
			&\qquad
			-\frac{2}{\pi} \hat{U} \left[  U_\omega \left(  \alpha ( c - \bc_\epsilon) + ( \alpha - \pp) c  \right) \right]
		    \\	&\qquad\qquad
			-\frac{2}{\pi} \hat{U} \left( U_\omega - U_{\omega_0} \right) \pp \bc_\epsilon  
			\\
			&\qquad\qquad\qquad 
			+ \frac{2 \alpha \epsilon}{i \pi} \hat{U} K \pi_{\geq 2} \left(  L_{\omega}' c - i [ K^{-1} U_\omega c ] *c  \right)  .
	\end{alignat*}
We use the triangle inequality to estimate its norm, splitting it into the  five pieces:
	\begin{alignat*}{1}
	\left\|	\frac{2}{\pi} \hat{U}  ( c - \bc_\epsilon) \right\|
	&\leq \frac{2}{\pi} \frac{5}{4} r_c 
	= \frac{5}{2 \pi}  r_c \nonumber \\
	\left\|		- \frac{2 \epsilon}{i \pi} \hat{U} K  (i  \alpha  e^{-i \omega}  -  \pp  ) \e_2  \right\|
	&\leq   \frac{4 \epsilon }{\pi} \frac{1}{\sqrt{5}} \left( \pp \dw + \da \right)  
	=   \tfrac{2 \epsilon }{\sqrt{5}} \left(  \dw + \tfrac{2}{\pi}\da \right) \nonumber \\
	\left\| 	-\tfrac{2}{\pi} \hat{U} \left[  U_\omega \left(  \alpha ( c - \bc_\epsilon) + ( \alpha - \pp) c  \right) \right]  \right\|
	& \leq \tfrac{2}{\pi}  \tfrac{5}{4} \left(  (\pp + \da) r_c + \da \dc  \right)
	= \tfrac{5}{2\pi} \left( \pp r_c + \da ( r_c + \dc) \right) \nonumber \\
	\left\| -\tfrac{2}{\pi} \hat{U} 
	\left( U_\omega - U_{\omega_0} \right) \pp \bc_\epsilon   \right\| 
	&\leq  \tfrac{2}{\pi}  \tfrac{2}{\sqrt{5}} (2 \dw)  \pp \tfrac{2 \epsilon}{\sqrt{5}} 
	=   \tfrac{8 \epsilon}{5} \dw \nonumber \\
	\left\|\tfrac{2 \alpha \epsilon}{i \pi} \hat{U} K \pi_{\geq 2} \left(  L_{\omega}' c - i [ K^{-1} U_\omega c ] *c  \right)  \right\|
	&\leq \frac{2 \alpha \epsilon}{ \pi}  \left( \| \hat{U}  K \pi_{\geq 2} L_{\omega}' c  \| +  \frac{ \rho \dc   }{\sqrt{5}}\right) ,
	\end{alignat*} 
where we have used Proposition~\ref{p:severalnorms} and Lemma~\ref{lem:deltatheta}.
Finally, we estimate 
	\begin{alignat}{1}
	\left\| \hat{U} K \pi_{\geq 2} L_{\omega}' c \right\| &= 
	\left \| \hat{U}  K \pi_{\geq 2} \left(- i \sigma^+( e^{- i \omega} I + K^{-1} U_{\omega}) + i \sigma^-(e^{i \omega} I - K^{-1} U_{\omega}) \right) c \right\| \nonumber \\
	&\leq \left \| \hat{U}  K \pi_{\geq 2}( \sigma^+ + \sigma^- ) c  \right \| + 
	\left \| \hat{U} \pi_{\geq 2} K ( \sigma^+ + \sigma^- ) K^{-1}  U_{\omega} c  \right \|  \nonumber  \\
	&\leq  \frac{1}{\sqrt{5}} ( \| \sigma^+ c\| + \|\pi_{\geq 2}\sigma^- c\|) + \frac{5}{4} \left( \| K \sigma^+ K^{-1} \| \dc  +  \| \pi_{\geq 2} K \sigma^- K^{-1} \| r_c \right) \nonumber  \\
	&\leq  \frac{1}{\sqrt{5}} ( \dc + r_c) + \frac{5}{4} \left(  \dc  + \frac{3}{2}r_c \right)  .  \label{e:longestimate}
	\end{alignat}
Hence, with $f_{*,\omega}$ as defined in~\eqref{e:fstaromega},
%
it follows that 	
	\[
	 \left\| A_{0,*}^{-1}  \pi_{\geq 2}\left( \frac{\partial F}{\partial \omega}(x) -  \epsilon \left[ \tfrac{4-3\pi}{10} + \tfrac{2(2 + \pi)}{5}i \right] \e_2  \right) \right\| \leq 	f_{*,\omega} .
		\]

\end{proof}


\begin{proposition}
			\label{prop:Zsc}
Define 
	\[
	f_{*,c} :=	\left[ \frac{5}{2} \left( \frac{1}{2} + \frac{1}{\pi} \right) \dw +  \frac{\da }{\sqrt{5}} \right] 
 +\epsilon \left[ \frac{8}{\pi \sqrt{5}} \da + \left(  \frac{2}{\sqrt{5}}   + \frac{25}{8} \right) \dw + \frac{4 (\pp+\da)  \dc}{\pi \sqrt{5}} \right] .
	\]
	Then for all $x= (\alpha,\omega,c) \in B_\epsilon(r,\rho)$
	\[
	f_{*,c} \geq    \left\|  A_{0,*}^{-1} \pi_{\geq 2} \left( \frac{\partial F}{\partial c}(x) \cdot b  - (A_{0,*} + \epsilon A_{1,*}) b \right) 
	\right\| , \qquad\text{for all $b\in\ell^K_0$ with $\|b\| \leq 1$}.
	\]
\end{proposition}

\begin{proof}
We write $A_* := A_{0,*} + \epsilon A_{1,*}$ and calculate
\begin{alignat*}{1}
\frac{\partial F}{\partial c} (x) \cdot b - A_* b 
& = 
 \bigl[( i \omega K^{-1} + \alpha U_{\omega}) - ( i \pp K^{-1} + \pp U_{\omega_0}) \bigr] b  + \alpha \epsilon  L_\omega b  - \pp \epsilon L_{\omega_0} b 
 \\& \qquad
	+ \alpha \epsilon \left[ [ U_\omega b] * c + [U_{\omega} c ]*b \right]
 \\ & =
	\bigl[ i ( \omega - \pp ) K^{-1} + ( \alpha - \pp) U_{\omega} + \pp ( U_{\omega} - U_{\omega_0}) \bigr] b 
	\\ & \qquad 
	+ \epsilon \bigl[ ( \alpha - \pp) L_{\omega} + \pp ( L_{\omega} - L_{\omega_0}) \bigr] b 
+ \alpha \epsilon \left( [U_{\omega } b] * c + [ U_{\omega } c ]*b \right) . 
\end{alignat*}
Hence, for $\|b\| \leq 1$, 
\begin{alignat}{1}
	\left\| A_{0,*}^{-1} \pi_{\geq 2} \left(   \frac{\partial F}{\partial c} (x) \cdot b -  A_* b  \right)  \right\|
	 &\leq 
	 \dw  \| A_{0,*}^{-1} K^{-1} \| + \pp \da \| A_{0,*}^{-1}  \| + \pp \| A_{0,*}^{-1}  ( U_{\omega} - U_{\omega_0}) \| \nonumber  \\
	& \qquad
	 + \epsilon  \left[ 4  \da \|  A_{0,*}^{-1}   \| + \pp \| A_{0,*}^{-1} \pi_{\geq 2} ( L_{\omega} - L_{\omega_0}) \| 
	+  2 \alpha \dc \| A_{0,*}^{-1}  \|  \right]  \label{e:intermediate},
\end{alignat}	
where all norms should be interpreted as operators on $\ell^1_0$.
	Since
$\frac{\partial U_\omega}{\partial  \omega} = - i K^{-1} U_{\omega}$
and $ A_{0,*}^{-1} = \frac{2}{i\pi} \hat{U} K$, it follows from Proposition~\ref{p:severalnorms} that  
\begin{equation}\label{e:AUoUo0}
	\| A_{0,*}^{-1}  (U_{\omega} - U_{\omega_0})  \| \leq  \frac{2}{\pi}  \dw \| \hat{U} \| 
	= \frac{5}{2 \pi } \dw .
\end{equation}
Next, we compute
	\begin{alignat*}{1}
	L_{\omega} - L_{\omega_0}  &= \sigma^+ \left[ (e^{-i \omega} + i) I + (U_{\omega} - U_{\omega_0})\right] + \sigma^- \left[ (e^{i \omega} - i) I + (U_{\omega} - U_{\omega_0}) \right] \\
	&=  (e^{-i \omega} + i)  \sigma^+ - i e^{i\omega} (i+e^{-i \omega})\sigma^- 
	+ (\sigma^+ + \sigma^-) (U_{\omega} - U_{\omega_0}) .
\end{alignat*}
Analogous to~\eqref{e:longestimate} and~\eqref{e:AUoUo0} we infer that
\begin{equation*}
	\|  A_{0,*}^{-1} \pi_{\geq 2} ( L_{\omega} - L_{\omega_0} ) \| \leq  \frac{4}{\pi \sqrt{5}} |i+ e^{-i \omega} | + \frac{5}{\pi} \| \hat{U} \| \dw \\
	\leq  \frac{4 }{\pi \sqrt{5}}\dw    + \frac{25}{4 \pi}  \dw .
\end{equation*} 
Finally, by putting all estimates together and once again using Proposition~\ref{p:severalnorms}, it follows from~\eqref{e:intermediate} that 
	\begin{alignat*}{1}
\left\| A_{0,*}^{-1} \pi_{\geq 2} \left(   \frac{\partial F}{\partial c} (x) \cdot b -  A_* b  \right)  \right\|
	&\leq
	 \left[ \frac{5}{2} \left( \frac{1}{2} + \frac{1}{\pi} \right) \dw +  \frac{\da }{\sqrt{5}} \right] 
	 \\
	& \qquad  +\epsilon \left[ \frac{8}{\pi \sqrt{5}} \da + \left(  \frac{2}{\sqrt{5}}   + \frac{25}{8} \right) \dw + \frac{4 (\pp+\da)  \dc}{\pi \sqrt{5}} \right] .
	\end{alignat*}
\end{proof}




\section{Appendix: A priori bounds on periodic orbits}
\label{appendix:aprioribounds}

In order to isolate periodic orbits, we need to separate them from the trivial solution. In this appendix we prove some lower bounds on the size of periodic orbits. First we work in the original Fourier coordinates. Then we derive refined bounds in rescaled coordinates.

Recall that periodic orbits of Wright's equation corresponds to  zeros of 
$G(\alpha,\omega,\c)=0$, as defined in~\eqref{e:defG}. Clearly $G(\alpha,\omega,0)=0$ for all frequencies $\omega>0$ and parameter values $\alpha>0$. There are bifurcations from this trivial solution for $\alpha=\alpha_n:=\pp(4n+1)$ for all $n\geq 0$. The corresponding natural frequency is $\omega=\alpha_n$, but there are bifurcations for any $\omega = \alpha_n/ \tilde{n}$ with $\tilde{n} \in \N$ as well, which are essentially copies of the primary bifurcation. The following proposition quantifies that away from these bifurcation points the trivial solution is isolated.

\begin{proposition}
	\label{prop:zeroneighborhood2}
	Suppose $G(\alpha,\omega,\c)=0$ for some $\alpha,\omega >0$. 
Then either $\c \equiv 0$ or 
	\begin{equation}\label{e:minoverk}
	\| \c \| \geq \min_{k \in \N}  \sqrt{ \left(1-k \,\frac{\omega}{\alpha} \right)^2 + 2  k \, \frac{\omega}{\alpha} \bigl( 1- \sin k \omega \bigr)} .
	\end{equation}	
\end{proposition}

\begin{proof}
We fix $\alpha,\omega>0$ and define 
\[
  \beta_1 := \min_{k \in \N} \, (\alpha-k\omega)^2 + 2 \alpha k \omega ( 1- \sin k \omega ).
\]
If $\beta_1=0$ then there is nothing to prove. From now on we assume that $\beta_1>0$. 
We recall that 
\[ G(\alpha,\omega,\c) = (i \omega K^{-1} + \alpha U_{\omega}) \c + \alpha \left[U_\omega \, \c \right] * \c .
\]
We note that $i \omega K^{-1} + \alpha U_{\omega}$ is invertible, since
for any $k \in \N$
\begin{alignat*}{1}
  |ik\omega + \alpha e^{-i k \omega}|^2
  & =  (\alpha \cos k \omega )^2 + ( \omega - \alpha \sin  k \omega )^2  \\
  &  = (k \omega)^2 + \alpha^2 - 2 \alpha  k \omega \sin k \omega   \\
&= (\alpha - k\omega)^2 + 2 \alpha k \omega ( 1 - \sin k \omega)  \\
&\geq \beta_1 >0.
\end{alignat*}
We may thus rewrite $G(\alpha,\omega,\c) = 0$ as 
\begin{equation}\label{e:quadratic}
 \c = - \alpha  (i \omega K^{-1} + \alpha U_{\omega})^{-1} ( \left[U_\omega \, \c \right] * \c ) .
\end{equation}
Since $\| ( \omega K^{-1} + \alpha U_{\omega})^{-1} \| = \beta_1^{-1/2}$
and $\| \left[U_\omega \, \c \right] * \c \| \leq \| \c \|^2 $,
we infer from~\eqref{e:quadratic} that
\[
  \| \c \| \leq  \alpha \beta_1^{-1/2} \|\c\|^2.
\]
We conclude that either $\c \equiv 0$ or $\| \c \| \geq \beta_1^{1/2} /\alpha $.
\end{proof}

\begin{proposition}
	\label{prop:G1Minimizer}
	Suppose that $ \omega \geq 1.1$ and $ \alpha \in (0,2]$. Define 
	\begin{equation}
g_k(\omega,\alpha) =		\left(1-k \,\tfrac{\omega}{\alpha} \right)^2 + 2  k \, \tfrac{\omega}{\alpha} \bigl( 1- \sin k \omega \bigr) .
	\end{equation}
Then $ g_1 <g_k$ for all $ k \geq 2$. 
\end{proposition}

\begin{proof}
	
This is equivalent to showing that 
\[
(1- \tfrac{\omega}{\alpha})^2 + 2 \tfrac{\omega}{\alpha} ( 1 - \sin \omega ) 
<
(1- k \tfrac{\omega}{\alpha})^2 + 2k  \tfrac{\omega}{\alpha} ( 1 - \sin k\omega )  
\qquad\text{for } k\geq 2.
\]
Making the substitution $ x = \tfrac{\omega}{\alpha}$, we can simplify this to the equivalent inequality 
\[
 (k^2-1) x  + 2  \sin \omega - 2k  \sin k\omega >0 .
\]
Since $\alpha \leq 2$, we have $x\geq \omega/2$. Hence it suffices to prove that
%
%
\begin{equation}
  h_k(\omega) := \frac{k^2-1}{2} \omega + 2 \sin \omega  - 2 k \sin k\omega  >0 
  \qquad\text{for all } k\geq 2.
\label{eq:GminRedux}
\end{equation}
We first consider $k=2$. It is clear that $h_2(\omega) > 0$ for $\omega > 4$.
We note that $h_2$ has a simple zero at $ \omega \approx 1.07146$ and it is easy to check using interval arithmetic that $h_2(\omega)$ is positive for $\omega \in [1.1,4]$. Hence $h_2(\omega) > 0$ for all $\omega \geq 1.1$.

For $k=3$ and $k=4$ we can repeat a similar argument. For $k \geq 5$ it is immediate that $h_k(\omega) > \frac{k^2-1}{2}-2-2k \geq 0$ for $\omega > 1$.
%
%
%
%
\end{proof}

As discussed in Section~\ref{s:preliminaries},
the function $G(\alpha,\omega,\c)$ gets replaced by $\tF_\epsilon(\alpha,\omega,\tc)$ in rescaled coordinates. 
In these coordinates we derive a result analogous to Proposition~\ref{prop:zeroneighborhood2} below, see Lemma~\ref{lem:thecone}.
First we bound the inverse of the operator $\B \in B(\ell^1_0)$ defined by
\[
  \B:= i \frac{\omega}{\alpha} I +  U_{\omega} K +  \epsilon L_{\omega} K,
\]
where $K$, $U_{\omega}$ and $L_\omega$ have been introduced in Section~\ref{s:preliminaries}.
\begin{lemma}\label{lem:gamma}
	Let $\epsilon \geq 0$ and $\alpha,\omega>0$. Let
\[
  \gamma := 	  \frac{1}{2} +
  \epsilon \left( \frac{2}{3} + \max\left\{  \frac{\sqrt{2 - 2 \sin (\omega-\pp) }}{2} ,\frac{2}{3} \right\} \right).
\] 
If $\gamma < \omega / \alpha$ then the operator $\B$ is invertible and the inverse is bounded by
\[
	  \| \B^{-1} \| \leq \frac{1}{\frac{\omega}{\alpha}- \gamma}.
\]
\end{lemma} 

\begin{proof}
Writing 
\[
  \B= i \frac{\omega}{\alpha} \left( I + \frac{\alpha}{i\omega} \left( U_{\omega} +  \epsilon L_{\omega} \right) K \right)
\]
and using a (formal) Neumann series argument, we obtain
\[
  \| \B^{-1} \| \leq \frac{\alpha}{\omega}
   \sum_{n=0}^\infty \left( \frac{ \alpha}{ \omega} \right)^n \|(U_{\omega} + \epsilon L_{\omega}) K\|^n
   \leq \frac{\frac{\alpha}{\omega} }{1- \frac{ \alpha}{ \omega} \|(U_{\omega} + \epsilon L_{\omega}) K\|} 
   = \frac{1}{\frac{\omega}{\alpha}- \|(U_{\omega} + \epsilon L_{\omega}) K\|} .
\]
It remains to prove the estimate $\|(U_{\omega} + \epsilon L_{\omega}) K\| \leq \gamma$.
Then, in particular, for $\gamma < \omega/\alpha$ the formal argument is rigorous.

Recalling that $L_\omega= \sigma^+ (e^{-i\omega} I  + U_\omega) + \sigma^- (e^{i\omega} I  + U_\omega)$, we use the triangle inequality
\[
\|(U_{\omega} + \epsilon L_{\omega}) K\| 
\leq \| U_\omega K \|  + \epsilon \| \sigma^+ (e^{-i\omega} I  + U_\omega) K \| + \epsilon \| \sigma^- (e^{i\omega} I  + U_\omega) K \|,
\] 
and estimate each term separately as operator on $\ell^1_0$. 
We recall the formula~\eqref{e:operatornorm} for the operator norm.
Using that $\| K \tc \| \leq \frac{1}{2} \|\tc\|$ for all $\tc \in \ell^1_0$,
the first term is bounded by $\| U_\omega K \| \leq \frac{1}{2}$.
Since $\sigma^-$ shifts the sequence to the left and we consider the operators acting on $\ell^1_0$, we obtain $\| \sigma^- (e^{i\omega} I  + U_\omega) K \| \leq \frac{2}{3}$.
For the final term, $\| \sigma^+ (e^{-i\omega} I  + U_\omega) K \|$, 
to obtain a slightly more refined estimate,
we first consider the action of $\sigma^+ (e^{-i\omega} I  + U_\omega) K$ 
on $\e_2$. We observe that 
\[ 
  | e^{- i \omega} + e^{-2 i \omega}| = \sqrt{2 - 2 \sin (\omega-\pp) },
\]
hence $\| \sigma^+ (e^{-i\omega} I  + U_\omega) K \e_2 \| \leq 
\sqrt{2 - 2 \sin (\omega-\pp) } $, 
leading to
\[
 \| \sigma^+ (e^{-i\omega} I  + U_\omega) K \| 
 \leq \max\left\{  \frac{\sqrt{2 - 2 \sin (\omega-\pp) }}{2} ,\frac{2}{3} \right\}.
\]
We conclude that 
\[
\|(U_{\omega} + \epsilon L_{\omega}) K\| 
\leq 
  \frac{1}{2} +
  \epsilon \left( \frac{2}{3} + \max\left\{  \frac{\sqrt{2 - 2 \sin (\omega-\pp) }}{2} ,\frac{2}{3} \right\} \right).
\]
\end{proof}

\begin{lemma}\label{lem:thecone}
	Fix $ \epsilon \geq 0$, $\alpha,\omega>0$.
	Assume that $\B$ is invertible.
	Let $b_0$ be a bound on $\| \B^{-1} \|$.
	Define 
	\[
	z^{\pm} = b_0^{-1} \pm \sqrt{b_0^{-2}-  2\epsilon^2 } .
	\]
	Let $ \tc \in \ell^1_0$ be such that $\tF_\epsilon(\alpha, \omega,\tc) = 0$, then either $ \|\tc\| \leq  z^-$ or $  \|\tc\| \geq z^+ $. 
	\noindent
	Additionally, $ \| K^{-1} \tc \| \leq b_0 (2\epsilon^2+ \|\tc\|^2)$.
\end{lemma}

\begin{proof}
	If  $ \tF_\epsilon( \alpha, \omega, \tc) =0$ then it follows that the equations $\pi_c \tF_\epsilon=0$ can be rearranged as 
\begin{equation}\label{e:eBc2}
  \tc = - K \B^{-1} (  \epsilon^2  e^{- i \omega} \e_2 + \ [ U_{\omega} \tc ] * \tc ) .
\end{equation}
	Taking norms, and using that $\| K \tc \| \leq \frac{1}{2} \| \tc\|$ for all $\tc\in \ell^1_0$, we obtain 
\begin{equation}\label{e:quadineq}
\|\tc \|  \leq  \frac{1}{2} \| B^{-1}\| \left( \epsilon^2 \|\e_2\|  + \| [ U_{\omega} \tc ] * \tc \| \right)
\leq \frac{1}{2} b_0 \left( 2 \epsilon^2  + \| \tc \|^2 \right).
\end{equation}
The quadratic $x^2 - 2 b_0^{-1} x +   2\epsilon^2 $
has two zeros $z^+$ and $ z^-$ given by
	\[
	z^{\pm} = b_0^{-1} \pm \sqrt{b_0^{-2}-  2\epsilon^2 } .
	\]
The inequality~\eqref{e:quadineq} thus implies that either $ \|\tc\| \leq z^-$ or $ \|\tc\| \geq z^+$.

Furthermore, it follows from~\eqref{e:eBc2} that $\| K^{-1} \tc \| \leq  \| \B^{-1} \| \, (2 \epsilon^2 + \|\tc \|^2) \leq b_0 (2 \epsilon^2+ \|\tc\|^2)$.
\end{proof}

In practice we use the bound $\| \B^{-1} \| \leq b_*^{-1}$,
where
\[
  b_*(\epsilon) := \frac{\omega}{\alpha} - \frac{1}{2} - \epsilon  \left(\frac{2}{3}+ \frac{1}{2}\sqrt{2 + 2 |\omega-\pp| } \right).
\]
When doing so, we will refer to $z^\pm$ as $ z^\pm_*$. 
Additionally, we will need the following monotonicity property.
\begin{lemma}
	\label{lem:ZminusBound}
	Fix $\alpha, \omega, \epsilon_0 >0$ and assume that $ \epsilon_0 \leq b_*(\epsilon_0) /\sqrt{2}$.
	Define 
\[
  z_*^- (\epsilon):= b_*(\epsilon)-\sqrt{(b_*(\epsilon))^2 -2 \epsilon^2}.
 \]
Let $C_0 := \frac{z_*^-(\epsilon_0)}{\epsilon_0}$.
Then
	\begin{equation}
	 z_*^-(\epsilon) \leq C_0 \epsilon
	 \qquad \text{for all } 0 \leq \epsilon \leq \epsilon_0.
	 \label{eq:ConeLemma}
	\end{equation}
\end{lemma}

\begin{proof}
	Let $x:=\sqrt{2} \epsilon/b_*(\epsilon) \geq 0$. Clearly $\frac{d}{d\epsilon} x >0$.
It thus suffices to observe that
\[
  \frac{z_*^- (\epsilon)}{\epsilon} = \sqrt{2}\,\frac{1 - \sqrt{1-x^2}}{x}
\]	
is increasing for $x \in [0,1]$. 
%
%
\end{proof}




\section{Appendix: Implicit Differentiation}
\label{sec:Appendix_Implicit_Diff}

%
%

We will approximate
\[
\frac{\partial F}{\partial  \epsilon}( x ) = \alpha e^{-i \omega} \e_2 + \alpha L_\omega c + \alpha [ U_\omega c] * c 
\]
by
\begin{alignat}{1}
\Gamma & := \pp \tfrac{3i -1}{5} \epsilon \e_1 - \pp i \e_2 - \pp \tfrac{3+i}{5} \epsilon \e_3 \label{e:defGamma}  \\
&=   -\pp i \e_2 + \pp L_{\omega_0} \bc_\epsilon , \label{e:defGamma2}
\end{alignat}
which has been chosen so that $\frac{\partial F}{\partial  \epsilon} (\pp,\pp,\bc_\epsilon) - \Gamma = \cO(\epsilon^2)$.
\begin{lemma}
	\label{lem:ImplicitApprox}
When we write 
$A^\dagger \Gamma = (\alpha', \omega', c') \in \R^2 \times \ell^K_0$, then 
\begin{alignat*}{1}
		\alpha ' &= - \tfrac{2}{5} ( \tfrac{3 \pi}{2}-1) \epsilon ,\\
		\omega ' &= \tfrac{2}{5} \epsilon , \\
		c '	 &= \left[ (\tfrac{1+2i}{5}) - 
		\epsilon^2 \tfrac{9 }{250} (7-i)
		 \right] \e_2 + \epsilon \tfrac{3i-1}{10} \, \e_3 .
	\end{alignat*}
\end{lemma}

\begin{proof}
	First we calculate the $ \alpha$ and $ \omega$ components of the image of $ A^\dagger $:	
	\begin{alignat}{1}
	\pi_{\alpha,\omega} A^\dagger  &= A_{0,1}^{-1} i_\C^{-1} \pi_1 [ I - \epsilon A_1 A_0^{-1}] ] \nonumber \\
	&=   A_{0,1}^{-1} i_\C^{-1} \pi_1 [ I - \epsilon \pp L_{\omega_0}  A_{0,*}^{-1}] \nonumber \\
	&=   A_{0,1}^{-1} i_\C^{-1} \pi_1 [ I - \epsilon \sigma^{-} ( iI + U_{\omega_0} ) ( i K^{-1} + U_{\omega_0})^{-1}] \nonumber \\ 
	&=  A_{0,1}^{-1} i_\C^{-1} [ \pi_1- \epsilon ( \tfrac{3 +i}{5}) \pi_2 ].
	\label{e:piaoAdag}
	\end{alignat}
Here we have used projections $\pi_k a = a_k$ for $a=\{a_k\}_{k \geq 1} \in \ell^1$.
	We now calculate the $\alpha $ and $\omega$ components of $ A^\dagger \Gamma$. 
	It follows from~\eqref{e:defGamma} and~\eqref{e:piaoAdag} that
	\begin{alignat*}{1}
	\pi_{\alpha,\omega} A^\dagger \Gamma  
	&= A_{0,1}^{-1} i_\C^{-1} \left[ \pp \tfrac{3i -1}{5} \epsilon +  \pp  \tfrac{3 +i}{5}  i \epsilon \right]
	\\
	&= \tfrac{\pi \epsilon}{5} A_{0,1}^{-1} i_\C^{-1} ( 3 i -1) 
	 \\
	&= - \frac{2 \epsilon}{5}
	\left(
	\begin{array}{c}
	\tfrac{3\pi}{2}-1 \\
	-1
	\end{array}
	\right) .
	\end{alignat*}
	
	We now calculate 
\begin{equation}\label{e:picAdagGamma}
	\pi_c A^{\dagger} \Gamma  = A_{0,*}^{-1}  \pi_{\geq 2} [ I - \epsilon A_1 A_0^{-1} ]\Gamma ,
\end{equation}
where
$A_1 A_0^{-1}$ decomposes as
\begin{equation}\label{e:A1A0decomposition}
  A_1 A_0^{-1} = 
  \e_2 [i_\C A_{1,2}A_{0,1}^{-1} i_\C^{-1} \pi_1] 
  +A_{1,*} A_{0,*}^{-1}  \pi_{\geq 2} . 
\end{equation}
We first calculate  
\begin{alignat}{1}
	A_{0,*}^{-1} \pi_{\geq 2} \Gamma &= \tfrac{2}{ \pi } ( i K^{-1} + U_{\omega_0} )^{-1}  [ - \pp i \e_2 - \pp \tfrac{3+i}{5}\epsilon \e_3 ] \nonumber\\
	&= -(2i-1)^{-1} \e_2 - (3i+i)^{-1} \tfrac{3+i}{5} \epsilon \e_3 \nonumber\\
	&= \tfrac{1+2i}{5} \e_2 + \epsilon \tfrac{ 3 i-1}{20} \e_3  \label{e:A0piGamma}. 
\end{alignat}

Since $\Gamma$ has three nonzero components only, we next compute the action of $A_{0,*}^{-1}  \pi_{\geq 2} A_1 A_0^{-1} $ on each of these.
Taking into account the decomposition~\eqref{e:A1A0decomposition},
we first compute its action on $\lambda \e_1$ for $\lambda \in \C$.
After a straightforward but tedious calculation we obtain
\begin{alignat*}{1}
A_{0,*}^{-1}  \pi_{\geq 2} A_1 A_0^{-1}  \lambda \e_1
&= 
 [i_\C A_{1,2}A_{0,1}^{-1} i_\C^{-1} \lambda ] A_{0,*}^{-1} \e_2
\\
&= -\tfrac{2}{25\pi} \bigl[ (11+2i) \text{Re} \lambda  + (-6+8i) \text{Im} \lambda \bigr]  \e_2.
\end{alignat*}
Next, we compute the action of $A_{0,*}^{-1}  \pi_{\geq 2} A_1 A_0^{-1} $
on  $\e_k$ for $k=2,3$:
\begin{alignat*}{1}
A_{0,*}^{-1}  \pi_{\geq 2} A_1 A_0^{-1} \e_2
&=
 A_{0,*}^{-1} [ \pp \sigma^+ (  e^{-i\omega_0} I  + U_{\omega_0}) ]A_{0,*}^{-1} \e_2 
\\& 
=  \tfrac{2}{\pi} \tfrac{3+i}{20} \e_3  ,
  \\
A_{0,*}^{-1}  \pi_{\geq 2} A_1 A_0^{-1} \e_3
&=
 A_{0,*}^{-1} [ \pp \sigma^- ( e^{i\omega_0} I + U_{\omega_0}) ]A_{0,*}^{-1} \e_3  \\
& = 
-\tfrac{2}{\pi} \tfrac{1+2i}{10}  \e_2 ,
\end{alignat*}
where we have used that $(e^{-i\omega_0}I + U_{\omega_0}) \e_3$ vanishes.
Hence, by using the explicit expression~\eqref{e:defGamma} for~$\Gamma$ we obtain
\begin{equation}\label{e:actionek}
- \epsilon A_{0,*}^{-1}  \pi_{\geq 2}  A_1 A_0^{-1} \Gamma 
= 
 -\epsilon^2 \frac{29-22i}{125} \e_2 + \epsilon \frac{3i-1}{20} \e_3 -\epsilon^2 \frac{1+7i}{50} \e_2.
\end{equation}
Finally, combining~\eqref{e:picAdagGamma}, \eqref{e:A0piGamma} and~\eqref{e:actionek} completes the proof.
%
%
%
%
\end{proof}
 
\begin{lemma}
	\label{lem:ImplicitLast}
Let 
	\begin{alignat}{1}
	\hat{f}_{\epsilon,1} &:= \tfrac{1}{2} \dc^0 \left(   \sqrt{2} \da  +3  \dw ( \pp + \da ) \right)  +  r_c  ( \pp + \da) 
	\left(1 + \dc^0  + \tfrac{1}{2} r_c \right)  , \label{e:feps1} \\
	\hat{f}_{\epsilon,c} &:= 
	\tfrac{2}{\pi \sqrt{5}} \left[ 
	 2 \left( \da + \pp \dw \right) + \dc^0  [ \sqrt{2} \da + 3 \dw (\pp+\da) ]
	+(\pp+\da) ( 4 r_c + \dc^2)
	\right] .\label{e:fepsc}
	\end{alignat}
	Then  the vector 
	$
	[	(1+\tfrac{4}{\pi^2})^{1/2}\hat{f}_{\epsilon,1},\tfrac{2}{\pi} \hat{f}_{\epsilon,1}, \hat{f}_{\epsilon,c}]^{T}
	$
	is an upper bound on $A_0^{-1}  ( \tfrac{\partial F}{\partial  \epsilon} (x) -\Gamma )$ for any $ x \in B_\epsilon ( r,\rho)$.
\end{lemma}

\begin{proof}
	The $\alpha$- and $\omega$-component of  $A_0^{-1}  ( \tfrac{\partial F}{\partial  \epsilon} (x) -\Gamma )$ are given by $ A_{0,1}^{-1} i_\C^{-1} \pi_1 [  \tfrac{\partial F}{\partial  \epsilon} (x) -\Gamma ]$.
	If we can show that  $| \pi_1 [  \tfrac{\partial F}{\partial  \epsilon} (x) -\Gamma ]   |  \leq \hat{f}_{\epsilon,1}$, then it follows from the explicit expression for $A_{0,1}^{-1}$ that
	 $[ (1+\tfrac{4}{\pi^2})^{1/2}\hat{f}_{\epsilon,1} ,\tfrac{2}{\pi} \hat{f}_{\epsilon,1} ]^T$ 
	is an upper bound on 
	 $ \pi_{\alpha,\omega} A_0^{-1}  ( \tfrac{\partial F}{\partial  \epsilon} (x) -\Gamma )$.
	Let us write $ c = \bce +h_c$ for some $ h_c \in \ell^1_0$ with $ \|h_c\| \leq r_c$.  Recalling  \eqref{e:defGamma2}, we obtain
	\begin{alignat*}{1}
	\pi_1[  \tfrac{\partial F}{\partial  \epsilon} (x) -\Gamma ] &=   \pi_1 \bigl[ \alpha L_\omega c + \alpha [ U_\omega c ] * c - \pp L_{\omega_0} \bce  \bigr]  \\
	&= \pi_1 \bigl[ \alpha \sigma^{-}( e^{i \omega} + e^{-2 i \omega} ) \bce- \pp \sigma^{-}(i-1) \bce \bigr] 
	+ \pi_1  \bigl[ \alpha \sigma^{-}( e^{i \omega} + e^{-2 i \omega} )h_c + \alpha [ U_\omega c] * c  \bigr] \\
	& = \pi_1 \bigl[ (\alpha - \pp) (i-1) \bce  + \alpha ( e^{i \omega} -i + e^{-2 i \omega} +1)\bce \bigr] 
	\\ & \hspace*{6.55cm}
	+ \pi_1  \bigl[ \alpha \sigma^{-}( e^{i \omega} + e^{-2 i \omega} )h_c + \alpha [ U_\omega c] * c  \bigr] .
	\end{alignat*}
We note that
\[
  \pi_1 ( [ U_\omega c] * c ) =  \pi_1 ([ U_\omega (\bce+h_c)] * (\bce+h_c) )
  =  \pi_1 ( [ U_\omega \bce] * h_c + [ U_\omega h_c] * \bce +
   [ U_\omega h_c] * h_c ).
 \]
Hence, using Lemma~\ref{lem:deltatheta} we obtain the estimate 
%
		\begin{equation*}
		\bigl| \pi_1 [  \tfrac{\partial F}{\partial  \epsilon} (x) -\Gamma ] \bigr|
		 \leq \tfrac{1}{2} \dc^0 \left(   \sqrt{2} \da  +3  \dw ( \pp + \da ) \right)  +  r_c  ( \pp + \da) 
		 \left(1 + \dc^0  + \tfrac{1}{2} r_c \right) .
		\end{equation*}
We thus find that
	$ | \pi_1 [  \tfrac{\partial F}{\partial  \epsilon} (x) -\Gamma ]|   \leq \hat{f}_{\epsilon,1} $, with $\hat{f}_{\epsilon,1}$ defined in~\eqref{e:feps1}.
	
	The $c$-component of  $A_0^{-1}  ( \tfrac{\partial F}{\partial  \epsilon} (x) -\Gamma )$ is given by $ A_{0,*}^{-1}  \pi_{\geq 2} [  \tfrac{\partial F}{\partial  \epsilon} (x) -\Gamma ]$.
	We will use the estimate $ \| A_{0,*}^{-1}\| \leq \frac{2}{\pi \sqrt{5}}$, so that it remains to determine a bound on $ \| \pi_{\geq 2} [  \tfrac{\partial F}{\partial  \epsilon} (x) -\Gamma ]\|$.  
Using~\eqref{e:defGamma2} we compute
	\begin{equation*}
	\pi_{\geq 2} [  \tfrac{\partial F}{\partial  \epsilon} (x) -\Gamma ] =  \alpha e^{-i \omega} \e_2  +  \pp i \e_2 + 
	\pi_{\geq 2} \bigl( \alpha L_{\omega} \bce - \pp L_{\omega_0} \bce + \alpha L_{\omega} h_c + \alpha [U_\omega c] * c  \bigr).
	\end{equation*}
We split the right hand side into three parts, which we estimate separately. First
		\[
		\left\| \pi_2 \bigl[ \alpha L_{\omega} h_c + \alpha [U_\omega c] * c  \bigr] \right\| \leq  (\pp+ \da)  ( 4 r_c + \dc^2).
		\]	
Next, we calculate  
	\begin{alignat*}{1}
\pi_{\geq 2}	\left[\alpha L_{\omega} \bce - \pp L_{\omega_0} \bce \right] &= \alpha  \sigma^+ (e^{-i \omega} +e^{-2 i \omega}) \bce - \pp \sigma^+ (-i-1) \bce \\
	&= \left[ (\alpha - \pp) ( -i-1)\tfrac{2-i}{5}\epsilon  + \alpha ( e^{-i\omega} +e^{-2 i \omega} -(i+1)) \tfrac{2-i}{5}\epsilon \right] \e_3 ,
\end{alignat*}
hence 
	\begin{equation*}
	\left\|	\pi_{\geq 2}	\left[\alpha L_{\omega} \bce - \pp L_{\omega_0} \bce \right] \right\| \leq 
	\dc^0  [ \sqrt{2} \da + 3 \dw (\pp+\da) ] .
	\end{equation*}
Finally, we estimate 
	\begin{equation*}
	\left\| ( \alpha e^{-i\omega}  + \pp i) \e_2 \right\| = 
	2 \left| ( \alpha - \pp) e^{-i \omega} + \pp ( e^{-i \omega} +i) \right|
	\leq 2 \left( \da + \pp \dw \right).
	\end{equation*}
Collecting all estimates, we thus find that
	$ \| \pi_c A_{0}^{-1}  [  \tfrac{\partial F}{\partial  \epsilon} (x) -\Gamma ] \|   \leq \hat{f}_{\epsilon,c} $, with $\hat{f}_{\epsilon,c}$ defined in~\eqref{e:fepsc}.
	%
	%
	%
	%
\end{proof}

Recall that $\II$ is used to denote the $ 3 \times 3 $ identity matrix.
\begin{corollary}
	\label{cor:QUpperBound}
Let $\overline{A_0^{-1} A_1} $ be defined in Proposition~\ref{prop:A0A1}.	
The vector 
	\[
	(\II+\epsilon \overline{A_0^{-1} A_1} ) \cdot  [(1+\tfrac{4}{\pi^2})^{1/2}\hat{f}_{\epsilon,1},\tfrac{2}{\pi} \hat{f}_{\epsilon,1}, \hat{f}_{\epsilon,c}]^{T}
	\]
	is an upper bound on $ A^\dagger ( \tfrac{\partial F}{\partial  \epsilon} (x) -\Gamma ) $ for any $x \in B_\epsilon(r,\rho)$.
\end{corollary}
\begin{sloppypar}
\begin{proof}
	From Lemma \ref{lem:ImplicitLast} it follows that 
$
[(1+\tfrac{4}{\pi^2})^{1/2}\hat{f}_{\epsilon,1},\tfrac{2}{\pi} \hat{f}_{\epsilon,1}, \hat{f}_{\epsilon,c}]^T
$
	is an upper bound on $A_0^{-1}  ( \tfrac{\partial F}{\partial  \epsilon} (x) -\Gamma )$. 
Since 
	$A^\dagger = (I-\epsilon A_0^{-1}A_1) A_0^{-1}$ and 
	$\II+\epsilon \overline{A_0^{-1} A_1}$ is an upper bound on $I-\epsilon A_0^{-1}A_1$, the result follows from  Lemma~\ref{lem:ImplicitLast}. 
\end{proof}
\end{sloppypar}

We combine Lemmas~\ref{lem:ImplicitApprox} and~\ref{lem:ImplicitLast}
into an upper bound on $A^{\dagger} \frac{\partial F}{\partial  \epsilon}(\hat{x}_\epsilon)$.

\begin{lemma}\label{lem:Qeps}
%
	Define $\QQ_\epsilon^0 , \QQ_\epsilon \in \R_+^3$ as follows:
\begin{alignat}{1}
\QQ_\epsilon^0 &:= 
\left[ \frac{2}{5}\left(\frac{3\pi}{2}-1 \right)  \epsilon,
\frac{2}{5} \epsilon , 
\frac{2}{\sqrt{5}} + \frac{2}{\sqrt{10}}\epsilon  +
\frac{18}{5\sqrt{50}}  \epsilon^2
\right]^T , \nonumber \\
\QQ_\epsilon &:= \QQ_\epsilon^0  + 
(\II+\epsilon \overline{A_0^{-1} A_1} ) \cdot  \bigl[(1+\tfrac{4}{\pi^2})^{1/2}\hat{f}_{\epsilon,1},\tfrac{2}{\pi} \hat{f}_{\epsilon,1}, \hat{f}_{\epsilon,c}\bigr]^{T} . \label{e:defQeps}
\end{alignat}
Then the vector $\QQ_\epsilon \in \R^3_+$ is an upper bound  on $A^{\dagger} \frac{\partial F}{\partial  \epsilon}(x)$ 
for any  $x \in B_\epsilon(r,\rho)$.
\end{lemma}
\begin{proof}
It follows from Lemma \ref{lem:ImplicitApprox} that the vector 
$\QQ_\epsilon^0$
is an upper bound on $A^{\dagger} \Gamma$
(for example, the third component of $\QQ_\epsilon^0$ is a bound on $\|c'\|$).
It follows from Corollary~\ref{cor:QUpperBound} that 
\[
	(\II+\epsilon \overline{A_0^{-1} A_1} ) \cdot  [(1+\tfrac{4}{\pi^2})^{1/2}\hat{f}_{\epsilon,1},\tfrac{2}{\pi} \hat{f}_{\epsilon,1}, \hat{f}_{\epsilon,c}]^{T}
\]
 is an upper bound on $ A^\dagger ( \tfrac{\partial F}{\partial  \epsilon} (x) -\Gamma ) $.
We conclude from the triangle inequality that $\QQ_\epsilon $ is an upper bound on $A^\dagger  \tfrac{\partial F}{\partial  \epsilon} (x)$. 
\end{proof}
 
Finally, we prove the bounds needed to control the derivative $\frac{d}{d\epsilon} \hat{\alpha}_\epsilon$ in Section~\ref{s:Jones}
(in particular the implicit differentiation argument in Theorem~\ref{thm:NoFold}).



\begin{lemma}
		\label{lem:Meps}
		Fix $ \epsilon_0 > 0 , \rr \in \R^3_+ $ and $\rho >0$ as in the hypothesis of Proposition~\ref{prop:TightEstimate}. 
Let $ 0 < \epsilon \leq \epsilon_0$ and let $ \hat{x}_\epsilon \in B_\epsilon(\epsilon^2  \rr,\rho)$ denote the unique solution to $F(x) = 0$. 
Recall the definitions of  
$\ZZ_\epsilon \in  \emph{Mat}(\R_+^3 , \R_+^3)$
and $\QQ_\epsilon \in  \R_+^3$ in Equations~\eqref{e:defZeps} and~\eqref{e:defQeps}.  
Define 
		\begin{alignat*}{1}
		M_\epsilon &:= \frac{1}{\epsilon^2} \left(	(\II+\epsilon \overline{A_0^{-1} A_1} ) \cdot  \bigl[(1+\tfrac{4}{\pi^2})^{1/2}\hat{f}_{\epsilon,1},\tfrac{2}{\pi} \hat{f}_{\epsilon,1}, \hat{f}_{\epsilon,c} \bigr]^{T} \right)_1  ,\\
		M'_\epsilon  &:=  \frac{1}{ \epsilon^2 } \bigl( \ZZ_\epsilon (\II-\ZZ_\epsilon)^{-1} \QQ_\epsilon \bigr)_1  ,
		\end{alignat*}
where the subscript denotes the first component of the vector.
Then $M_\epsilon$ and $M'_\epsilon$ are positive, increasing in $\epsilon$, and satisfy the inequalities
 \begin{alignat}{1}
 \left| \pi_\alpha A^{\dagger} \left( \tfrac{\partial F}{\partial  \epsilon}(\hat{x}_\epsilon) - \Gamma_\epsilon \right)  \right|  &\leq\epsilon^2 M_\epsilon , \label{eq:Mepsilon}\\
 \left( \ZZ_\epsilon (\II-\ZZ_\epsilon)^{-1} \QQ_\epsilon \right)_1 &\leq \epsilon^2 M'_\epsilon . \label{eq:MMepsilon}
 \end{alignat}
\end{lemma}

\begin{proof}
To first show that $(\II - \ZZ_{\epsilon })^{-1}$ is well defined,  we note that by Proposition~\ref{prop:TightEstimate} 
the radii polynomials $ P(\epsilon,\epsilon^2 \rr,\rho)$ are all negative. As was shown in the proof of Theorem~\ref{thm:RadPoly},
the operator norm of $ \ZZ_\epsilon$ on $ \R^3$ equipped with the norm $ \| \cdot \|_{\epsilon^2 \rr}$ is given by some $\kappa <1$, whereby the Neumann series of $ (\II - \ZZ_\epsilon)^{-1}$ converges. 
	
From the definition of $M_\epsilon$ and Corollary~\ref{cor:QUpperBound}, inequality \eqref{eq:Mepsilon} follows. Inequality \eqref{eq:MMepsilon} is a direct consequence of the definition of $M'_\epsilon$.
	Since the functions $\hat{f}_{\epsilon,1}$ and $\hat{f}_{\epsilon,c}$ are positive, then $M_\epsilon$ and $\QQ_\epsilon$ are positive. 
	Since the matrix $ \ZZ_\epsilon$ has positive entries only,  the Neumann series for $(\II-\ZZ_\epsilon)^{-1}$  has summands with exclusively positive entries, whereby $M'_\epsilon$ is positive. 

	Next we show that the components of $\ZZ_\epsilon$ and $\QQ_\epsilon - \QQ_\epsilon^0$ are polynomials in $\epsilon$ with positive coefficients and their lowest degree terms are at least quadratic. 
	To do so, it suffices to prove as much for the functions $ \hat{f}_{\epsilon,1},\hat{f}_{\epsilon,c},f_{1,\alpha},f_{1,\omega}, f_{1,c} , f_{*,\alpha}, f_{*,\omega},f_{*,c}$. 
	We note that all of these functions are given as polynomials with positive coefficients in the variables $ \epsilon, \da,\dw,\dc,r_c,\dc^0$ 
	(recall that $\rho $ is fixed and does not vary with $\epsilon$).
	Since $(r_\alpha , r_\omega,r_c) = \epsilon^2 (\rr_\alpha , \rr_\omega,\rr_c)$, then by Definition~\ref{def:DeltaDef} the terms $\da,\dw,r_c$ are all $ \cO(\epsilon^2)$. 
	Furthermore, whenever any of the terms  $ \epsilon, \dc, \dc^0$ 
appears, it is multiplied by another term of order at least $\cO(\epsilon)$.
	It  follows that every component of 
	 $ \ZZ_\epsilon$ and $ \QQ_{\epsilon} - \QQ_{\epsilon}^0$ is a polynomial in $ \epsilon$ with positive coefficients for which the lowest degree term is at least quadratic. 


From these considerations it follows that the components of both 
$M_\epsilon = \epsilon^{-2}(\QQ_{\epsilon} - \QQ_{\epsilon}^0)_1$ and $ \epsilon^{-2}\ZZ_\epsilon$ are polynomials in $\epsilon$ with positive coefficients. 
It also follows that both $\QQ_\epsilon$ and $(\II-\ZZ_\epsilon)^{-1}$ are increasing in $\epsilon$, whereby $M'_\epsilon$ is increasing in $\epsilon$.
\end{proof}

\end{document}